\documentclass[12pt]{article}
\usepackage[a4paper,vmargin={3cm,3cm},hmargin={2.5cm,2.5cm},bottom=30mm]{geometry}
\pdfoutput=1
\usepackage{dsfont}
\usepackage{mathrsfs}
\usepackage{multirow}
\usepackage{slashbox,multirow}
\usepackage{threeparttable}
\usepackage{lscape}
\usepackage[sectionbib]{natbib}
\usepackage{fancyhdr,latexsym,amsmath,amsfonts,amssymb,amsbsy,amsthm,url}
\usepackage{amsmath}
\usepackage[utf8]{inputenc}
\usepackage{tikz}
\date{}

 \usepackage{color}
\usepackage[pdfpagemode=UseNone,bookmarksopen=false,colorlinks=blue]{hyperref}
\hypersetup{ colorlinks=true, linkcolor=blue, citecolor=blue,
  filecolor=blue, urlcolor=red}

\allowdisplaybreaks
\bibfont{\footnotesize}

\newtheorem{assumptionC}{C}

\newtheorem{theorem}{Theorem}

\newtheorem{remark}{Remark}

\theoremstyle{remark}

\begin{document}



\title{\textbf{A comparison of estimators of mean and its functions in finite populations}}
\author{Anurag Dey and Probal Chaudhuri\\
\textit{Indian Statistical Institute, Kolkata}}
\maketitle


\begin{abstract}
Several well known estimators of finite population mean and its functions are investigated under some standard sampling designs. Such functions of  mean  include the variance, the correlation coefficient and the regression coefficient in the population as special cases. We compare the performance of these estimators under different sampling designs based on their asymptotic distributions.  Equivalence classes of estimators under different sampling designs are constructed  so  that estimators in the same class have equivalent performance in terms of asymptotic mean squared errors (MSEs). Estimators in different equivalence classes are then compared under some superpopulations satisfying linear models.  It is shown that the pseudo empirical likelihood (PEML) estimator of the population mean under simple random sampling without replacement (SRSWOR) has the lowest asymptotic MSE among all the estimators under different sampling designs considered in this paper. It is also shown that for the variance, the correlation coefficient and the regression coefficient of the population, the  plug-in estimators  based on the PEML estimator have the lowest asymptotic MSEs among all the estimators considered in this paper under SRSWOR.  On the other hand,  for any high entropy $\pi$PS (HE$\pi$PS) sampling design, which uses the auxiliary information, the plug-in  estimators of those parameters based on the H\'ajek estimator have the lowest asymptotic MSEs among all the estimators considered in this paper.\\ 
\end{abstract}
\newpage


\section{Introduction}\label{sec 1}
Suppose that $\mathcal{P}$=$\{1,2,\ldots,N\}$ is a finite population of size $N$, $s$ is a sample of size $n$ $(< N)$ from $\mathcal{P}$, and $\mathcal{S}$ is the collection of all possible samples having size $n$. Then, a sampling design $P(s)$ is a probability distribution on $\mathcal{S}$ such that $0 \leq P(s)\leq 1$ for all $s\in \mathcal{S}$ and $\sum_{s\in\mathcal{S}}P(s)$=$1$.  In this paper, simple random sampling without replacement (SRSWOR), Lahiri-Midzuno-Sen (LMS) sampling design (see \cite{Lahiri1951sampling}, \cite{Midzuno1952sampling} and \cite{Sen1953sampling}), Rao-Hartley-Cochran (RHC) sampling design (see \cite{rao1962simple}) and high entropy $\pi$PS (HE$\pi$PS) sampling designs (see Section \ref{sec 4}) are considered.  Note that all of the above sampling designs except SRSWOR are implemented utilizing some auxiliary variable.
\par

Let $(Y_i,X_i)$ be the value of $(y,x)$ for the $i^{th}$ population unit, $i$=$1,\ldots,N$, where $y$ is a univraite or multivariate study variable, and $x$ is a positive real valued size/auxiliary variable. Suppose that $\overline{Y}$=$\sum_{i=1}^N Y_i/N$ is the finite population mean of $y$. The Horvitz-Thompson (HT) estimator (see \cite{horvitz1952generalization})) and the RHC (see \cite{rao1962simple}) estimator are commonly used design unbiased estimators of $\overline{Y}$. Other well known estimators of $\overline{Y}$ are the H\'ajek estimator (see \cite{hajek1971}, \cite{sarndal2003model} and references therein),  the ratio estimator (see \cite{MR0474575}), the product estimator (see \cite{MR0474575}), the generalized regression (GREG) estimator (see \cite{MR1707846}) and the pseudo empirical likelihood (PEML) estimator (see \cite{MR1707846}). However, these latter estimators are not always design unbiased. For the expressions of the above estimators, the reader is referred to the Appendix. Now, suppose that $y$ is a $\mathbb{R}^d$-valued $(d\geq 1)$ study variable, and $g(\sum_{i=1}^N h(Y_i)/N)$ is a population parameter. Here, $h$: $ \mathbb{R}^d\rightarrow\mathbb{R}^p$ is a function with $p\geq 1$ and $g$: $ \mathbb{R}^p\rightarrow\mathbb{R}$ is a continuously differentiable function. All vectors in Euclidean spaces will be taken as row vectors and superscript $T$ will be used to denote their transpose. Examples of such a parameter are the variance, the correlation coefficient, the regression coefficient, etc. associated with a finite population. For simplicity, we will often write $h(Y_i)$ as $h_i$. Then, $g(\overline{h})$=$g(\sum_{i=1}^N h_i/N)$ is estimated by plugging in the estimator $\hat{\overline{h}}$ of $\overline{h}$. 
\par

In this article, our objective is to find asymptotically efficient (in terms of  mean squared error (MSE))  estimator of $g(\overline{h})$. In Section \ref{sec 4}, based on the asymptotic distribution of the estimator of $g(\overline{h})$ under above sampling designs, we construct equivalence classes of estimators such that any two estimators in the same class have the same asymptotic MSE. We consider the special case, when $g(\overline{h})$=$\overline{Y}$, and compare equivalence classes of estimators under superpopulations satisfying linear models  in Section \ref{subsec 4.1}.  Among different estimators under different sampling designs considered in this article, the PEML estimator of the population mean under SRSWOR turns out to be the estimator with the lowest asymptotic MSE.  Also, the PEML estimator has the same asymptotic MSE under SRSWOR and LMS sampling design.  Interestingly, we observe that the performance of the PEML estimator under RHC and any HE$\pi$PS sampling designs, which use auxiliary information is worse than its performance under SRSWOR. Earlier, it was shown that the GREG estimator is asymptotically at least as efficient as the HT, the ratio and the product estimators under SRSWOR  (see \cite{MR0474575}).  It will follow from our analysis that the PEML estimator is asymptotically equivalent to the GREG estimator under all the sampling designs considered in this paper.
\par

In  Section \ref{subsec 4.1},  we consider the cases, when $g(\overline{h})$ is the variance, the correlation coefficient and the regression coefficient in the population. Note that if the estimators of the population variance are constructed by plugging in the HT, the ratio, the product or the GREG estimators of the population means, then the estimators of the variance may become negative.  For this reason,  one also faces problem with the plug-in estimators of the correlation coefficient and the regression coefficient as these estimators require estimators of population variances. On the other hand, if the estimators of the above-mentioned parameters are constructed by plugging in the H\'ajek or the PEML estimators of the population means, such a problem does not occur. Therefore, for these parameters, we compare only those equivalence classes, which contain the  plug-in  estimators based on the H\'ajek and the PEML estimators. From this comparison  under superpopulations satisfying linear models,  we once again conclude that  for any of these parameters, the plug-in estimator based on the PEML estimator has asymptotically the lowest MSE among all the estimators considered in this article under SRSWOR  as well as LMS sampling design.  Moreover, under any HE$\pi$PS sampling design, which uses the auxiliary information, the plug-in estimator based on the H\'ajek estimator has asymptotically the lowest MSE among all the estimators considered in this article.  
\par

 \cite{scott1981asymptotic} proved that the ratio estimator has the same asymptotic distribution under SRSWOR and LMS sampling design. \cite{MR1707846} showed that the PEML estimator is asymptotically equivalent to the GREG estimator under some conditions on the sampling design, which are satisfied by SRSWOR and RHC sampling design. However, asymptotic equivalence classes as in Table \ref{table 1.2} in Section \ref{sec 4}, which consist of several estimators of a function of the population means under several sampling designs, were not constructed by any earlier author.
\par

 \cite{raj1954sampling} compared the sample mean under the simple random sampling with replacement with the usual unbiased estimator of the population mean under the probability proportional to size sampling with replacement, when the study variable and the size variable are exactly linearly related. \cite{avadhani1970comparison} compared the ratio estimator of the population mean under SRSWOR with the RHC estimator under RHC sampling design, when an approximate linear relationship holds between the study variable and the size variable. \cite{avadhani1972comparison} carried out the comparison of the ratio estimator of the population mean under LMS sampling design and the RHC estimator under RHC sampling design, when the study variable and the size variable are approximately linearly related. It was shown that the GREG estimator of the population mean is asymptotically at least as efficient as the HT, the ratio and the product estimators under SRSWOR (see \cite{MR0474575}). However, the above comparisons  included neither the PEML estimator nor HE$\pi$PS sampling designs. 
\par

 Some empirical studies carried out in Section \ref{subsec 3} using synthetic and real data demonstrate that the numerical and the theoretical results corroborate each other. We make some remarks on our major findings in Section \ref{sec 8}. Proofs of the results are given in the Appendix.
\section{Comparison of different estimators of $g(\overline{h})$}\label{sec 4}
In this section, we compare the estimators of $g(\overline{h})$, which are obtained by plugging in the estimators of $\overline{h}$ mentioned in Table \ref{table 1.1} below.
\begin{table}[h]
\caption{Estimators of $\overline{h}$}
\label{table 1.1}
\begin{center}
\begin{tabular}{|c|c|}
\hline 
\multirow{2}{*}{}Sampling& \multirow{2}{*}{Estimators} \\
designs & \\
\hline
\multirow{2}{*}{SRSWOR} & HT (which coincides with H\'ajek estimator), ratio,\\
&product, GREG and PEML estimators\\
\hline
\multirow{2}{*}{LMS} & HT, H\'ajek, ratio, product, GREG and\\
& PEML estimators\\
\hline
\multirow{2}{*}{HE$\pi$PS}& HT (which coincides with ratio and product\\
&estimators), H\'ajek, GREG and PEML estimators\\
\hline
RHC& RHC, GREG and PEML estimators\\
\hline
\end{tabular}
\end{center}
\end{table}
First, we find equivalence classes of estimators of $g(\overline{h})$ such that any two estimators in the same class are asymptotically normal with the same mean $g(\overline{h})$ and  the same  variance. 
\par

We define our asymptotic framework as follows. Let $\{\mathcal{P}_{\nu}\}$ be a sequence of nested populations with $N_\nu$, $n_\nu \rightarrow \infty$ as $\nu \rightarrow \infty$ ( see \cite{MR648029}, \cite{wang2011asymptotic}, \cite{conti2015inference}, \cite{MR3670194}, \cite{han2021complex} and references therein),  where $N_\nu$ and $n_{\nu}$ are, respectively, the population size and the sample size corresponding to the $\nu^{th}$ population. Henceforth, we shall suppress the subscript $\nu$ that tends to $\infty$ for the sake of simplicity. Throughout this paper, we consider the following condition (cf. Assumption $1$ in \cite{cardot2011horvitz}, A$4$ in \cite{conti2014estimation}, A$1$ in \cite{cardot2014variance} A$4$ in \cite{conti2015inference} and (HT$3$) in \cite{MR3670194}) 
\setcounter{assumptionC}{-1} 

\begin{assumptionC}\label{ass 4}
$n/N\rightarrow \lambda$ as $\nu\rightarrow\infty$, where $0\leq \lambda<1$.
\end{assumptionC}
Before we state the main results, let us discuss the HE$\pi$PS sampling design and some conditions on $\{(X_i,h_i):1\leq i\leq N\}$ (recall that $h_i$=$h(Y_i)$). A sampling design $P(s)$ satisfying the condition, $D(P||R)$= $\sum_{s \in \mathcal{S}}P(s)\log$ $(P(s)/R(s))\rightarrow0$ as $\nu\rightarrow\infty$, for some rejective sampling design (see \cite{MR0178555}) $R(s)$ is known as the high entropy sampling design ( see \cite{MR1624693}, \cite{conti2014estimation}, \cite{cardot2014variance}, \cite{MR3670194} and references therein).  A sampling design $P(s)$ is called the HE$\pi$PS sampling design if it is a high entropy sampling design, and its inclusion probabilities satisfy the condition $\pi_i$=$nX_i/\sum_{i=1}^N X_i$ for $i$=$1,\ldots,N$. An example of the HE$\pi$PS sampling design is the Rao-Sampford (RS) sampling design (see \cite{sampford1967sampling} and \cite{MR1624693}). We now state the following conditions. 
\begin{assumptionC}\label{ass 2}
$\{P_\nu\}$ is such that $\sum_{i=1}^N ||h_i||^4/N$=$O(1)$ and $\sum_{i=1}^N X_i^4/N$ =$O(1)$ as $\nu\rightarrow\infty$. Further, $\lim_{\nu\rightarrow\infty}\overline{h}$ exists, and $\overline{X}$=$\sum_{i=1}^N X_i/N$ and $S^2_x$= $\sum_{i=1}^N (X_i-\overline{X})^2/N$ are bounded away from $0$ as $\nu\rightarrow\infty$. Moreover, $\nabla g(\mu_0)\neq 0$, where $\mu_0$=$\lim_{\nu\rightarrow \infty}\overline{h}$ and $\nabla g$ is the gradient of $g$. 
\end{assumptionC}
\begin{assumptionC}\label{ass B2}
$\max_{1\leq i \leq N} X_{i}/\min_{1\leq i \leq N} X_{i}$=$O(1)$ as $\nu\rightarrow\infty$.
\end{assumptionC}
Let $\textbf{V}_i$ be one of $h_i$, $h_i-\overline{h}$, $h_i-\overline{h}X_i/\overline{X}$, $h_i+\overline{h}X_i/\overline{X}$ and $h_i-\overline{h}-S_{xh}(X_i-\overline{X})/S_x^2$ for $i$=$1,\ldots,N$, $\overline{h}$=$\sum_{i=1}^N h_i/N$ and $S_{xh}$=$\sum_{i=1}^N X_ih_i/N-\overline{h}$ $\overline{X}$. Define $\textbf{T}$=$\sum_{i=1}^N \textbf{V}_i(1-\pi_i)/\sum_{i=1}^N \pi_i(1-\pi_i)$, where $\pi_i$ is the inclusion probability of the $i^{th}$ population unit. Also, in the case of RHC sampling design, define $\overline{\textbf{V}}$=$\sum_{i=1}^N \textbf{V}_i/N$, $\overline{X}$=$\sum_{i=1}^N X_i/N$ and $\gamma$=$\sum_{i=1}^n N_i(N_i-1)/N(N-1)$, where $N_i$ is the size of the $i^{th}$ group formed randomly in RHC sampling design (see \cite{rao1962simple}), $i$=$1,\ldots,n$. Now, we state the following conditions on the population values and the sampling designs.
\begin{assumptionC}\label{ass 3}
$P(s)$ is such that $nN^{-2}\sum_{i=1}^{N}(\textbf{V}_i-\textbf{T}\pi_i)^T(\textbf{V}_i-\textbf{T}\pi_i)(\pi_i^{-1}-1)$ converges to some positive definite (p.d.) matrix as $\nu\rightarrow\infty$.
\end{assumptionC}
\begin{assumptionC}\label{ass 1}
$n\gamma\overline{X}N^{-1}\sum_{i=1}^{N}(\textbf{V}_i-X_i\overline{\textbf{V}}/\overline{X})^T(\textbf{V}_i-X_i\overline{\textbf{V}}/\overline{X})/ X_i$ converges to some p.d. matrix as $\nu\rightarrow\infty$.
\end{assumptionC}  
 Similar conditions like  C\ref{ass 2}, C\ref{ass 3} and C\ref{ass 1} are often used in sample survey literature (see  Assumption $3$ in \cite{cardot2011horvitz}, A$3$ and A$6$ in both \cite{conti2014estimation} and \cite{conti2015inference}, (HT$2$) in \cite{MR3670194}, and F$2$ and F$3$ in \cite{han2021complex}).  Conditions C\ref{ass 2} and C\ref{ass 1} hold (\textit{almost surely}), whenever $\{(X_i,$ $h_i):1\leq i \leq N\}$ are generated from a superpopulation model satisfying appropriate moment conditions (see Lemma S$2$ in the supplement). The condition $\sum_{i=1}^N ||h_i||^4/N$=$O(1)$ holds, when $h$ is a bounded function (e.g., $h(y)$=$y$ and $y$ is a binary study variable). Condition C\ref{ass B2} implies that the variation in the population values $X_1,\ldots,X_N$ cannot be too large.  Under any $\pi$PS sampling design, C\ref{ass B2} is equivalent to the condition that $L\leq  N \pi_{i}/n  \leq L^{\prime}$ for some constants $L,L^{\prime}>0$, any $i$=$1,\ldots,N$ and all sufficiently large $\nu\geq 1$. This latter condition was considered earlier in the literature (see (C$1$) in \cite{MR3670194} and Assumption $2$-\textsl{(i)} in \cite{wang2011asymptotic}). Condition C\ref{ass B2} holds (\textit{almost surely}), when $\{X_i\}_{i=1}^N$ are generated from a superpopulation distribution, and the support of the distribution of $X_i$ is bounded away from $0$ and $\infty$.  Condition C\ref{ass 3} holds (\textit{almost surely}) for SRSWOR, LMS sampling design and any $\pi$PS sampling design under  appropriate  superpopulation models (see Lemma S$2$ in the supplement). For the RHC sampling design, we also assume that $\{N_i\}_{i=1}^n$ are as follows.
 \begin{equation}\label{eq 150}
\begin{split}
N_i=
\begin{cases}
N/n,\text{ for } i=1,\cdots,n, \text{ when } N/n \text{ is an integer},\\
\lfloor N/n\rfloor,\text{ for } i=1,\cdots,k, \text{ and }\\
\lfloor N/n\rfloor+1,\text{ for } i=k+1,\cdots,n, \text{ when }N/n \text{ is not an integer},\\
\end{cases}
\end{split}
\end{equation}
where $k$ is such that $\sum_{i=1}^n N_i$=$N$. Here, $\lfloor N/n \rfloor$ is the integer part of $N/n$. \cite{rao1962simple} showed that this choice of $\{N_i\}_{i=1}^n$ minimizes the variance of the RHC estimator. Now, we state the following theorem.
\begin{theorem}\label{cor 4}
Suppose that C\ref{ass 4} through C\ref{ass 3} hold. Then, classes $1, 2,3$ and $4$ in Table \ref{table 1.2} describe equivalence classes of estimators for $g(\overline{h})$ under SRSWOR and LMS sampling design.
\end{theorem} 
For  next two theorems,  we assume that $n \max_{1\leq i\leq N} X_i/\sum_{i=1}^N X_i <1$. Note that this condition is required to hold for any without replacement $\pi$PS sampling design.
\begin{theorem}\label{thm 5}
$(i)$ If C\ref{ass 4} through C\ref{ass 3} hold, then classes $5, 6$ and $7$ in Table \ref{table 1.2} describe equivalence classes of estimators for $g(\overline{h})$ under any HE$\pi$PS sampling design.\\
$(ii)$ Under RHC sampling design, if  C\ref{ass 4} through C\ref{ass B2} and C\ref{ass 1} hold, then classes $8$ and $9$ in Table \ref{table 1.2} describe equivalence classes of estimators for $g(\overline{h})$. 
\end{theorem}
\begin{table}[h]
\centering
\caption{Disjoint equivalence classes of estimators for $g(\overline{h})$}
\vspace{.1cm}

\label{table 1.2}
\begin{threeparttable}[b]
\begin{tabular}{|c|c|c|c|c|c|c|}
\hline 
 &\multicolumn{6}{c|}{Estimators of $\overline{h}$}\\
\hline
 Sampling& GREG and & \multirow{2}{*}{HT}& \multirow{2}{*}{RHC} & \multirow{2}{*}{H\'ajek} & \multirow{2}{*}{Ratio} & \multirow{2}{*}{Product}\\
design&PEML&&&&&\\
\hline
SRSWOR & class $1$ &\tnote{1} class $2$ &  & \tnote{1} class $2$ & class $3$ & class $4$ \\
\hline
LMS & class $1$ & class $2$ &   & class $2$ & class $3$ & class $4$\\
\hline
HE$\pi$PS&  class $5$  &\tnote{2}  class $6$  &   &   class $7$  & \tnote{2}  class $6$  & \tnote{2}  class $6$\\
\hline
RHC & class $8$ &  & class $9$  & & & \\
\hline
\end{tabular}
\begin{tablenotes}
\item[1] The HT and the H\'ajek estimators coincide under SRSWOR.
\item[2] The HT, the ratio and the product estimators coincide under HE$\pi$PS sampling designs.
\end{tablenotes}
\end{threeparttable}
\end{table}
\begin{remark}\label{rem 2}
 It is to be noted that if C\ref{ass 2} through C\ref{ass 3} hold, and C\ref{ass 4} holds with $\lambda$=$0$, then in Table \ref{table 1.2}, class $8$ is merged with class $5$, and class $9$ is merged with class $6$. For details, see Section $S3$ in the supplement.  
\end{remark}
 Next, suppose that $W_i$=$\nabla g({\overline{h}})h_i^T$ for $i$=$1,\ldots,N$, $\overline{W}$=$\sum_{i=1}^N W_i/N$, $S_{xw}$= $\sum_{i=1}^N W_i X_i/N-\overline{W}$ $\overline{X}$, $S^2_w$= $\sum_{i=1}^N W_i^2/N$ $-\overline{W}^2$, $S^2_x$=$\sum_{i=1}^N X_i^2/ N-\overline{X}^2$ and $\phi$=$\overline{X}-(n/N)\sum_{i=1}^N X_i^2/N\overline{X}$. Now, we state the following theorem. 
\begin{theorem}\label{thm 6}
Suppose that the assumptions of Theorems \ref{cor 4} and \ref{thm 5} hold. Then, Table \ref{table 1.3} gives the expressions of asymptotic MSEs, $\Delta_1^2,\ldots,\Delta_9^2$, of the estimators in equivalence classes $1,\ldots,9$ in Table \ref{table 1.2}, respectively. 
\end{theorem}
\begin{remark}\label{rem 1}
 It can be shown in a straightforward way from Table \ref{table 1.3} that $\Delta^2_1\leq \Delta^2_i$ for $i$=$2,3$ and $4$. Thus, both the  plug-in  estimators of $g(\overline{h})$ that are based on the GREG and the PEML estimators are asymptotically as good as, if not better than, the  plug-in  estimators based on the HT (which coincides with the H\'ajek estimator), the ratio and the product estimators under SRSWOR, and the  plug-in  estimators based on the HT, the H\'ajek, the ratio and the product estimators under LMS sampling design. 
\end{remark}
\begin{table}[h]
\caption{Asymptotic MSEs of estimators for $g(\overline{h})$ (note that for simplifying notations,  the subscript  $\nu$ is dropped from the expressions on which limits are taken.)}
\label{table 1.3}
\begin{center}
\begin{tabular}{|c|c|}
\hline
$\Delta^2_1$=$(1-\lambda)\lim\limits_{\nu\rightarrow\infty}\big(S^2_{w}-(S_{xw}/S_x)^2\big)$ \\
\hline
$\Delta^2_2$=$(1-\lambda)\lim\limits_{\nu\rightarrow\infty}S^2_w$\\
\hline
$\Delta^2_3$=$(1-\lambda)\lim\limits_{\nu\rightarrow\infty}\big(S^2_{w}-2\overline{W}S_{xw}/\overline{X}+\left(\overline{W}/\overline{X}\right)^2S^2_x\big)$\\
\hline
$\Delta^2_4$=$(1-\lambda)\lim\limits_{\nu\rightarrow\infty}\big(S^2_{w}+2\overline{W}S_{xw}/\overline{X}+\left(\overline{W}/\overline{X}\right)^2S^2_x\big)$\\
\hline
 $\Delta^2_5$=$\lim\limits_{\nu\rightarrow\infty}(1/N)\sum_{i=1}^N \big(W_i-\overline{W}-(S_{xw}/S_x^2) (X_i-\overline{X}) \big)^2\times$\\
 $\big((\overline{X}/X_i)-(n/N)\big)$\\
\hline
 $\Delta^2_6$=$\lim\limits_{\nu\rightarrow\infty}(1/N)\sum_{i=1}^N\big\{ W_i+\phi^{-1} \overline{X}^{-1}X_i\big((n/N)\sum_{i=1}^N W_i X_i/N-\overline{W}$ $\overline{X}\big) \big\}^2\times$\\
 $\big\{(\overline{X}/X_i)-(n/N)\big\}$\\
\hline
 $\Delta^2_7$=$\lim\limits_{\nu\rightarrow\infty}(1/N)\sum_{i=1}^N \big(W_i-\overline{W}+(n/N\phi \overline{X})X_i S_{xw} \big)^2\times$\\
 $\big((\overline{X}/X_i)-(n/N)\big)$\\ 
\hline
$\Delta^2_8$=$\lim\limits_{\nu\rightarrow\infty} n\gamma   (\overline{X}/N) \sum_{i=1}^N\big(W_i-\overline{W}-(S_{xw}/S_x^2) (X_i-\overline{X})\big)^2/ X_i$\\
\hline
$\Delta^2_9$=$\lim\limits_{\nu\rightarrow\infty} n\gamma \big((\overline{X}/N)\sum_{i=1}^N W_i^2/X_i-\overline{W}^2\big)$\\
\hline
\end{tabular}
\end{center}
\end{table}
 Let us now consider some examples of $g(\overline{h})$ in Table \ref{table 1.7} below.  Conclusions of  Theorems \ref{cor 4} through \ref{thm 6}, and Remarks \ref{rem 2} and \ref{rem 1}  hold for all the parameters in Table \ref{table 1.7}. Here, we recall from the introduction that for the variance, the correlation coefficient and the regression coefficient, we consider only the  plug-in  estimators that are based on the H\'ajek and the PEML estimators.
\begin{table}[h]
\centering  
\caption{Examples of $g(\overline{h})$}
\label{table 1.7}
\begin{threeparttable}[b]
\begin{tabular}{|c|c|c|c|c|}
\hline 
Parameter& $d$& $p$& $h$& $g$\\
\hline
Mean& $1$ & $1$&  $h(y)$=$y$& $g(s)$=$s$\\
\hline
Variance& $1$&$2$& $h(y)$=$(y^2,y)$& $g(s_1,s_2)$=$s_1-s_2^2$\\
\hline
Correlation& \multirow{2}{*}{$2$}& \multirow{2}{*}{$5$}& $h(z_1,z_2)$=$(z_1, z_2,$& $g(s_1,s_2,s_3, s_4, s_5)$=$(s_5-s_1 s_2)/$\\
coefficient& & &$z^2_1,z^2_2, z_1 z_2)$&$((s_3-s^2_1)(s_4-s^2_2))^{1/2}$ \\
\hline
Regression&\multirow{2}{*}{$2$}& \multirow{2}{*}{$4$}& $h(z_1,z_2)$=$(z_1, z_2,$& $g(s_1,s_2,s_3, s_4, s_5)$=\\
coefficient& & &$z^2_2, z_1 z_2)$&$(s_4-s_1 s_2)/(s_3-s^2_2)$ \\
\hline
\end{tabular}
\end{threeparttable}
\end{table} 
\section{ Comparison of estimators under superpopulation models}\label{subsec 4.1} 
 In this section, we derive asymptotically efficient estimators for the mean, the variance, the correlation coefficient and the regression coefficient under superpopulations  satisfying linear regression  models. Earlier, \cite{raj1954sampling} \cite{MR0474578}, \cite{avadhani1970comparison}, \cite{avadhani1972comparison}  and  \cite{MR0474575} used the linear relationship between the $Y_i$'s and the $X_i$'s for comparing different estimators of the mean. However, they did not use any probability distribution for the $(Y_i,X_i)$'s. Subsequently, \cite{rao2003small}, \cite{fuller2011sampling}, \cite{chaudhuri2014} (see chap. $5$) and some other authors  considered the linear relationship between the $Y_i$'s and the $X_i$'s and a probability distribution for the $(Y_i,X_i)$'s for constructing different estimators and studying their behaviour.  However, the problem of finding asymptotically the most efficient estimator for the mean among a large class of estimators as considered in this paper was not done  earlier in the literature. Also, large sample comparisons of the plug-in estimators of the variance, the correlation coefficient and the regression coefficient considered in this paper were not carried out  in the earlier literature.  Suppose that $\{(Y_i,X_i):1\leq i\leq N\}$ are i.i.d. random vectors defined on a probability space $(\Omega,\mathbb{F},\mathbb{P})$. Without any loss of generality, for convenience, we take $\sigma_x^2$=$E_{\mathbb{P}}(X_i-E_{\mathbb{P}}(X_i))^2$ =$1$. This might require rescaling the variable $x$. Here, $E_{\mathbb{P}}$ denotes the expectation with respect to the probability measure $\mathbb{P}$. Recall that the population values $X_1,\ldots,X_N$ are used to implement some of the sampling designs. In such a case, we consider a function $P(s,\omega)$ on $\mathcal{S}\times\Omega$ so that $P(s,\cdot)$ is a random variable on $\Omega$ for each $s\in \mathcal{S}$, and $P(\cdot,\omega)$ is a probability distribution on $\mathcal{S}$ for each $\omega\in \Omega$ (see \cite{MR3670194}). Note that $P(s,\omega)$ is the sampling design for any fixed $\omega$ in this case. Then, the $\Delta_j^2$'s in Table \ref{table 1.3} can be expressed in terms of superpopulation moments of $(h(Y_i),X_i)$ by strong law of large numbers (SLLN). In that case,  we can easily compare different classes of estimators in Table \ref{table 1.2} under linear models.  Let us first state the following conditions on superpopulation distribution $\mathbb{P}$. 
\begin{assumptionC}\label{ass C1}
 $X_i\leq b$ \textit{a.s.} $[\mathbb{P}]$ for some $0<b<\infty$, $E_{\mathbb{P}}(X_i)^{-2}<\infty$, and $\max_{1\leq i\leq N} X_i/ $ $\min_{1\leq i\leq N} X_i$=$O(1)$ as $\nu\rightarrow\infty$ \textit{a.s.} $[\mathbb{P}]$.  Also, the support of the distribution of $(h(Y_i),X_i)$ is not a subset of a hyper-plane in $\mathbb{R}^{p+1}$.
\end{assumptionC}
 The condition, $X_i\leq b$ \textit{a.s.} $[\mathbb{P}]$ for some $0<b<\infty$, in C\ref{ass C1} and C\ref{ass 4} along with $0\leq \lambda<E_{\mathbb{P}}(X_i)/ b$ ensure that $n\max_{1\leq i\leq N} X_i/\sum_{i=1}^N X_i <1$ for all sufficiently large $\nu$ \textit{a.s.} $[\mathbb{P}]$, which is required to hold for any without replacement $\pi$PS sampling design. On the other hand, the condition, $\max_{1\leq i\leq N} X_i/\min_{1\leq i\leq N} X_i$ =$O(1)$ as $\nu\rightarrow\infty$ \textit{a.s.} $[\mathbb{P}]$, in C\ref{ass C1}  implies that C\ref{ass B2} holds \textit{a.s.} $[\mathbb{P}]$. Further, C\ref{ass C1} ensures that C\ref{ass 1} holds \textit{a.s.} $[\mathbb{P}]$ (see Lemma S$2$ in the supplement).  C\ref{ass C1}  also ensures that C\ref{ass 3} holds under LMS and any $\pi$PS sampling designs \textit{a.s.} $[\mathbb{P}]$ (see Lemma S$2$ in the supplement). 
\par

 Let us first consider the case, when $g(\overline{h})$  is the mean of $y$ (see the $2^{nd}$ row in Table \ref{table 1.7}).  Further, suppose that $Y_i$=$\alpha+\beta X_i+\epsilon_i$ for $\alpha,\beta \in \mathbb{R}$ and $i$=$1,\ldots,N$, where $\{\epsilon_i\}_{i=1}^N$ are i.i.d. random variables and are independent of $\{X_i\}_{i=1}^N$ with $E_{\mathbb{P}}(\epsilon_i)$=$0$ and $E_{\mathbb{P}}(\epsilon_i)^4<\infty$. Then, we have the following theorem.

\begin{theorem}\label{thm 2}
Suppose that C\ref{ass 4} holds with $0\leq\lambda<E_{\mathbb{P}}(X_i)/b$, and C\ref{ass C1} holds. Then, a.s. $[\mathbb{P}]$, the PEML estimator under SRSWOR as well as LMS sampling design has the lowest asymptotic MSE among all the estimators of the population mean under different sampling designs considered in this paper.
\end{theorem}
\begin{remark}
Note that for SRSWOR, the PEML estimator of the population mean has the lowest asymptotic MSE among all the estimators considered in this paper a.s. $[\mathbb{P}]$, when C\ref{ass 4} holds with $ 0 \leq \lambda<1$ and C\ref{ass C1} holds (see the proof of Theorem \ref{thm 2}).
\end{remark}
 
\begin{theorem}\label{Cor 3}
Suppose that C\ref{ass 4} holds with  $0\leq\lambda<E_{\mathbb{P}}(X_i)/b$, and C\ref{ass C1} holds. Then, a.s. $[\mathbb{P}]$, the performance of the PEML estimator of the population mean under RHC and any HE$\pi$PS sampling designs, which use auxiliary information, is worse than its performance under SRSWOR.
\end{theorem}
Recall from the introduction that for the variance, the correlation coefficient and the regression coefficient, we compare only those equivalence classes, which contain the  plug-in  estimators based on the H\'ajek and the PEML estimators. We first state the following condition. 
\begin{assumptionC}\label{ass C3}
$\xi>2\max\{\mu_1,\mu_{-1}/(\mu_1\mu_{-1}-1)\}$, where $\xi$=$\mu_3-\mu_2\mu_1$ is the covariance between $X_i^2$ and $X_i$ and $\mu_j$=$E_{\mathbb{P}}(X_i)^j$, $j$=$-1,1,2,3$.
\end{assumptionC}
 The above condition is used to prove part $(ii)$ in each of Theorems \ref{thm 4} and \ref{thm 3}. This condition holds when the $X_i$'s follow well known distributions like Gamma (with shape parameter value larger than $1$ and any scale parameter value), Beta (with the second shape parameter value greater than the first shape parameter value and the first shape parameter value larger than $1$), Pareto (with shape parameter value lying in the interval $(3,(5+\sqrt{17})/2)$ and any scale parameter value), Log-normal (with both the parameters taking any value) and Weibull (with shape parameter value lying in the interval $(1, 3.6)$ and any scale parameter value).  Now, consider the case, when $g(\overline{h})$ is the variance of $y$  (see the $3^{rd}$ row in Table \ref{table 1.7}).  Recall the linear model $Y_i$=$\alpha+\beta X_i+\epsilon_i$ from above and assume that $E_{\mathbb{P}}(\epsilon_i)^8<\infty$. Then, we have the following theorem.

\begin{theorem}\label{thm 4}
$(i)$ Let us first consider SRSWOR and LMS sampling design and suppose that C\ref{ass 4} and C\ref{ass C1} hold. Then, a.s. $[\mathbb{P}]$, the plug-in estimator of the population variance based on the PEML estimator has the lowest asymptotic MSE among all the estimators considered in this paper.\\ 
$(ii)$ Next, consider any HE$\pi$PS sampling design and suppose that C\ref{ass 4} holds with $0\leq\lambda<E_{\mathbb{P}}(X_i)/b$, and C\ref{ass C1} and C\ref{ass C3} hold. Then, a.s. $[\mathbb{P}]$, the plug-in estimator of the population variance based on the H\'ajek estimator has the lowest asymptotic MSE among all the estimators considered in this paper.
\end{theorem}
 Now, suppose that $y$=$(z_1,z_2)\in\mathbb{R}^2$ and consider the case, when $g(\overline{h})$ is the correlation coefficient between $z_1$ and $z_2$  (see the $4^{th}$ row in Table \ref{table 1.7}).  Let us also consider the case, when $g(\overline{h})$ is the regression coefficient of $z_1$ on $z_2$  (see the $5^{th}$ row in Table \ref{table 1.7}).  Further, suppose that $Y_i$=$\alpha+\beta X_i+\epsilon_i$ for $Y_i$=$(Z_{1i},Z_{2i})$, $\alpha,\beta\in \mathbb{R}^2$ and $i$=$1,\ldots,N$, where $\{\epsilon_i\}_{i=1}^N$ are i.i.d. random vectors in $\mathbb{R}^2$ independent of $\{X_i\}_{i=1}^N$ with $E_{\mathbb{P}}(\epsilon_{i})$=$0$ and $E_{\mathbb{P}}||\epsilon_{i}||^8<\infty$. Then, we have the following theorem.

\begin{theorem}\label{thm 3}
$(i)$ Let us first consider SRSWOR and LMS sampling design and suppose that C\ref{ass 4} and C\ref{ass C1} hold. Then, a.s. $[\mathbb{P}]$, the plug-in estimator of each of the correlation and the regression coefficients in the population based on the PEML estimator has the lowest asymptotic MSE among all the estimators considered in this paper. \\
$(ii)$ Next, consider any HE$\pi$PS sampling design and suppose that C\ref{ass 4} holds with $0\leq\lambda<E_{\mathbb{P}}(X_i)/b$, and C\ref{ass C1} and C\ref{ass C3} hold. Then, a.s. $[\mathbb{P}]$, the plug-in estimator of each of the above parameters based on the H\'ajek estimator has the lowest asymptotic MSE among all the estimators considered in this paper. 
\end{theorem}

\section{Data analysis}\label{subsec 3}
In this section, we carry out an empirical comparison of the estimators of the mean, the variance, the correlation coefficient and the regression coefficient, which are discussed in this paper, based on both real and synthetic data. Recall that for the above parameters, we have considered several estimators and sampling designs, and conducted a theoretical comparison of those estimators in Sections \ref{sec 4} and \ref{subsec 4.1}. For empirical comparison, we exclude some of the estimators considered in theoretical comparison so that the results of the comparison become concise and comprehensive. The reasons for excluding those estimators are given below. 
\begin{itemize}
\item[(i)] Since the GREG estimator is well known to be asymptotically better than the HT, the ratio and the product estimators under SRSWOR (see \cite{MR0474575}), we exclude these latter estimators under SRSWOR.
\item[(ii)] Since the MSEs of the estimators under LMS sampling design become very close to the MSEs of the same estimators under SRSWOR as expected from  Theorem \ref{cor 4},  we do not report these results under LMS sampling design.  Moreover, SRSWOR is a simpler and more commonly used sampling design than LMS sampling design. 
\end{itemize}
Thus we consider the estimators mentioned in Table \ref{table 1.4} below for the empirical comparison. 
\begin{table}[h]
\caption{Estimators considered for the empirical comparison}
\label{table 1.4}
\begin{center}
\begin{threeparttable}[b]
\begin{tabular}{|c|c|}
\hline 
Parameters& Estimators \\
\hline
\multirow{5}{*}{Mean} & GREG and PEML estimators under SRS-\\
& WOR; HT, H\'ajek, GREG and PEML\\
& estimators under \tnote{$3$} RS sampling design;\\
& and RHC, GREG and PEML estimators \\
& under RHC sampling design\\
\hline
Variance, correlation & Obtained by plugging in H\'ajek and PEML \\
coefficient and regression&  estimators under SRSWOR and\tnote{1} RS \\
coefficient& sampling design, and PEML estimator  \\
&under RHC sampling design\\
\hline
\end{tabular}
\begin{tablenotes}
\item[3] We consider RS sampling design since it is a HE$\pi$PS sampling design, and it is easier to implement than other HE$\pi$PS sampling designs.
\end{tablenotes}
\end{threeparttable}
\end{center}
\end{table}
Recall from Table \ref{table 1.1} that the HT, the ratio and the product estimators of the mean coincide under any HE$\pi$PS sampling design. We draw $I$=$1000$ samples each of sizes $n$=$75$, $100$ and $125$ using sampling designs mentioned in Table \ref{table 1.4}. We use the $R$ software for drawing samples as well as computing different estimators. For RS sampling design, we use the `pps' package in $R$, and for the PEML estimator, we use $R$ codes in \cite{wu2005algorithms}. Two estimators $g(\hat{\overline{h}}_1)$ and $g(\hat{\overline{h}}_2)$ of $g(\overline{h})$ under sampling designs $P_1(s)$ and $P_2(s)$, respectively, are compared empirically by means of the relative efficiency defined as
$$
RE(g(\hat{\overline{h}}_1),P_1|g(\hat{\overline{h}}_2),P_2)=MSE_{P_2}(g(\hat{\overline{h}}_2))/MSE_{P_1}(g(\hat{\overline{h}}_1)),
$$ 
where $MSE_{P_j}(g(\hat{\overline{h}}_j))$=$I^{-1}\sum_{l=1}^I(g(\hat{\overline{h}}_{jl})-g(\overline{h}_0))^2$ is the empirical mean squared error of $g(\hat{\overline{h}}_j)$ under $P_j(s)$, $j$=$1,2$. Here, $\hat{\overline{h}}_{jl}$ is the estimate of $\overline{h}$ based on the $j^{th}$ estimator and the $l^{th}$ sample, and $g(\overline{h}_0)$ is the true value of the parameter $g(\overline{h})$, $j$=$1,2$, $l$=$1,\ldots,I$. $g(\hat{\overline{h}}_1)$ under $P_1(s)$ will be more efficient than $g(\hat{\overline{h}}_2)$ under $P_2(s)$ if $RE(g(\hat{\overline{h}}_1),P_1|g(\hat{\overline{h}}_2),P_2)>1$.
\par

 Next, for each of the parameters considered in this section, we compare average lengths of asymptotically $95\%$ confidence intervals (CIs) constructed based on several estimators used in this section. In order to construct asymptotically $95\%$ CIs, we need an estimator of the asymptotic MSE of $\sqrt{n}(g(\hat{\overline{h}})-g(\overline{h}))$ and we shall discuss it in detail now. If we consider SRSWOR or RS sampling design, it follows from the proofs of Theorems \ref{cor 4} and \ref{thm 5} that the asymptotic MSE of $\sqrt{n}(g(\hat{\overline{h}})-g(\overline{h}))$ is $\tilde{\Delta}^2_1$=$\lim_{\nu\rightarrow\infty}nN^{-2}\nabla g(\overline{h}) \sum_{i=1}^{N}(\textbf{V}_i-\textbf{T}\pi_i)^T(\textbf{V}_i-\textbf{T}\pi_i)(\pi_i^{-1}-1) \nabla g(\overline{h})^T$, where $\textbf{T}$=$\sum_{i=1}^N \textbf{V}_i(1-\pi_i)/$ $\sum_{i=1}^N \pi_i(1-\pi_i)$. Moreover, $\textbf{V}_i$ is $h_i$ or $h_i-\overline{h}$ or $h_i-\overline{h}-S_{xh}(X_i-\overline{X})/S_x^2$ if $\hat{\overline{h}}$ is $\hat{\overline{h}}_{HT}$ or $\hat{\overline{h}}_{H}$ or $\hat{\overline{h}}_{PEML}$ (as well as $\hat{\overline{h}}_{GREG}$) with $d(i,s)$=$(N\pi_i)^{-1}$, respectively. Recall from the paragraph following C\ref{ass B2} that $S_{xh}$=$\sum_{i=1}^N X_ih_i/N-\overline{X}$ $\overline{h}$. Following the idea of \cite{cardot2014variance}, we estimate $\tilde{\Delta}^2_1$ by 
\begin{align}\label{eq 5}
\begin{split}
&\hat{\Delta}^2_1=nN^{-2}\nabla g(\hat{\overline{h}})\sum_{i\in s}(\hat{\textbf{V}}_i-\hat{\textbf{T}}\pi_i)^T (\hat{\textbf{V}}_i-\hat{\textbf{T}}\pi_i)(\pi_i^{-1}-1)\pi_i^{-1}\nabla g(\hat{\overline{h}})^T,
\end{split}
\end{align}
where $\hat{\textbf{T}}$=$\sum_{i\in s} \hat{\textbf{V}}_i(\pi_i^{-1}-1)/\sum_{i\in s} (1-\pi_i)$, $\hat{\overline{h}}$=$\hat{\overline{h}}_{HT}$ in the case of the mean, the variance and the regression coefficient, and $\hat{\overline{h}}$=$\hat{\overline{h}}_{H}$ in the case of the correlation coefficient. Here, $\hat{\textbf{V}}_i$ is $h_i$ or $h_i-\hat{\overline{h}}_{HT}$ or $h_i-\hat{\overline{h}}_{HT}-\hat{S}_{xh,1}(X_i-\hat{\overline{X}}_{HT})/\hat{S}_{x,1}^2$ if $\hat{\overline{h}}$ is $\hat{\overline{h}}_{HT}$ or $\hat{\overline{h}}_{H}$ or $\hat{\overline{h}}_{PEML}$ (as well as $\hat{\overline{h}}_{GREG}$) with $d(i,s)$=$(N\pi_i)^{-1}$. Further, $\hat{S}_{xh,1}$=$\sum_{i\in s}(N\pi_i)^{-1} X_i h_i-\hat{\overline{X}}_{HT} \hat{\overline{h}}_{HT}$ and $\hat{S}_{x,1}^2$=$\sum_{i\in s}(N\pi_i)^{-1} X_i^2-\hat{\overline{X}}_{HT}^2$. We estimate $\overline{h}$ in $\nabla g(\overline{h})$ by $\hat{\overline{h}}_{HT}$ in the case of the mean, the variance and the regression coefficient because $\hat{\overline{h}}_{HT}$ is an unbiased estimator, and it is easier to compute than the other estimators of $\overline{h}$ considered in this paper. On the other hand, different estimators of the correlation coefficient that are considered in this paper may become undefined if we estimate $\overline{h}$ by any estimator other than $\hat{\overline{h}}_{H}$ and $\hat{\overline{h}}_{PEML}$ (see the $4^{th}$ paragraph in the introduction). In this case, we choose $\hat{\overline{h}}_{H}$ because it is easier to compute than $\hat{\overline{h}}_{PEML}$.
\par

Next, if we consider RHC sampling design, it follows from the proof of Theorem \ref{thm 5} that the asymptotic MSE of $\sqrt{n}(g(\overline{h})-g(\hat{\overline{h}}))$ is $\tilde{\Delta}^2_2$=$\lim_{\nu\rightarrow\infty}n\gamma \overline{X}N^{-1}\times$ $\nabla g(\overline{h})\sum_{i=1}^{N}(\textbf{V}_i-X_i\overline{\textbf{V}}/\overline{X})^T (\textbf{V}_i-X_i\overline{\textbf{V}}/\overline{X})X_i^{-1}\nabla g(\overline{h})^T$, where $\gamma$ and $\overline{\textbf{V}}$ are as in the paragraph following C\ref{ass B2}. Moreover, $\textbf{V}_i$ is $h_i$ or $h_i-\overline{h}-S_{xh}(X_i-\overline{X})/S_x^2$ if $\hat{\overline{h}}$ is $\hat{\overline{h}}_{RHC}$ or $\hat{\overline{h}}_{PEML}$ (as well as $\hat{\overline{h}}_{GREG}$) with $d(i,s)$=$G_i/NX_i$, respectively. We estimate $\tilde{\Delta}^2_2$ by 
\begin{align}\label{eq 6}
\begin{split}
&\hat{\Delta}^2_2=n\gamma \overline{X}N^{-1} \nabla g(\hat{\overline{h}}) \sum_{i\in s}(\hat{\textbf{V}}_i-X_i\hat{\overline{\textbf{V}}}_{RHC}/\overline{X})\times\\
&(\hat{\textbf{V}}_i-X_i\hat{\overline{\textbf{V}}}_{RHC}/ \overline{X})(G_i X_i^{-2})\nabla g(\hat{\overline{h}} )^T,
\end{split}
\end{align}
where $\hat{\overline{\textbf{V}}}_{RHC}$=$\sum_{i\in s}\hat{\textbf{V}}_i G_i/N X_i$, $\hat{\overline{h}}$=$\hat{\overline{h}}_{RHC}$ in the case of the mean, the variance and the regression coefficient, and $\hat{\overline{h}}$=$\hat{\overline{h}}_{PEML}$ in the case of the correlation coefficient. Here, $\hat{\textbf{V}}_i$ is $h_i$ or $h_i-\hat{\overline{h}}_{RHC}-\hat{S}_{xh,2}(X_i- \overline{X})/\hat{S}_{x,2}^2$ if $\hat{\overline{h}}$ is $\hat{\overline{h}}_{RHC}$ or $\hat{\overline{h}}_{PEML}$ (as well as $\hat{\overline{h}}_{GREG}$) with $d(i,s)$=$G_i/N X_i$. Further, $\hat{S}_{xh,2}$=$\sum_{i\in s}  h_i G_i/N- \overline{X}$ $ \hat{\overline{h}}_{RHC}$ and $\hat{S}_{x,1}^2$=$\sum_{i\in s} X_i G_i/N-\overline{X}^2$. In the case of the mean, the variance and the regression coefficient, we estimate $\overline{h}$ in $\nabla g(\overline{h})$ by $\hat{\overline{h}}_{RHC}$ for the same reason as discussed in the preceding paragraph, where we discuss the estimation of $\overline{h}$ by $\hat{\overline{h}}_{HT}$ under SRSWOR and RS sampling design. On the other hand, in the case of the correlation coefficient, we estimate $\overline{h}$ in $\nabla g(\overline{h})$ by $\hat{\overline{h}}_{PEML}$ under RHC sampling design so that the estimator of the correlation coefficient appeared in the expression of $\nabla g(\overline{h})$ in this case becomes well defined.
\par

We draw $I$=$1000$ samples each of sizes $n$=$75$, $100$ and $125$ using sampling designs mentioned in Table \ref{table 1.4}. Then, for each of the parameters, the sampling designs and the estimators mentioned in Table \ref{table 1.4}, we construct $I$ many asymptotically $95\%$ CIs based on these samples and compute the average and the standard deviation of their lengths. 
\subsection{Analysis based on synthetic data}\label{subsec 3.1}
In this section, we consider the population values $\{(Y_i,X_i):1\leq i\leq N\}$ on $(y,x)$ generated from a linear model as follows. We choose $N$=$5000$ and generate the $X_i$'s from a gamma distribution with mean $1000$ and standard deviation (s.d.) $200$. Then, $Y_i$ is generated from the linear model $Y_i$=$500+X_i+\epsilon_i$ for $i$=$1,\ldots,N$, where $\epsilon_i$ is generated independently of $\{X_i\}_{i=1}^N$ from a normal distribution with mean $0$ and s.d. $100$. We also generate the population values $\{(Y_i,X_i):1\leq i\leq N\}$ from a linear model, when $y$=$(z_1,z_2)$ is a bivariate study variable. The population values $\{X_i\}_{i=1}^N$ are generated in the same way as in the earlier case. Then, $Y_i$=$(Z_{1i},Z_{2i})$ is generated from the linear model $Z_{ji}$=$\alpha_j+X_i+\epsilon_{ji}$ for $i$=$1,\ldots,N$, where $\alpha_1$=$500$ and $\alpha_2$=$1000$. The $\epsilon_{1i}$'s are generated independently of the $X_i$'s from a normal distribution with mean $0$ and s.d. $100$, and the $\epsilon_{2i}$'s are generated independently of the $X_i$'s and the $\epsilon_{1i}$'s from a normal distribution with mean $0$ and s.d. $200$. We consider the estimation of  the mean and the variance of $y$ for the first dataset and the correlation and the regression coefficients between $z_1$ and $z_2$ for the second dataset. 
\par
The results of the empirical comparison based on synthetic data are summarized as follows. For each of the mean, the variance, the correlation coefficient and the regression coefficient, the  plug-in  estimator based on the PEML estimator under SRSWOR turns out to be more efficient than any other estimator under any other sampling design (see Tables $2$ through $6$ in the supplement) considered in Table \ref{table 1.4}  when compared in terms of relative efficiencies. Also, for each of the above parameters, asymptotically $95\%$ CI based on the PEML estimator under SRSWOR has the least average length (see Tables $7$ through $11$ in the supplement).  Thus the empirical results stated here corroborate the theoretical results stated in Theorems \ref{thm 2} through \ref{thm 3}.
\subsection{Analysis based on real data}\label{subsec 3.2}
In this section, we consider a dataset on the village amenities in the state of West Bengal in India obtained from the Office of the Registrar General \& Census Commissioner, India (\href{https://censusindia.gov.in/nada/index.php/catalog/1362}{https://censusindia.gov.in/nada/index.php/catalog/1362}). Relevant study variables for this dataset are described in Table \ref{table 1.5} below. We consider the following estimation problems for a population consisting of $37478$ villages. For these estimation problems, we use the number of people living in village $x$ as the size variable. 
\begin{table}[h]
\caption{Description of study variables}
\label{table 1.5}
\begin{center}
\begin{tabular}{|c|c|} 
\hline
$y_1$& Number of primary schools in village\\
\hline
$y_2$& Scheduled castes population size in village\\
\hline
$y_3$& Number of secondary schools in village \\
\hline
$y_4$& Scheduled tribes population size in village\\
\hline
\end{tabular}
\end{center}
\end{table}
\begin{itemize}
\item[(i)] First, we consider the estimation of the mean and the variance of each of $y_1$ and $y_2$. It can be shown from the scatter plot  and the least square regression line  in Figure $1$ in the supplement that $y_1$ and $x$ have an approximate linear relationship. Also, the correlation coefficient between $y_1$ and $x$ is $0.72$. On the other hand, $y_2$ and $x$ do not seem to have a linear relationship (see  the scatter plot and the least square regression line in  Figure $2$ in the supplement).
\item[(ii)] Next, we consider the estimation of the correlation and the regression coefficients of $y_1$ and $y_3$ as well as of $y_2$ and $y_4$. The scatter plot  and the least square regression line  in Figure $3$ in the supplement  show  that $y_3$ does not seem to be dependent on $x$. Further, we see from the scatter plot  and the least square regression line  of $y_4$ and $x$ (see Figure $4$ in the supplement) that $y_4$ and $x$ do not seem to have a linear relationship.
\end{itemize} 
The results of the empirical comparison based on real data are summarized in Table \ref{table 1.6} below. 
\begin{table}[h]
\caption{Most efficient estimators  in terms of relative efficiencies (it follows from Tables $22$ through $31$ in the supplement that asymptotically $95\%$ CIs based on most efficient estimators have least average lengths.)}
\label{table 1.6}
\begin{center}
\begin{tabular}{|c|c|} 
\hline
Parameters& Most efficient estimators\\
\hline
\multirow{2}{*}{Mean and variance of $y_1$}&  The plug-in estimator based on the  the\\
& PEML estimator under SRSWOR\\
\hline
Mean of $y_2$& The HT estimator under RS sampling design\\
\hline
\multirow{2}{*}{Variance of $y_2$}&  The plug-in estimator based on the  H\'ajek \\
& estimator under RS sampling design\\
\hline
Correlation and regression &  The plug-in estimator based on the \color{black}PEML \\
coefficients of $y_1$ and $y_3$& estimator under SRSWOR\\
\hline
Correlation and regression &  The plug-in estimator based on the  H\'ajek\\
coefficients of $y_2$ and $y_4$& estimator under RS sampling design\\
\hline
\end{tabular}
\end{center}
\end{table}
For further details see Tables $12$ through $31$ in the supplement.  The approximate linear relationship between $y_1$ and $x$ (see the scatter plot and the least square regression line in Figure $1$ in the supplement) could be a possible reason why the plug-in estimator based on the PEML estimator under SRSWOR becomes the most efficient for each of the mean and the variance of $y_1$ among all the estimators under different sampling designs considered in this section. Also, possibly for the same reason, the plug-in estimators of the correlation and the regression coefficients between $y_1$ and $y_3$ based on the PEML estimator under SRSWOR become the most efficient among all the estimators under different sampling designs considered in this section. 
\par

 On the other hand, any of $y_2$, and $y_4$ does not seem to have a linear relationship with $x$ (see the scatter plots and the least square regression lines in Figures $2$ and $4$ in the supplement). Possibly, because of this reason, the plug-in estimators of the parameters related to $y_2$ and $y_4$ based on the PEML estimator are not able to outperform the the plug-in estimators of those parameters based on the HT and the H\'ajek estimators. Next, we observe that there are substantial correlation present between $y_2$ and $x$ (correlation coefficient=$0.67$), and $y_4$ and $x$ (correlation coefficient=$0.25$). Possibly, because of this, under RS sampling design, which uses the auxiliary information, the plug-in estimators of the parameters related to $y_2$ and $y_4$ based on the HT and the H\'ajek estimators become the most efficient among all the estimators under different sampling designs considered in this section.  
\section{Concluding discussion and remarks}\label{sec 8}
It follows from Theorem \ref{thm 2} that the PEML estimator of the mean under SRSWOR becomes asymptotically either more efficient than or equivalent to any other estimator under any other sampling design considered in this paper. It also follows from  Theorems \ref{cor 4} and \ref{thm 5}  that the GREG estimator of the mean is asymptotically equivalent to the PEML estimator under different sampling designs considered in this paper.  However, our numerical studies (see Section \ref{subsec 3}) based on finite samples indicate that the PEML estimator of the mean performs slightly better than the GREG estimator under all the sampling designs considered in Section \ref{subsec 3} (see Tables $2$, $12$ and $14$ in the supplement).  Moreover, as pointed out in the introduction, if the estimators of the variance, the correlation coefficient and the regression coefficient are constructed by plugging in the GREG estimator of the mean, then the estimators of the population variances involved in these parameters may become negative. On the other hand, if the estimators of these parameters are constructed by plugging in the PEML estimator of the mean, then such a problem does not occur. Further, for these parameters, depending on sampling designs, the  plug-in  estimator based on either the PEML or the H\'ajek estimator turns out to be asymptotically best among different estimators that we have considered (see Theorems \ref{thm 4} and \ref{thm 3}).
\par

 We see from Theorem \ref{thm 2} that for the population mean, the PEML estimator, which is not design unbiased, performs better than design unbiased estimators like the HT and the RHC estimators. Further, as pointed out in the introduction, the plug-in estimators of the population variance based on the HT and the RHC estimators may become negative. This affects the plug-in estimators of the correlation and the regression coefficients based on the HT and the RHC estimators. 
\par

 It follows from Table \ref{table 1.2} that under LMS sampling design, the large sample performances of all the estimators of functions of means considered in this paper are the same as their large sample performances under SRSWOR. The LMS sampling design was introduced to make the ratio estimator of the mean unbiased. It follows from Remark \ref{rem 1} in Section \ref{sec 4} that the performance of the ratio estimator of the mean is worse than several other estimators that we have considered even under LMS sampling design. 
\par

The coefficient of variation is another well known finite population parameter, which can be expressed as a function of population mean $g(\overline{h})$. We have $d$=$1$, $p$=$2$, $h(y)$=$(y^2,y)$ and $g(s_1,s_2)$=$\sqrt{s_1-s_2^2}/ s_2$ in this case. Among the estimators considered in this paper, the  plug-in  estimators of $g(\overline{h})$ that are based on the PEML and the H\'ajek estimators of the mean can be used for estimating this parameter since it involves the finite population variance (see the $4^{th}$ paragraph in the introduction).  We have avoided reporting the comparison of the estimators of the coefficient of variation in this paper because of complex mathematical expressions.  However, the asymptotic results stated in Theorems \ref{thm 4} and \ref{thm 3} also hold for this parameter. 
\par

An empirical comparison of the biased estimators considered in this paper and their bias-corrected versions are carried out based on jackknifing in Section S$4$ in the supplement. It follows from this comparison that for all the parameters considered in this paper, the bias-corrected estimators become worse than the original biased estimators in the cases of both the synthetic and the real data. This is because, although bias-correction results in reduction of biases in the original biased estimators, the variances of these estimators increase substantially after bias-correction. 

\section*{Supplementary material}
In the supplement, we discuss some conditions from the main paper and demonstrate situations, where these conditions hold. Then, we state and prove some additional mathematical results. We also give the proofs of Remark \ref{rem 2} and Theorems \ref{thm 5}, \ref{thm 6}, \ref{thm 4} and \ref{thm 3}. The biased estimators considered in this paper are then compared empirically with their bias-corrected versions based on jackknifing in terms of MSE. Finally, we provide the numerical results related to the analysis based on both synthetic and real data (see Section \ref{subsec 3}).     

\section*{Acknowledgments}
The authors gratefully acknowledge careful reading of an earlier version of the paper by an anonymous reviewer and an associate editor. Critical comments and constructive suggestions from the reviewer and the associate editor led to significant improvement of the paper.  The authors would also like to thank Prof. Aloke Kar and Prof. Sandip Mitra  for several discussions about Section \ref{subsec 3.2} of the paper.

\section*{Appendix}
Let us begin by providing the expressions (see Table \ref{table A} below) 
\begin{table}[h]
\caption{Estimators of $\overline{Y}$}
\label{table A}
\begin{center}
\begin{tabular}{|c|c|} 
\hline
Estimator& Expression\\
\hline
HT& $\hat{\overline{Y}}_{HT}$=$\sum_{i\in s} (N\pi_i)^{-1}Y_i$\\
\hline
RHC& $\hat{\overline{Y}}_{RHC}$=$\sum_{i\in s}G_i Y_i/NX_i$\\
\hline
H\'ajek & $\hat{\overline{Y}}_{H}$=$\sum_{i\in s}\pi^{-1}_iY_i/\sum_{i\in s}\pi_i^{-1}$\\
\hline
Ratio& $\hat{\overline{Y}}_{RA}$=$(\sum_{i\in s}\pi_i^{-1} Y_i/\sum_{i\in s}\pi_i^{-1} X_i)\overline{X}$ \\
\hline
Product& $\hat{\overline{Y}}_{PR}$=$\sum_{i\in s}(N\pi_i)^{-1}Y_i\sum_{i\in s}(N\pi_i)^{-1} X_i/$ $\overline{X}$\\
\hline
GREG& $\hat{\overline{Y}}_{GREG}$=$\hat{\overline{Y}}_{*}+\hat{\beta}(\overline{X}-\hat{\overline{X}}_{*})$\\
\hline
PEML& $\hat{\overline{Y}}_{PEML}$=$\sum_{i\in s}c_iY_i$\\
\hline
\end{tabular}
\end{center}
\end{table}
of those estimators of $\overline{Y}$, which are considered in this paper. In Table \ref{table A}, $\{\pi_i\}_{i=1}^N$ denote inclusion probabilities, and $G_i$ is the total of the $x$ values of that randomly formed group from which the $i^{th}$ population unit is selected in the sample by RHC sampling design (cf. \cite{chaudhuri2006feasibility}). In the case of the GREG estimator, $\hat{\overline{Y}}_{*}$=$\sum_{i\in s}d(i,s)Y_i/\sum_{i\in s}d(i,s)$, $\hat{\overline{X}}_{*}$=$\sum_{i\in s}d(i,s)$ $\times X_i/\sum_{i\in s}d(i,s)$ and $\hat{\beta}$=$\sum_{i\in s}d(i,s)(Y_i-\hat{\overline{Y}}_{*})(X_i-\hat{\overline{X}}_{*})/\sum_{i\in s}d(i,s) (X_i-\hat{\overline{X}}_{*})^2$, where $\{d(i,s):i\in s\}$ are sampling design weights. Finally, the $c_i$'s ($>0$) in the PEML estimator are obtained by maximizing $\sum_{i\in s}d(i,s)\log(c_i)$ subject to $\sum_{i\in s}c_i$=$1$ and $\sum_{i\in s}c_i(X_i-\overline{X})$=$0$. Following \cite{MR1707846}, we consider both the GREG and the PEML estimators with $d(i,s)$=$(N\pi_i)^{-1}$ under SRSWOR, LMS sampling design and any HE$\pi$PS sampling design, and with $d(i,s)$=$G_i/NX_i$ under RHC sampling design.
\par
 
Let us denote the HT, the RHC, the H\'ajek, the ratio, the product, the GREG and the PEML estimators of population means of $h(y)$ by $\hat{\overline{h}}_{HT}$, $\hat{\overline{h}}_{RHC}$, $\hat{\overline{h}}_{H}$,  $\hat{\overline{h}}_{RA}$, $\hat{\overline{h}}_{PR}$, $\hat{\overline{h}}_{GREG}$ and $\hat{\overline{h}}_{PEML}$, respectively. Now, we give the proofs of Theorems \ref{cor 4}, \ref{thm 2} and \ref{Cor 3}. The proofs of Remark \ref{rem 2} and Theorems \ref{thm 5}, \ref{thm 6}, \ref{thm 4} and \ref{thm 3} are given in Section S$3$ of the supplement.
\begin{proof}[\textbf{Proof of Theorem \ref{cor 4}}]
Let us consider SRSWOR  and  LMS sampling design. It follows from  $(i)$ in  Lemma S$6$ in the supplement that $\sqrt{n}(\hat{\overline{h}}-\overline{h})\xrightarrow{\mathcal{L}}N(0,\Gamma)$ as $\nu\rightarrow\infty$ for some p.d. matrix $\Gamma$, when $\hat{\overline{h}}$ is one of $\hat{\overline{h}}_{HT}$, $\hat{\overline{h}}_{H}$, $\hat{\overline{h}}_{RA}$, $\hat{\overline{h}}_{PR}$, and $\hat{\overline{h}}_{GREG}$ with $d(i,s)$=$(N\pi_i)^{-1}$ under any of these sampling designs. Now, note that $\max_{i\in s}|X_i-\overline{X}|$=$o_p(\sqrt{n})$, and $\sum_{i\in s}\pi_i^{-1}(X_i-\overline{X})/\sum_{i\in s}\pi_i^{-1}(X_i-\overline{X})^2$=$O_p(1/\sqrt{n})$ as $\nu\rightarrow\infty$ under the above sampling designs (see Lemma S$8$ in the supplement). Then, by applying Theorem $1$ of \cite{MR1707846} to each real valued coordinate of $\hat{\overline{h}}_{PEML}$ and $\hat{\overline{h}}_{GREG}$, we get $\sqrt{n}(\hat{\overline{h}}_{PEML}-\hat{\overline{h}}_{GREG})$=$o_p(1)$ as $\nu\rightarrow\infty$ for $d(i,s)$=$(N\pi_i)^{-1}$ under these sampling designs. This implies that $\hat{\overline{h}}_{PEML}$ and $\hat{\overline{h}}_{GREG}$ with $d(i,s)$=$(N\pi_i)^{-1}$ have the same asymptotic distribution. Therefore, if $\hat{\overline{h}}$ is one of $\hat{\overline{h}}_{HT}$, $\hat{\overline{h}}_{H}$, $\hat{\overline{h}}_{RA}$, $\hat{\overline{h}}_{PR}$, and $\hat{\overline{h}}_{GREG}$ and $\hat{\overline{h}}_{PEML}$ with $d(i,s)$=$(N\pi_i)^{-1}$, we have
\begin{equation}\label{eq 4}
\sqrt{n}(g(\hat{\overline{h}})-g(\overline{h}))\xrightarrow{\mathcal{L}}N(0,\Delta^2)\text{ as }\nu\rightarrow\infty
\end{equation}
under any of the above-mentioned sampling designs for some $\Delta^2>0$ by the delta method and the condition $\nabla g(\mu_0)\neq 0$ at $\mu_0$=$\lim_{\nu\rightarrow\infty}\overline{h}$. It can be shown from the proof of  $(i)$ in  Lemma S$6$ in the supplement that $\Delta^2$=$\nabla g(\mu_0) \Gamma_1$ $(\nabla g(\mu_0))^T$, where $\Gamma_1$=$\lim_{\nu\rightarrow\infty}nN^{-2}\times$ $\sum_{i=1}^{N}(\textbf{V}_i-\textbf{T}\pi_i)^T(\textbf{V}_i-\textbf{T}\pi_i)(\pi_i^{-1}-1)$. It can also be shown from Table $1$ in the supplement that under each of the above sampling designs, $\textbf{V}_i$ in $\Gamma_1$ is $h_i$ or $h_i-\overline{h}$ or $h_i-\overline{h}X_i/\overline{X}$ or $h_i+\overline{h}X_i/\overline{X}$ or $h_i-\overline{h}-S_{xh}(X_i-\overline{X})/S_x^2$ if $\hat{\overline{h}}$ is $\hat{\overline{h}}_{HT}$ or $\hat{\overline{h}}_{H}$ or $\hat{\overline{h}}_{RA}$ or $\hat{\overline{h}}_{PR}$, or $\hat{\overline{h}}_{GREG}$ with $d(i,s)$=$(N\pi_i)^{-1}$, respectively.
\par

Now, by  $(i)$ in  Lemma S$7$ in the supplement, we have  
\begin{equation}\label{eq 1}
\sigma^2_{1}=\sigma^2_{2}=(1-\lambda)\lim_{\nu\rightarrow\infty}\sum_{i=1}^N(A_i-\bar{A})^2/N. 
\end{equation}
where $\sigma^2_1$ and $\sigma^2_2$ are as defined in the statement of Lemma S$7$, and $A_i$=$ \nabla g(\mu_0) \textbf{V}_i^T$ for different choices of $\textbf{V}_i$ mentioned in the preceding paragraph. Note that $g(\hat{\overline{h}}_{GREG})$ and $g(\hat{\overline{h}}_{PEML})$ have the same asymptotic distribution under each of SRSWOR and LMS sampling design since $\sqrt{n}(\hat{\overline{h}}_{PEML}-\hat{\overline{h}}_{GREG})$=$o_p(1)$ for $\nu\rightarrow\infty$ under these sampling designs as pointed out earlier in this proof. Further, \eqref{eq 1} implies that $g(\hat{\overline{h}}_{GREG})$ with $d(i,s)$=$(N\pi_i)^{-1}$ has the same asymptotic MSE under SRSWOR and LMS sampling design. Thus $g(\hat{\overline{h}}_{GREG})$ and $g(\hat{\overline{h}}_{PEML})$ with $d(i,s)$=$(N\pi_i)^{-1}$ under SRSWOR and LMS sampling design form class $1$ in Table \ref{table 1.2}.
\par

Next, \eqref{eq 1} yields that $g(\hat{\overline{h}}_{HT})$ has the same asymptotic MSE under SRSWOR and LMS sampling design. It also follows from \eqref{eq 1} that $g(\hat{\overline{h}}_{H})$ has the same asymptotic MSE under SRSWOR and LMS sampling design.  Now, note that $g(\hat{\overline{h}}_{HT})$ and $g(\hat{\overline{h}}_{H})$ coincide under SRSWOR. Thus $g(\hat{\overline{h}}_{HT})$ under SRSWOR, and $g(\hat{\overline{h}}_{HT})$ and  $g(\hat{\overline{h}}_{H})$ under LMS sampling design form class $2$ in Table \ref{table 1.2}. 
\par

Next, \eqref{eq 1} implies that $g(\hat{\overline{h}}_{RA})$ has the same asymptotic MSE under SRSWOR and LMS sampling design. Further, \eqref{eq 1} implies that $g(\hat{\overline{h}}_{PR})$ has the same asymptotic MSE under SRSWOR and LMS sampling design. Thus $g(\hat{\overline{h}}_{RA})$ under SRSWOR and LMS sampling design forms class $3$ in Table \ref{table 1.2}, and $g(\hat{\overline{h}}_{PR})$ under those sampling designs forms class $4$ in Table \ref{table 1.2}.  This completes the proof of Theorem \ref{cor 4}.
\end{proof}
\begin{proof}[\textbf{Proof of Theorem \ref{thm 2}}] 
 Note that C\ref{ass 2} and C\ref{ass B2} hold \textit{a.s.} $[\mathbb{P}]$ since C\ref{ass C1} holds and $E_{\mathbb{P}}(\epsilon_i)^4<\infty$. Also, note that C\ref{ass 3} holds \textit{a.s.} $[\mathbb{P}]$ under SRSWOR and LMS sampling design (see Lemma S$2$ in the supplement). Then, under the above sampling designs, conclusions of  Theorems \ref{cor 4} and \ref{thm 6}  hold \textit{a.s.} $[\mathbb{P}]$ for $d$=$p$=$1$, $h(y)$=$y$ and $g(s)$=$s$. Note that $W_i$=$\nabla g(\overline{h})h_i^T$=$Y_i$. Also, note that the $\Delta^2_i$'s in Table \ref{table 1.3} can be expressed in terms of superpopulation moments of $(Y_i,X_i)$ \textit{a.s.} $[\mathbb{P}]$ by SLLN since $E_{\mathbb{P}}(\epsilon_i)^4<\infty$. Recall from the beginning of Section \ref{subsec 4.1} that we have taken $\sigma_x^2$=$1$. Then, we have $\Delta^2_2-\Delta^2_1$=$(1-\lambda)\sigma^2_{xy}$, $\Delta^2_3-\Delta^2_1$=$(1-\lambda)(\sigma_{xy}-E_{\mathbb{P}}(Y_i)/\mu_1)^2$ and $\Delta^2_4-\Delta^2_1$=$(1-\lambda)(\sigma_{xy}+E_{\mathbb{P}}(Y_i)/\mu_1)^2$ \textit{a.s.} $[\mathbb{P}]$, where $\mu_1$=$E_{\mathbb{P}}(X_i)$ and $\sigma_{xy}$=$cov_{\mathbb{P}}(X_i,Y_i)$. Hence, $\Delta^2_1 < \Delta^2_i$ \textit{a.s.} $[\mathbb{P}]$ for $i$=$2,3,4$.
\par

 Next consider the case of  $0\leq \lambda< E_{\mathbb{P}}(X_i)/b$.  Note that  $n\gamma\rightarrow c$ as $\nu\rightarrow\infty$ for some $c\geq 1-\lambda$ by Lemma S$1$ in the supplement. Also, note that \textit{a.s.} $[\mathbb{P}]$, C\ref{ass 1} holds in the case of RHC sampling design and C\ref{ass 3} holds in the case of any HE$\pi$PS sampling design (see Lemma S$2$ in the supplement).  Then, under RHC and any HE$\pi$PS sampling designs, conclusions of  Theorems \ref{thm 5} and \ref{thm 6}  hold \textit{a.s.} $[\mathbb{P}]$ for $d$=$p$=$1$, $h(y)$=$y$ and $g(s)$=$s$. Further, we have  $\Delta^2_5-\Delta^2_1$=$\big\{E_{\mathbb{P}}\big(Y_i-E_{\mathbb{P}}(Y_i))^2\big(\mu_1/X_i-\lambda\big)-\mu_1^2\sigma_{xy}\big(\sigma_{xy} cov_{\mathbb{P}}(X_i,1/X_i)-2cov_{\mathbb{P}}(Y_i,1/X_i)\big)+\lambda\sigma^2_{xy}\big\}-(1-\lambda)\big\{\sigma^2_y-\sigma_{xy}^2\big\}$, $\Delta^2_6-\Delta^2_5$= $E_{\mathbb{P}}\big(Y_i^2\big(\mu_1/X_i-\lambda\big)\big)-\big\{\lambda E_{\mathbb{P}}(Y_i X_i)-E_{\mathbb{P}}(Y_i)\mu_1\big\}^2/\chi \mu_1-\big\{E_{\mathbb{P}}\big(Y_i-E_{\mathbb{P}}(Y_i)-\sigma_{xy}(X_i-\mu_1)\big)^2\big(\mu_1/X_i-\lambda\big)\big\}$, $\Delta^2_7-\Delta^2_5$=$\big\{\mu_1^2\sigma_{xy}\big(\sigma_{xy} cov_{\mathbb{P}}(X_i,1/X_i)-2cov_{\mathbb{P}}(Y_i,1/X_i)\big)-\lambda\sigma^2_{xy}-\lambda^2 \sigma^2_{xy}/\mu_1\chi\big\}$, $\Delta^2_8-\Delta^2_1$=$c\big\{\mu_1E_{\mathbb{P}}(Y_i-E_{\mathbb{P}}(Y_i))^2/X_i-\mu_1^2\sigma_{xy}(\sigma_{xy}cov_{\mathbb{P}}(X_i,$ $1/X_i)-2cov_{\mathbb{P}}(Y_i,1/X_i))\big\}-(1-\lambda)\big\{\sigma^2_y-\sigma_{xy}^2\big\}$ and $\Delta^2_9-\Delta^2_1$=$c\big\{\mu_1 E_{\mathbb{P}}(Y_i^2/X_i)-E_{\mathbb{P}}^2(Y_i)\big\}-(1-\lambda)\big\{\sigma^2_y-\sigma_{xy}^2\big\}$  \textit{a.s.} $[\mathbb{P}]$, where $\sigma^2_y$=$var_{\mathbb{P}}(Y_i)$,  $\chi$=$\mu_1-\lambda (\mu_2/\mu_1)$ and $\mu_2$=$E_{\mathbb{P}}(X_i)^2$. Here, we note that $\chi$=$E_{\mathbb{P}}\big(X_i^2(\mu_1/X_i-\lambda)\big)/\mu_1>0$ because C\ref{ass C1} holds and C\ref{ass 4} holds with $0\leq \lambda< E_{\mathbb{P}}(X_i)/b$.  Moreover, from the linear model set up, we can show that  $\Delta^2_5-\Delta^2_1$=$\sigma^2(\mu_1\mu_{-1}-1)>0$, $\Delta^2_6-\Delta^2_5$=$ E_{\mathbb{P}}\big\{(\alpha +\beta X_i)-\chi^{-1}X_i(\alpha+\beta\mu_1-\lambda\alpha-\lambda\beta\mu_2/\mu_1)\big\}^2\big\{\mu_1/X_i-\lambda\big\}\geq 0$, $\Delta^2_7-\Delta^2_5$=$\beta^2 E_{\mathbb{P}}\big\{(X_i-\mu_1)-\lambda\chi^{-1}X_i(\mu_1-\mu_2/\mu_1)\big\}^2\big\{\mu_1/X_i-\lambda\big\}\geq 0$, $\Delta^2_8-\Delta^2_1$=$\sigma^2\big(c\mu_1\mu_{-1}-(1-\lambda)\big)\geq c\sigma^2(\mu_1\mu_{-1}-1)> 0$ and $\Delta^2_9-\Delta^2_1$=$\sigma^2\big(c\mu_1 \mu_{-1}-(1-\lambda)\big)+c\alpha^2(\mu_1\mu_{-1}-1) > 0$  \textit{a.s.} $[\mathbb{P}]$, where $\sigma^2$=$E_{\mathbb{P}}(\epsilon_i)^2$.  Note that $\Delta^2_6-\Delta^2_5\geq 0$ and $\Delta^2_7-\Delta^2_5\geq 0$ because C\ref{ass C1} holds and C\ref{ass 4} holds with $0\leq \lambda< E_{\mathbb{P}}(X_i)/b$.   Therefore, $\Delta^2_1< \Delta^2_i$ \textit{a.s.} $[\mathbb{P}]$ for $i$=$2,\ldots, 9  $.  This completes the proof of Theorem \ref{thm 2}.
\end{proof}
\begin{proof}[\textbf{Proof of Theorem \ref{Cor 3}}]
 The proof follows in a straightforward way from Theorem \ref{thm 2}.
\end{proof}

\vskip .55cm
\noindent
Anurag Dey\\
{\it Indian Statistical Institute, Kolkata}
\vskip 1.3pt
\noindent
E-mail: deyanuragsaltlake64@gmail.com
\vskip 1.3pt

\noindent
Probal Chaudhuri\\
{\it Indian Statistical Institute, Kolkata}
\vskip 1.3pt
\noindent
E-mail:probalchaudhuri@gmail.com
\end{document}




\title{\textbf{Supplementary material for ``A comparison of estimators of mean and its functions in finite populations"}}
\author{Anurag Dey and Probal Chaudhuri\\
\textit{Indian Statistical Institute, Kolkata}}
\maketitle


\begin{abstract}
In this supplement, we discuss conditions C$1$ through C$4$ from the main paper and demonstrate situations, where these conditions hold. Then, we state and prove some additional mathematical results. We also give the proofs of Remark $1$ and Theorems $2$, $3$, $6$ and $7$ of the main text. The biased estimators considered in the main paper are then compared empirically with their bias-corrected versions based on jackknifing in terms of MSE. Finally, we provide the numerical results related to the analysis based on both synthetic and real data.\\ 
\end{abstract}

\textbf{Keywords and phrases:} Asymptotic normality, Equivalence classes of estimators, High entropy sampling designs, Inclusion probability, Linear regression model, Rejective sampling design, Relative efficiency, Superpopulation models.

\section{Discussion of conditions and related results}\label{sec s1}
In this section, we demonstrate some situations, when conditions C$1$ through C$4$ in the main article hold.  Before that we prove and state the following lemma. Recall from the paragraph following C$2$ in the main text that $\gamma$=$\sum_{i=1}^n N_i(N_i-1)/N(N-1)$ with $N_i$ being the size of the $i^{th}$ group formed randomly in RHC sampling design. 
 
\begin{lemma}\label{lem 6}
Suppose that $C0$ holds. Then, $n\gamma\rightarrow c$ for some $c\geq 1-\lambda>0$ as $\nu\rightarrow\infty$, where $\lambda$ is as in $C0$.
\end{lemma}
\begin{proof}
Let us first consider the case of $\lambda$=$0$. Note that 
\begin{align}\label{eq 2}
\begin{split}
&n (N/n-1)(N-n)/(N(N-1))\leq n\gamma\leq\\
& n (N/n+1)(N-n)/(N(N-1))
\end{split}
\end{align} 
by $(1)$ in Section $2$ of the main text. Moreover, $n (N/n+1)(N-n)/(N(N-1))$=$(1+n/N)(N-n)/(N-1)\rightarrow 1$ and $n (N/n-1)(N-n)/(N(N-1))$=$(1-n/N)(N-n)/(N-1)\rightarrow 1$ as $\nu\rightarrow\infty$ because C$0$ holds and $\lambda$=$0$. Thus we have $n\gamma\rightarrow 1$ as $\nu\rightarrow\infty$ in this case.
\par

Next, consider the case, when $\lambda>0$ and $\lambda^{-1}$ is an integer. Here, we consider the following sub-cases. Let us first consider the sub-case, when $N/n$ is an integer for all sufficiently large $\nu$. Then, by $(1)$, we have $n\gamma$=$(N-n)/(N-1)$ for all sufficiently large $\nu$. Now, since C$0$ holds, we have 
\begin{equation}
(N-n)/(N-1)\rightarrow 1-\lambda\text{ as }\nu\rightarrow\infty.
\end{equation}
\par

Further, consider the sub-case, when $N/n$ is a non-integer and $N/n-\lambda^{-1}\geq 0$ for all sufficiently large $\nu$. Then by $(1)$ in Section $2$ of the main text, we have 
\begin{equation}\label{eq 3}
n\gamma=(N/(N-1))(n/N) \lfloor N/n\rfloor \big(2-\big((n/N) \lfloor N/n \rfloor \big)-(n/N)\big) 
\end{equation}
for all sufficiently large $\nu$. Now, since C$0$ holds, we have $0\leq N/n-\lambda^{-1}<1$ for all sufficiently large $\nu$. Then, $\lfloor N/n \rfloor$=$\lambda^{-1}$ for all sufficiently large $\nu$, and hence
\begin{equation}\label{eq 4}
(N/(N-1)) (n/N) \lfloor N/n \rfloor \bigg(2-\big((n/N) \lfloor N/n\rfloor\big)-(n/N)\bigg)\rightarrow 1-\lambda
\end{equation}
as $\nu\rightarrow\infty$.
\par

Next, consider the sub-case, when $N/n$ is a non-integer and $N/n-\lambda^{-1}< 0$ for all sufficiently large $\nu$. Then, the result in \eqref{eq 3} holds by $(1)$, and $-1\leq N/n-\lambda^{-1}<0$ for all sufficiently large $\nu$ by C$0$. Therefore, $\lfloor N/n \rfloor$=$\lambda^{-1}-1$ for all sufficiently large $\nu$, and hence the result in \eqref{eq 4} holds. Thus, in the case of $\lambda>0$ and $\lambda^{-1}$ being an integer, $n\gamma$ converges to $1-\lambda$ as $\nu\rightarrow\infty$ through all the sub-sequences, and hence $n\gamma\rightarrow 1-\lambda$ as $\nu\rightarrow\infty$. Thus we have $c$=$1-\lambda$ in this case.
\par
\vspace{.2cm}

Finally, consider the case, when $\lambda>0$, and $\lambda^{-1}$ is a non-integer. Then, $N/n$ must be a non-integer for all sufficiently large $\nu$, and hence $n\gamma$=$(N/(N-1)) (n/N) \lfloor N/n \rfloor \big(2-\big((n/N) \lfloor N/n \rfloor \big)-(n/N)\big)$ for all sufficiently large $\nu$ by $(1)$ in Section $2$ of the main text. Note that in this case, $N/n-\lfloor \lambda^{-1} \rfloor\rightarrow \lambda^{-1}-\lfloor \lambda^{-1} \rfloor\in(0,1)$ as $\nu\rightarrow\infty$ by C$0$. Therefore, $\lfloor \lambda^{-1} \rfloor<N/n<\lfloor \lambda^{-1} \rfloor+1$ for all sufficiently large $\nu$, and hence $\lfloor N/n \rfloor$=$\lfloor \lambda^{-1} \rfloor$ for all sufficiently large $\nu$. Thus $n\gamma\rightarrow \lambda \lfloor \lambda^{-1} \rfloor (2-\lambda \lfloor \lambda^{-1} \rfloor-\lambda)$ as $\nu\rightarrow\infty$ by C$0$. Now, if $m$=$ \lfloor \lambda^{-1} \rfloor$ and $\lambda^{-1}$ is a non-integer, then $(m+1)^{-1}<\lambda<m^{-1}$. Therefore, $\lambda \lfloor \lambda^{-1}\rfloor (2-\lambda \lfloor\lambda^{-1}\rfloor-\lambda)-1+\lambda$=$-\big(1-(2m+1)\lambda+m(m+1)\lambda^2\big)$=$-(1-m\lambda)(1-(m+1)\lambda)>0$. Thus we have $c$=$\lambda \lfloor \lambda^{-1} \rfloor (2-\lambda \lfloor \lambda^{-1} \rfloor-\lambda)>1-\lambda$ in this case. This completes the proof of the Lemma.
\end{proof}
Next, recall $\{\textbf{V}_i\}_{i=1}^N$ from the paragraph preceding the condition C$3$ and $b$ from the condition C$5$ in the main text.  Let us define $\Sigma_1$=$nN^{-2}\sum_{i=1}^{N}(\textbf{V}_i-\textbf{T}\pi_i)^T(\textbf{V}_i-\textbf{T}\pi_i)(\pi_i^{-1}-1)$ and $\Sigma_2$=$n\gamma\overline{X}N^{-1}\sum_{i=1}^{N}(\textbf{V}_i-X_i\overline{\textbf{V}}/\overline{X})^T(\textbf{V}_i-X_i\overline{\textbf{V}}/\overline{X})/ X_i$, where $\textbf{T}$=$\sum_{i=1}^N \textbf{V}_i(1-\pi_i)/\sum_{i=1}^N \pi_i(1-\pi_i)$, the $\pi_i$'s are inclusion probabilities and $\overline{\textbf{V}}$=$\sum_{i=1}^N \textbf{V}_i/N$. Now, we state the following lemma.

\begin{lemma}\label{lem s1}
$(i)$ Suppose that $C0$ and C$5$ hold, and $\{(h(Y_i),X_i):1\leq i \leq N\}$ are generated from a superpopulation distribution $\mathbb{P}$ with $E_{\mathbb{P}}||h(Y_i)||^4<\infty$. Then, C$1$, C$2$ and C$4$ hold \textit{a.s.} $[\mathbb{P}]$.\\ 
$(ii)$ Further, if $C0$ and C$5$ hold, and $E_{\mathbb{P}}||h(Y_i)||^2<\infty$, then C$3$ holds \textit{a.s.} $[\mathbb{P}]$ under SRSWOR and LMS sampling design. Moreover, if $C0$ holds with $0\leq\lambda<E_{\mathbb{P}}(X_i)/b$, C$5$ holds, and $E_{\mathbb{P}}||h(Y_i)||^2<\infty$, then C$3$ holds \textit{a.s.} $[\mathbb{P}]$ under any $\pi$PS sampling design.
\end{lemma}

\begin{proof}
 As before, for simplicity, let us write $h(Y_i)$ as $h_i$. Under the conditions C$5$ and $E_{\mathbb{P}}||h(Y_i)||^4<\infty$, C$1$ holds \textit{a.s.} $[\mathbb{P}]$ by SLLN. Also, under C$5$, C$2$ holds \textit{a.s.} $[\mathbb{P}]$. Next, by SLLN, $\lim_{\nu\rightarrow\infty}\Sigma_2$=$ c\color{black} E_{\mathbb{P}}(X_i)E_{\mathbb{P}}[(h_i-(E_{\mathbb{P}}(X_i))^{-1}X_i$ $E_{\mathbb{P}}(h_i))^T(h_i-(E_{\mathbb{P}}(X_i))^{-1} X_i E_{\mathbb{P}}(h_i))$ $X_i^{-1}]$ \textit{a.s.} $[\mathbb{P}]$ for $\textbf{V}_i$=$h_i$, $h_i-\overline{h}X_i/\overline{X}$ and $h_i+\overline{h}X_i/\overline{X}$ because $n\gamma\rightarrow  c\color{black}$ as $\nu\rightarrow\infty$ by Lemma  S\ref{lem 6}.  Similarly, $\lim_{\nu\rightarrow\infty}\Sigma_2$=$ c\color{black} E_{\mathbb{P}}(X_i)E_{\mathbb{P}}[(h_i$ $-E_{\mathbb{P}}(h_i))^T(h_i-E_{\mathbb{P}}(h_i))/X_i]$ \textit{a.s.} $[\mathbb{P}]$ for $\textbf{V}_i$=$h_i-\overline{h}$, and $\lim_{\nu\rightarrow\infty}\Sigma_2$=$ c\color{black} E_{\mathbb{P}}(X_i)E_{\mathbb{P}}[$ $(h_i-E_{\mathbb{P}}(h_i)-C_{xh}(X_i-E_{\mathbb{P}}(X_i)))^T(h_i-E_{\mathbb{P}}(h_i)-C_{xh}(X_i-E_{\mathbb{P}}(X_i)))/X_i]$ \textit{a.s.} $[\mathbb{P}]$ for $\textbf{V}_i$=$h_i-\overline{h}-S_{xh}(X_i-\overline{X})/S^2_x$. Here, $C_{xh}$=$(E_{\mathbb{P}}(h_i X_i)-E_{\mathbb{P}}(h_i)E_{\mathbb{P}}(X_i))/$ $(E_{\mathbb{P}}(X_i)^2-(E_{\mathbb{P}}(X_i))^2)$. Note that the above limits are p.d. matrices because C$5$ holds. Therefore, C$4$ holds \textit{a.s.} $[\mathbb{P}]$.  This completes the proof of \textsl{(i)} in Lemma S\ref{lem s1}.
\par

 Next, note that $\Sigma_1$=$(1-n/N)(\sum_{i=1}^N\textbf{V}_i^T\textbf{V}_i/N-\overline{\bf{V}}^T\overline{\bf{V}})$ under SRSWOR. Then, C$3$ holds \textit{a.s.} $[\mathbb{P}]$ by directly applying SLLN. Under LMS sampling design, C$3$ can be shown to hold \textit{a.s.} $[\mathbb{P}]$ in the same way as the proof of the result $\sigma_1^2$=$\sigma_2^2$ in the proof of Lemma $2$ in the Appendix. Next, we have  $\lim_{\nu\rightarrow\infty}\Sigma_1$=$E_{\mathbb{P}}\big[ \big\{h_i+\chi^{-1}(E_{\mathbb{P}}(X_i))^{-1}X_i\big(\lambda E_{\mathbb{P}}(h_iX_i)-E_{\mathbb{P}}(h_i) E_{\mathbb{P}}(X_i)\big) \big\}^T \big\{h_i+\chi^{-1}(E_{\mathbb{P}}(X_i))^{-1}X_i\big(\lambda E_{\mathbb{P}}(h_iX_i)-E_{\mathbb{P}}(h_i) E_{\mathbb{P}}(X_i)\big) \big\}\big\{ E_{\mathbb{P}}(X_i)/X_i-\lambda\big\}\big]$  \textit{a.s.} $[\mathbb{P}]$ for $\textbf{V}_i$=$h_i$, $h_i-\overline{h}X_i/\overline{X}$ and $h_i+\overline{h}X_i/\overline{X}$ under any $\pi$PS sampling design (i.e., a sampling design with $\pi_i$=$nX_i/\sum_{i=1}^N X_i$) by SLLN because C$0$ and C$5$ hold, and $E_{\mathbb{P}}||h_i||^2<\infty$.  Here, $\chi$=$E_{\mathbb{P}}(X_i)-\lambda (E_{\mathbb{P}}(X_i)^2/E_{\mathbb{P}}(X_i))$.  Moreover, under any $\pi$PS sampling design, we have  $\lim_{\nu\rightarrow\infty}\Sigma_1$=$E_{\mathbb{P}}\big[ \big\{h_i-E_{\mathbb{P}}(h_i)+\lambda\chi^{-1}(E_{\mathbb{P}}(X_i))^{-1}X_i C_{xh} \big\}^T \big\{h_i-E_{\mathbb{P}}(h_i)+\lambda\chi^{-1}(E_{\mathbb{P}}(X_i))^{-1}X_i C_{xh} \big\}\times$ $\big\{ E_{\mathbb{P}}(X_i)/X_i-\lambda\big\}\big]$  \textit{a.s.} $[\mathbb{P}]$ for $\textbf{V}_i$=$h_i-\overline{h}$ and  $\lim_{\nu\rightarrow\infty}\Sigma_1$= $E_{\mathbb{P}}\big[\big\{h_i-E_{\mathbb{P}}(h_i)-C_{xh}(X_i-E_{\mathbb{P}}(X_i))\big\}^T \big\{h_i-E_{\mathbb{P}}(h_i)-C_{xh}(X_i-E_{\mathbb{P}}(X_i))\big\}\big\{ E_{\mathbb{P}}(X_i)/X_i-\lambda\big\}\big]$  \textit{a.s.} $[\mathbb{P}]$ for $\textbf{V}_i$=$h_i-\overline{h}-S_{xh}(X_i-\overline{X})/S^2_x$.  Note that the above limits are p.d. matrices because C$5$ holds and C$0$ holds with $0\leq\lambda<E_{\mathbb{P}}(X_i)/b$.  Therefore, C$3$ holds \textit{a.s.} $[\mathbb{P}]$ under any $\pi$PS sampling design. This completes the proof of \textsl{(ii)} in Lemma S\ref{lem s1}.
\end{proof}

\section{Additional mathematical details}\label{sec 1s}
In this section, we state and prove some technical results, which will be required to prove the theorems stated in the main text.
\begin{lemma}\label{lem 1}
Suppose that C$2$ holds. Then, LMS sampling design is a high entropy sampling design. Moreover, under LMS sampling design, there exist constants $L, L^{\prime}>0$ such that
\begin{equation}\label{eq 1}
L\leq \min_{1\leq i\leq N}(N\pi_{i}/n)\leq \max_{1\leq i\leq N}(N\pi_{i}/n)\leq L^{\prime}
\end{equation}
for all sufficiently large $\nu$ .
\end{lemma}
The condition \eqref{eq 1} was considered earlier in \cite{wang2011asymptotic}, \cite{MR3670194}, etc. However, the above authors did not discuss whether LMS sampling design satisfies \eqref{eq 1} or not. 
\begin{proof}
Suppose that $P(s)$ and $R(s)$ denote LMS sampling design and SRSWOR, respectively. Note that SRSWOR is a rejective sampling design. Then, $P(s)$=$(\overline{x}/\overline{X})/^NC_n$ and $R(s)$=$(^NC_n)^{-1}$, where $\overline{x}$=$\sum_{i\in s}X_i/n$ and $s\in\mathcal{S}$. By Cauchy-Schwarz inequality, we have $D(P||R)$=$E_{R}((\overline{x}/\overline{X})\log(\overline{x}/$ $\overline{X}))\leq K_1 E_{R}|\overline{x}/\overline{X}-1|\leq K_1 E_{R}(\overline{x}/\overline{X}-1)^2$ for some $K_1>0$ since C$2$ holds, and $\log(x)\leq|x-1|$ for $x>0$. Here $E_{R}$ denotes the expectation with respect to $R(s)$. Therefore, $nD(P||R)\leq K_1 (1-f)(N/(N-1)) (S^2_x/\overline{X}^2)\leq 2 K_1(\sum_{i=1}^N X_i^2/N\overline{X}^2)\leq 2 K_1(\max_{1\leq i \leq N} X_{i}/\min_{1\leq i \leq N} X_{i})^2$=$O(1)$ as $\nu\rightarrow\infty$, where $f$=$n/N$. Hence, $D(P||R)\rightarrow 0$ as $\nu\rightarrow\infty$. Thus LMS sampling design is a high entropy sampling design. 
\par

Next, suppose that $\{\pi_{i}\}_{i=1}^N$ denote inclusion probabilities of $P(s)$. Then, we have $\pi_{i}$=$(n-1)/(N-1)+(X_i/\sum_{i=1}^N X_i)((N-n)/(N-1))$ and $\pi_{i}-n/N$=$-(N-n)(N(N-1))^{-1}(X_i/\overline{X}-1)$. Further,
$$
\frac{|\pi_{i}-n/N|}{n/N}=\frac{N-n}{n(N-1)}\left|\frac{X_i}{\overline{X}}-1\right|\leq\frac{N-n}{n(N-1)}\left(\frac{\max_{1\leq i \leq N} X_{i}}{\min_{1\leq i \leq N} X_{ i}}+1\right).
$$
Therefore, $\max_{1\leq i \leq N}|N\pi_{i}/n-1|\rightarrow0$ as $\nu\rightarrow\infty$ by C$2$. Hence, $K_2\leq \min_{1\leq i\leq N}(N\pi_{i}/n)$ $\leq \max_{1\leq i\leq N} (N\pi_{i}/n)\leq K_3$ for all sufficiently large $\nu$ and some constants $K_2>0$ and $K_3>0$. Thus \eqref{eq 1} holds under LMS sampling design.
\end{proof}
Next, suppose that $\{\textbf{V}_i\}_{i=1}^N$, $\overline{\bf{V}}$, $\Sigma_1$ and $\Sigma_2$ are as in the previous Section \ref{sec s1}. Let us define $\hat{\overline{V}}_1$= $\sum_{i\in s}(N\pi_i)^{-1}V_i$ and $\hat{\overline{V}}_2$=$\sum_{i\in s} G_iV_i/NX_i$, where $G_i$'s are as in the paragraph containing Table $8$ in the main article. Now, we state the following lemma.
\begin{lemma}\label{lem 4} 
Suppose that $C0$ through C$3$ hold. Then, under SRSWOR, LMS sampling design and any HE$\pi$PS sampling design, we have $\sqrt{n}(\hat{\overline{\bf{V}}}_1-\overline{\bf{V}})\xrightarrow{\mathcal{L}}N(0,\Gamma_1)$ as $\nu\rightarrow\infty$, where $\Gamma_1$=$\lim_{\nu\rightarrow\infty}\Sigma_1$. Further, suppose that C$0$ through C$2$ and C$4$ hold. Then, we have $\sqrt{n}(\hat{\overline{\bf{V}}}_2-\overline{\bf{V}})\xrightarrow{\mathcal{L}}N(0,\Gamma_2)$ as $\nu\rightarrow\infty$ under RHC sampling, where $\Gamma_2$=$\lim_{\nu\rightarrow\infty}\Sigma_2$.
\end{lemma}
\begin{proof}
Note that SRSWOR is a high entropy sampling design since it is a rejective sampling design. Also, \eqref{eq 1} in Lemma S\ref{lem 1} holds trivially under SRSWOR. It follows from Lemma S\ref{lem 1} that LMS sampling design is a high entropy sampling design, and \eqref{eq 1} holds under this sampling design. Further, any HE$\pi$PS sampling design satisfies \eqref{eq 1} since C$2$ holds. Now, fix $\epsilon>0$ and $\textbf{m}\in \mathbb{R}^p$. Suppose that $L(\epsilon,\textbf{m})$=$(n^{-1}N^2 \textbf{m}\Sigma_1 \textbf{m}^T)^{-1} \sum_{i \in G(\epsilon,\textbf{m})}(\textbf{m}$ $(\textbf{V}_i-\textbf{T} \pi_i)^T)^2 (\pi^{-1}_i-1)$ for $G(\epsilon,\textbf{m})$=$\{1\leq i \leq N: |\textbf{m}(\textbf{V}_i-\textbf{T} \pi_i)^T|>\epsilon \pi_i N( n^{-1}$ $\textbf{m}\Sigma_1 \textbf{m}^T)^{1/2}\}$, $\textbf{T}$=$\sum_{i=1}^N \textbf{V}_i(1-\pi_i)/\sum_{i=1}^N \pi_i(1-\pi_i)$ and  $\textbf{Z}_i$=$(n/N\pi_i)\textbf{V}_i$ $-(n/N)\textbf{T}$, $i$=$1,\ldots,N$. Then, given any $\eta>0$, $L(\epsilon,\textbf{m})\leq (\textbf{m}\Sigma_1 \textbf{m}^T)^{-(1+\eta/2)}$ $ n^{-\eta/2}\epsilon^{-\eta}N^{-1}$ $\sum_{i=1}^N (||\textbf{m}||||\textbf{Z}_i||)^{2+\eta}(N\pi_i/n)$ since $|\textbf{m}\textbf{Z}_i^T|/(\sqrt{n} \epsilon (\textbf{m}\Sigma_1 \textbf{m}^T)^{1/2})>1$ for any $i\in G(\epsilon,\textbf{m})$. It follows from Jensen's inequality that $N^{-1}\sum_{i=1}^N ||\textbf{Z}_i||^{2+\eta}$ $(N\pi_i/n)\leq 2^{1+\eta}(N^{-1} \sum_{i=1}^N ||\textbf{V}_i (n/N\pi_i)||^{2+\eta}$ $(N\pi_i/n)+||(n/N)\textbf{T}||^{2+\eta})$ since $\sum_{i=1}^N \pi_i$=$n$. It also follows from C$1$, C$2$ and Jensen's inequality that $\sum_{i=1}^N ||\textbf{V}_i||^{2+\eta}$ $/N$=$O(1)$ as $\nu\rightarrow\infty$ for any $0<\eta\leq 2$. Further, $\sum_{i=1}^N\pi_i(1-\pi_i)/n$ is bounded away from $0$ as $\nu\rightarrow\infty$ under SRSWOR, LMS sampling design and any HE$\pi$PS sampling design because \eqref{eq 1} holds under these sampling designs, and C$0$ holds. Therefore, $N^{-1} \sum_{i=1}^N ||\textbf{V}_i (n/N\pi_i)||^{2+\eta}(N\pi_i/n)$=$O(1)$ and $||(n/N)\textbf{T}||^{2+\eta}$=$O(1)$, and hence $N^{-1}\sum_{i=1}^N ||\textbf{Z}_i||^{2+\eta}(N\pi_i/n)$=$O(1)$ as $\nu\rightarrow\infty$ under the above sampling designs. Then, $L(\epsilon,\textbf{m})\rightarrow 0$ as $\nu\rightarrow\infty$ for any $\epsilon>0$ under all of these sampling designs since C$3$ holds. Therefore, $\inf\{\epsilon>0:L(\epsilon,\textbf{m})\leq \epsilon\}\rightarrow 0$ as $\nu\rightarrow\infty$, and consequently the H\'ajek-Lindeberg condition holds for $\{\textbf{m}\textbf{V}_i^T\}_{i=1}^N$ under each of the above sampling designs. Also, $\sum_{i=1}^N \pi_i(1-\pi_i)\rightarrow \infty$ as $\nu\rightarrow\infty$ under these sampling designs. Then, from Theorem $5$ in \cite{MR1624693}, $\sqrt{n}\textbf{m}(\hat{\overline{\bf{V}}}_1-\overline{\bf{V}})^T\xrightarrow{\mathcal{L}}N(0,\textbf{m} \Gamma_1 \textbf{m}^T)$ as $\nu\rightarrow\infty$ under each of the above sampling designs for any $\textbf{m}\in\mathbb{R}^{p}$ and $\Gamma_1$=$\lim_{\nu\rightarrow\infty}\Sigma_1$. Hence, $\sqrt{n}(\hat{\overline{\bf{V}}}_1-\overline{\bf{V}})\xrightarrow{\mathcal{L}}N(0,\Gamma_1)$ as $\nu\rightarrow\infty$ under the above-mentioned sampling designs.
\vspace*{.2cm} 

Next, define $L(\textbf{m})$=$n\gamma(\max_{1\leq i\leq N} X_i)(N^{-1}\sum_{i=1}^n N_i^3(N_i-1)\sum_{i=1}^N(\textbf{m} (\textbf{V}_i \overline{X}/$ $X_i-\overline{\textbf{V}})^T)^4 X_i)^{1/2}$ $\times(\overline{X}^{3/2}\sum_{i=1}^n N_i(N_i-1) \textbf{m}\Sigma_2 \textbf{m}^T)^{-1}$, where $\gamma$=$\sum_{i=1}^n N_i (N_i-1)/N(N-1)$ as before. Note that as $\nu\rightarrow\infty$, $(N^{-1}\sum_{i=1}^N(\textbf{m}(\textbf{V}_i \overline{X}/ X_i-\overline{\textbf{V}})^T)^4 (X_i/\overline{X}))^{1/2}$=$O(1)$ and $(\max_{1\leq i\leq N} X_i)/\overline{X}$ =$O(1)$ since C$1$ and C$2$ hold.  Now, recall from Section $2$ in the main text that the $N_i$'s are considered as in $(1)$.  Then, under C$0$, we have $(\sum_{i=1}^n N_i^3(N_i-1))^{1/2}(\sum_{i=1}^n N_i(N_i-1))^{-1}$=$O(1/\sqrt{n})$ and $n\gamma$=$O(1)$ as $\nu\rightarrow\infty$. Therefore, $L(\textbf{m})\rightarrow 0$ as $\nu\rightarrow\infty$ since C$4$ holds. This implies that condition C$1$ in \cite{MR844032} holds for $\{\textbf{m}\textbf{V}_i^T\}_{i=1}^N$. Therefore, by Theorem $2.1$ in \cite{MR844032}, $\sqrt{n}\textbf{m}(\hat{\overline{\bf{V}}}_2-\overline{\bf{V}})^T \xrightarrow{\mathcal{L}}N(0,\textbf{m} \Gamma_2 \textbf{m}^T)$ as $\nu\rightarrow\infty$ under RHC sampling design for any $\textbf{m}\in\mathbb{R}^{p}$ and $\Gamma_2$=$\lim_{\nu\rightarrow\infty}\Sigma_2$. Hence, $\sqrt{n}(\hat{\overline{\bf{V}}}_2-\overline{\bf{V}})\xrightarrow{\mathcal{L}}N(0,\Gamma_2)$ as $\nu\rightarrow\infty$ under RHC sampling design.
\end{proof} 
Next, suppose that $\overline{\bf{W}}$=$\sum_{i=1}^N \textbf{W}_i/N$, $\hat{\overline{\bf{W}}}_1$=$\sum_{i\in s}(N\pi_i)^{-1}\textbf{W}_i$ and $\hat{\overline{\bf{W}}}_2$= $\sum_{i\in s}G_i \textbf{W}_i/NX_i$ for $\textbf{W}_i$=$(h_i, X_ih_i,$ $X_i^2)$, $i$=$1,\ldots,N$. Let us also define $\hat{\overline{X}}_1$=$\sum_{i\in s}$ $(N\pi_i)^{-1}X_i$. Now, we state the following lemma.
\begin{lemma}\label{lem 2}
Suppose that $C0$ through C$2$ hold. Then, under SRSWOR, LMS sampling design and any HE$\pi$PS sampling design, we have $\hat{\overline{\bf{W}}}_1-\overline{\bf{W}}$=$o_p(1)$, $\sqrt{n}(\hat{\overline{X}}_1-\overline{X})$=$O_p(1)$ and $\sqrt{n}(\sum_{i\in s}(N\pi_i)^{-1}$ $-1)$=$O_p(1)$ as $\nu\rightarrow\infty$. Moreover, under RHC sampling design, we have $\hat{\overline{\bf{W}}}_2-\overline{\bf{W}}$=$o_p(1)$ and $\sqrt{n}(\sum_{i\in s}G_i/NX_i-1)$=$O_p(1)$ as $\nu\rightarrow\infty$.
\end{lemma}
\begin{proof}
We first show that as $\nu\rightarrow\infty$, $\hat{\overline{\bf{W}}}_1-\overline{\bf{W}}$=$o_p(1)$, $\sqrt{n}(\hat{\overline{X}}_1-\overline{X})$=$O_p(1)$ and $\sqrt{n}(\sum_{i\in s}($ $N\pi_i)^{-1}-1)$=$O_p(1)$ under a high entropy sampling design $P(s)$ satisfying \eqref{eq 1} in Lemma S\ref{lem 1}. Fix $\textbf{m}\in\mathbb{R}^{2p+1}$. Suppose that $\tilde{R}(s)$ is a rejective sampling design with inclusion probabilities equal to those of $P(s)$ (cf. \cite{MR1624693}). Under $\tilde{R}(s)$, $var(\textbf{m}(\sqrt{n}(\hat{\overline{\textbf{W}}}_1-\overline{\textbf{W}})^T))$=$\textbf{m}(nN^{-2}$ $\sum_{i=1}^{N}(\textbf{W}_i-\textbf{T}\pi_i)^T(\textbf{W}_i-\textbf{T}\pi_i)(\pi_i^{-1}-1))\textbf{m}^T(1+e)$ (see Theorem $6.1$ in \cite{MR0178555}), where $\textbf{T}$=$\sum_{i=1}^N \textbf{W}_i(1-\pi_i)/\sum_{i=1}^N \pi_i(1-\pi_i)$, and $e\rightarrow 0$ as $\nu\rightarrow\infty$ whenever $\sum_{i=1}^N \pi_i(1-\pi_i)\rightarrow\infty$ as $\nu\rightarrow\infty$. Note that \eqref{eq 1} holds under $\tilde{R}(s)$, and hence $\sum_{i=1}^N \pi_i(1-\pi_i)\rightarrow\infty$ as $\nu\rightarrow\infty$ under $\tilde{R}(s)$ because \eqref{eq 1} holds under $P(s)$, and C$0$ holds. Then, $\textbf{m}(nN^{-2}\sum_{i=1}^{N}(\textbf{W}_i-\textbf{T}\pi_i)^T(\textbf{W}_i-\textbf{T}\pi_i)(\pi_i^{-1}-1))\textbf{m}^T\leq n N^{-2}\sum_{i=1}^N (\textbf{m}\textbf{W}_i^T)^2/\pi_i$=$O(1)$ under $\tilde{R}(s)$ since C$1$ holds. Therefore, $\sqrt{n}(\hat{\overline{\textbf{W}}}_1-\overline{\textbf{W}})$=$O_p(1)$ as $\nu\rightarrow\infty$ under $\tilde{R}(s)$ since $var(\textbf{m}(\sqrt{n}(\hat{\overline{\textbf{W}}}_1-\overline{\textbf{W}})^T))$=$O(1)$ as $\nu\rightarrow\infty$ for any $\textbf{m}\in\mathbb{R}^{2p+1}$ under $\tilde{R}(s)$. Now, $\sum_{s\in E}P(s)\leq\sum_{s\in E}\tilde{R}(s)+\sum_{s\in\mathcal{S}}|P(s)-\tilde{R}(s)|\leq \sum_{s\in E}\tilde{R}(s)+(2D(P||\tilde{R}))^{1/2}$ $\leq \sum_{s\in E}\tilde{R}(s)+(2D(P||R))^{1/2}$ (see Lemmas $2$ and $3$ in \cite{MR1624693}), where $E$=$\{s\in\mathcal{S}:||\sqrt{n}(\hat{\overline{\textbf{W}}}_1-\overline{\textbf{W}})||>\delta\}$ for $\delta>0$ and $R(s)$ is any other rejective sampling design. Let us consider a rejective sampling design $R(s)$ such that $D(P||R)\rightarrow 0$ as $\nu\rightarrow\infty$. Therefore, given any $\epsilon>0$, there exists a $\delta>0$ such that $\sum_{s\in E}P(s)\leq \epsilon$ for all sufficiently large $\nu$. Hence, as $\nu\rightarrow\infty$, $\sqrt{n}(\hat{\overline{\textbf{W}}}_1-\overline{\textbf{W}})$=$O_p(1)$ and $\hat{\overline{\textbf{W}}}_1-\overline{\textbf{W}}$=$o_p(1)$ under $P(s)$. Similarly, we can show that as $\nu\rightarrow\infty$, $\sqrt{n}(\hat{\overline{X}}_1-\overline{X})$=$O_p(1)$ and $\sqrt{n}(\sum_{i\in s}(N\pi_i)^{-1}-1)$=$O_p(1)$ under $P(s)$. Now, recall from the proof of Lemma S\ref{lem 4} that SRSWOR and LMS sampling design are high entropy sampling designs, and they satisfy \eqref{eq 1}. Also, any HE$\pi$PS sampling design satisfies \eqref{eq 1}. Therefore, as $\nu\rightarrow\infty$, $\hat{\overline{\bf{W}}}_1-\overline{\bf{W}}$=$o_p(1)$, $\sqrt{n}(\hat{\overline{X}}_1-\overline{X})$=$O_p(1)$ and $\sqrt{n}(\sum_{i\in s}($ $N\pi_i)^{-1}-1)$=$O_p(1)$ under the above-mentioned sampling designs. 
\par

Under RHC sampling design, $var(\textbf{m}(\sqrt{n}(\hat{\overline{\textbf{W}}}_2-\overline{\textbf{W}})^T))$=$\textbf{m}(n\gamma\overline{X}N^{-1}\sum_{i=1}^{N}$ $(\textbf{W}_i-X_i\overline{\textbf{W}}/\overline{X})^T(\textbf{W}_i-X_i\overline{\textbf{W}}/\overline{X})/ X_i)\textbf{m}^T$ (see \cite{MR844032}). Recall from the proof of Lemma S\ref{lem 4} that $n\gamma$=$O(1)$ as $\nu\rightarrow\infty$. Then, $var(\textbf{m}(\sqrt{n}(\hat{\overline{\textbf{W}}}_2-\overline{\textbf{W}})^T))\leq n\gamma(\overline{X}/N)\sum_{i=1}^N (\textbf{m}\textbf{W}_i^T)^2/X_i$ =$O(1)$ as $\nu\rightarrow\infty$ since C$1$ and C$2$ hold. Hence, as $\nu\rightarrow\infty$, $\sqrt{n}(\hat{\overline{\textbf{W}}}_2-\overline{\textbf{W}})$=$O_p(1)$ and $\hat{\overline{\textbf{W}}}_2-\overline{\textbf{W}}$=$o_p(1)$ under RHC sampling design. Similarly, we can show that as $\nu\rightarrow\infty$, $\sqrt{n}(\sum_{i\in s}G_i/NX_i-1)$=$O_p(1)$ under RHC sampling design.
\end{proof}

Recall from the $2^{nd}$ paragraph in the Appendix that we denote the HT, the RHC, the H\'ajek, the ratio, the product, the GREG and the PEML estimators of population means of $h(y)$ by $\hat{\overline{h}}_{HT}$, $\hat{\overline{h}}_{RHC}$, $\hat{\overline{h}}_{H}$,  $\hat{\overline{h}}_{RA}$, $\hat{\overline{h}}_{PR}$, $\hat{\overline{h}}_{GREG}$ and $\hat{\overline{h}}_{PEML}$, respectively. Suppose that $\hat{\overline{h}}$ denotes one of $\hat{\overline{h}}_{HT}$, $\hat{\overline{h}}_{H}$, $\hat{\overline{h}}_{RA}$, $\hat{\overline{h}}_{PR}$, and $\hat{\overline{h}}_{GREG}$ with $d(i,s)$=$(N\pi_i)^{-1}$. Then, a Taylor type expansion of $\hat{\overline{h}}-\overline{h}$ can be obtained as $\hat{\overline{h}}-\overline{h}$=$\Theta(\hat{\overline{\textbf{V}}}_1-\overline{\textbf{V}})+\textbf{Z}$, where $\hat{\overline{\textbf{V}}}_1$=$\sum_{i\in s}(N\pi_i)^{-1}\textbf{V}_i$, and the $\textbf{V}_i$'s, $\Theta$ and $\textbf{Z}$ are as described in Table \ref{table B} below. 
\begin{table}[h]
\caption{Expressions of $\textbf{V}_i$, $\Theta$ and $\textbf{Z}$ for different $\hat{\overline{h}}$'s}
\label{table B}
\begin{center}
\begin{tabular}{|c|c|c|c|} 
\hline
$\hat{\overline{h}}$ & $\textbf{V}_i$ & $\Theta$ & $\textbf{Z}$\\
\hline
$\hat{\overline{h}}_{HT}$& $h_i$ &$1$ & $0$\\
\hline
$\hat{\overline{h}}_{H}$ & $h_i-\overline{h}$ & $(\sum_{i\in s}(N\pi_i)^{-1})^{-1}$ & $0$ \\
\hline
$\hat{\overline{h}}_{RA}$ &$h_i-\overline{h}X_i/\overline{X}$ &$\overline{X}/\hat{\overline{X}}_1$ & $0$\\
\hline
$\hat{\overline{h}}_{PR}$& $h_i+\overline{h}X_i/\overline{X}$ &$\hat{\overline{X}}_1/\overline{X}$ & $-(1-\hat{\overline{X}}_1/\overline{X}))^2\overline{h}$\\
\hline
$\hat{\overline{h}}_{GREG}$  with& $h_i-\overline{h}-$ &\multirow{2}{*}{$(\sum_{i\in s}(N\pi_i)^{-1})^{-1}$}& $(\hat{\overline{X}}_2-\overline{X})\times$\\
$d(i,s)$=$(N\pi_i)^{-1}$ & $S_{xh}(X_i-\overline{X})/S_x^2$ &  & $(S_{xh}/S_x^2-\hat{\beta}_1)$\\
\hline
\hline
$\hat{\overline{h}}_{RHC}$& $h_i$ &$1$ & $0$\\
\hline
$\hat{\overline{h}}_{GREG}$ with & $h_i-\overline{h}-$ &\multirow{2}{*}{$(\sum_{i\in s}G_i/NX_i)^{-1}$}& $\overline{X}((\sum_{i\in s}G_i/NX_i)^{-1}$ \\
$d(i,s)$=$G_i/N X_i$ & $S_{xh}(X_i-\overline{X})/S_x^2$ &  &$-1)(S_{xh}/S_x^2-\hat{\beta}_2)$ \\
\hline
\end{tabular}
\end{center}
\end{table}
On the other hand, if $\hat{\overline{h}}$ is either $\hat{\overline{h}}_{RHC}$ or $\hat{\overline{h}}_{GREG}$ with $d(i,s)$=$G_i/NX_i$, a Taylor type expansion of $\hat{\overline{h}}-\overline{h}$ can be obtained as $\hat{\overline{h}}-\overline{h}$=$\Theta(\hat{\overline{\textbf{V}}}_2-\overline{\textbf{V}})+\textbf{Z}$. Here, $\hat{\overline{\textbf{V}}}_2$=$\sum_{i\in s} G_i\textbf{V}_i/NX_i$, the $G_i$'s are as in the paragraph containing Table $8$ in the main text, and the $\textbf{V}_i$'s, $\Theta$ and $\textbf{Z}$ are once again described in Table \ref{table B}. In Table \ref{table B}, $\hat{\overline{X}}_1$=$\sum_{i\in s}(N\pi_i)^{-1}X_i$, $\hat{\overline{X}}_2$=$\hat{\overline{X}}_1/\sum_{i\in s}(N\pi_i)^{-1}$, $\hat{\beta}_1$=$(\sum_{i\in s}(N\pi_i)^{-1}$ $\sum_{i\in s}(N\pi_i)^{-1} h_i X_i-\hat{\overline{h}}_{HT}\hat{\overline{X}}_1)/(\sum_{i\in s}(N\pi_i)^{-1}\sum_{i\in s}(N\pi_i)^{-1}X_i^2-(\hat{\overline{X}}_1)^2)$ and $\hat{\beta}_2$=$(\sum_{i\in s}(G_i/NX_i)\sum_{i\in s}(G_i h_i/N)-\hat{\overline{h}}_{RHC}\overline{X})/(\sum_{i\in s}(G_i/NX_i)$ $\sum_{i\in s}(G_i X_i/N)$ $-\overline{X}^2)$. Now, we state the following Lemma.

\begin{lemma}\label{lem 3}
 $(i)$ Suppose that $C0$ through C$3$ hold. Further, suppose that $\hat{\overline{h}}$ is one of $\hat{\overline{h}}_{HT}$, $\hat{\overline{h}}_{H}$, $\hat{\overline{h}}_{RA}$, $\hat{\overline{h}}_{PR}$, and $\hat{\overline{h}}_{GREG}$ with $d(i,s)$=$(N\pi_i)^{-1}$. Then, under SRSWOR, LMS sampling design and any HE$\pi$PS sampling design,
\begin{equation}\label{eq s2} 
\sqrt{n}(\hat{\overline{h}}-\overline{h})\xrightarrow{\mathcal{L}}N(0,\Gamma)\text{ as }\nu\rightarrow\infty
\end{equation} 
for some p.d. matrix $\Gamma$.\\ 
$(ii)$ Next, suppose that $C0$ through C$2$ and C$4$ hold, and $\hat{\overline{h}}$ is $\hat{\overline{h}}_{RHC}$ or $\hat{\overline{h}}_{GREG}$ with $d(i,s)$=$G_i/NX_i$. Then, \eqref{eq s2} holds under RHC sampling design. 
\end{lemma}

\begin{proof}
 It can be shown from Lemma S\ref{lem 4} that $\sqrt{n}(\hat{\overline{\textbf{V}}}_1-\overline{\textbf{V}})\xrightarrow{\mathcal{L}}N(0,\Gamma_1)$ as $\nu\rightarrow\infty$ under SRSWOR, LMS sampling design and any HE$\pi$PS sampling design, where $\Gamma_1$=$\lim_{\nu\rightarrow\infty}nN^{-2}\times$ $\sum_{i=1}^{N}(\textbf{V}_i-\textbf{T}\pi_i)^T(\textbf{V}_i-\textbf{T}\pi_i)(\pi_i^{-1}-1)$ with $\textbf{T}$=$\sum_{i=1}^N \textbf{V}_i(1-\pi_i)/\sum_{i=1}^N\pi_i(1-\pi_i)$. Note that $\Gamma_1$ is a p.d. matrix under each of the above sampling designs as C$3$ holds under these sampling designs. Let us now consider from  Table \ref{table B} various choices of $\Theta$ and $\textbf{Z}$ corresponding to $\hat{\overline{h}}_{HT}$, $\hat{\overline{h}}_{H}$, $\hat{\overline{h}}_{RA}$, $\hat{\overline{h}}_{PR}$, and $\hat{\overline{h}}_{GREG}$ with $d(i,s)$=$(N\pi_i)^{-1}$. Then, it can be shown from Lemma S\ref{lem 2} that for each of these choices, $\sqrt{n}\textbf{Z}$=$o_p(1)$ and $\Theta-1$=$o_p(1)$ as $\nu\rightarrow\infty$ under the above-mentioned sampling designs.  Therefore, \eqref{eq s2} holds under those sampling designs with $\Gamma$=$\Gamma_1$.  This completes the proof of $(i)$ in Lemma \ref{lem 3}
\par

 We can show from Lemma S\ref{lem 4} that $\sqrt{n}(\hat{\overline{\textbf{V}}}_2-\overline{\textbf{V}})\xrightarrow{\mathcal{L}}N(0,\Gamma_2)$ as $\nu\rightarrow\infty$ under RHC sampling design, where $\Gamma_2$=$\lim_{\nu\rightarrow\infty} n\gamma$ $\overline{X}N^{-1}$ $\sum_{i=1}^{N}(\textbf{V}_i-X_i\overline{\textbf{V}}/\overline{X})^T(\textbf{V}_i-X_i\overline{\textbf{V}}/\overline{X})/ X_i$ with $\gamma$=$\sum_{i=1}^n N_i(N_i-1)/N(N-1)$. Note that $\Gamma_2$ is a p.d. matrix since C$4$ holds. Let us now consider from Table \ref{table B} different choices of $\Theta$ and $\textbf{Z}$ corresponding to $\hat{\overline{h}}_{RHC}$, and $\hat{\overline{h}}_{GREG}$ with $d(i,s)$=$G_i/NX_i$. Then, it follows from Lemma S\ref{lem 2} that for each of these choices, $\sqrt{n}\textbf{Z}$=$o_p(1)$ and $\Theta-1$=$o_p(1)$ as $\nu\rightarrow\infty$ under RHC sampling design. Therefore, \eqref{eq s2} holds under RHC sampling design with $\Gamma$=$\Gamma_2$.  This completes the proof of $(ii)$ in Lemma \ref{lem 3}
\end{proof}
 Let $\{\textbf{V}_i\}_{i=1}^N$ be as described in Table \ref{table B}. Recall $\Sigma_1$  and $\Sigma_2$ from the paragraph preceding Lemma S\ref{lem s1} in this supplement. Note that the expression of $\Sigma_1$ remains the same for different HE$\pi$PS sampling designs. Also, recall from the paragraph preceding Theorem $3$ in the main text that  $\phi$=$\overline{X}-(n/N)\sum_{i=1}^N X_i^2/N\overline{X}$.  Now, we state the following lemma.

\begin{lemma}\label{thm 8}
$(i)$ Suppose that $C0$ through C$3$ hold. Further, suppose that $\sigma^2_1$ and $\sigma^2_2$ denote $\lim_{\nu\rightarrow\infty}\nabla g(\mu_0)\Sigma_1\nabla g(\mu_0)^T$ under SRSWOR and LMS sampling design, respectively, where $\mu_0$=$\lim_{\nu\rightarrow\infty}\overline{h}$. Then, we have $\sigma^2_{1}$=$\sigma^2_{2}$=$(1-\lambda)\lim_{\nu\rightarrow\infty}\sum_{i=1}^N (A_i-\bar{A})^2/N$ for $A_i$=$\nabla g(\mu_0)\textbf{V}_i^T$, $i$=$1,\ldots,N$. \\
$(ii)$ Next, suppose that C$4$ holds, and $\sigma^2_3$=$\lim_{\nu\rightarrow\infty}$ $\nabla g(\mu_0) \Sigma_2 \nabla g(\mu_0)^T$ in the case of RHC sampling design. Then, we have $\sigma^2_{3}$=$\lim_{\nu\rightarrow\infty} n\gamma ((\overline{X}/N)$ $\sum_{i=1}^N A_i^2/X_i-\bar{A}^2)$. On the other hand, if $C0$ through C$3$ hold, and $\sigma^2_4$=$\lim_{\nu\rightarrow\infty}$ $\nabla g(\mu_0)\Sigma_1\nabla g(\mu_0)^T$ under any HE$\pi$PS sampling design, then we have $\sigma^2_4$= $\lim_{\nu\rightarrow\infty}\big\{(1/N)\sum_{i=1}^N A_i^2\big((\overline{X}/X_i)-(n/N)\big) -\phi^{-1} \overline{X}^{-1}\times$ $\big((n/N)\sum_{i=1}^N A_i X_i/N-\overline{A} \overline{X}\big)^2 \big\}$. Further, if $C0$ holds with $\lambda$=$0$ and C$1$ through C$3$ hold, then we have $\sigma^2_{4}$=$\sigma^2_{3}$=$\lim_{\nu\rightarrow\infty} ((\overline{X}/N)$ $\sum_{i=1}^N A_i^2/X_i-\bar{A}^2)$.
\end{lemma}

\begin{proof}
 Let us first note that the limits in the expressions of $\sigma^2_1$ and $\sigma^2_2$ exist in view of C$3$. Also, note that $\nabla g(\mu_0)\Sigma_1\nabla g(\mu_0)^T$=$nN^{-2}\sum_{i=1}^{N}(A_i-T_A \pi_ i)^2(\pi^{-1}_i-1)$=$nN^{-2}$ $[\sum_{i=1}^{N} A_i^2(\pi^{-1}_i-1)-(\sum_{i=1}^N A_i(1-\pi_i))^2/\sum_{i=1}^N \pi_i(1-\pi_i)]$, where $T_A$=$\sum_{i=1}^N A_i(1-\pi_i)/\sum_{i=1}^N \pi_i(1-\pi_i)$ and $A_i$=$\nabla g(\mu_0)\textbf{V}_i^T$. Now, substituting $\pi_i$=$n/N$ in the above expression for SRSWOR, we get $\sigma^2_1$= $\lim_{\nu\rightarrow \infty} nN^{-2} $ $[\sum_{i=1}^{N} A_i^2(N/n-1)-(\sum_{i=1}^N A_i(1-n/N))^2/n(1-n/N)]$=$\lim_{\nu\rightarrow \infty}$ $(1-n/N)$ $\sum_{i=1}^N(A_i-\bar{A})^2/N$. Since C$0$ holds, we have $\sigma^2_1$=$(1-\lambda)\lim_{\nu\rightarrow \infty}$ $\sum_{i=1}^N(A_i-\bar{A})^2/N$. Let $\{\pi_{i}\}_{i=1}^N$ be the inclusion probabilities of LMS sampling design. Then, $\sigma_2^2-\sigma_1^2$=$\lim_{\nu\rightarrow \infty} nN^{-2}[\sum_{i=1}^{N}A_i^2(\pi_{i}^{-1} -N/n)-((\sum_{i=1}^{N}A_i(1-\pi_{i}))^2/\sum_{i=1}^{N}\pi_{i}(1-\pi_{i})-(\sum_{i=1}^{N}A_i(1-n/N))^2/n(1-n/N))]$. Now, it can be shown from the proof of Lemma S\ref{lem 1} that $\max_{1\leq i \leq N}|N\pi_{i}/n-1|\rightarrow0$ as $\nu\rightarrow\infty$. Therefore, using C$1$, we can show that $\lim_{\nu\rightarrow \infty}nN^{-2} \sum_{i=1}^{N}$ $A_i^2(\pi_{i}^{-1}-N/n)$=$0$ and $\lim_{\nu\rightarrow \infty}nN^{-2}[(\sum_{i=1}^{N}A_i(1-\pi_{i}))^2/\sum_{i=1}^{N}\pi_{i}(1-\pi_{i})-(\sum_{i=1}^{N}A_i(1-n/N))^2/n(1-n/N)]$=$0$, and consequently $\sigma^2_1$=$\sigma^2_2$.   This completes the proof of $(i)$ in Lemma \ref{thm 8}.
\par

 Next, consider the case of RHC sampling design and note that the limit in the expression of $\sigma^2_3$ exists in view of C$4$. Also, note that $\nabla g(\mu_0)\Sigma_2\nabla g(\mu_0)^T$ =$n\gamma (\overline{X}/N)\sum_{i=1}^{N}($ $A_i-\bar{A}X_i/\overline{X})^2/X_i$=$n\gamma ((\overline{X}/N)$ $\sum_{i=1}^N A_i^2/$ $X_i-\bar{A}^2)$, where $\bar{A}$=$\sum_{i=1}^N A_i/N$ and $\gamma$= $\sum_{i=1}^n N_i(N_i-1)/N(N-1)$. Thus we have $\sigma^2_3$=$\lim_{\nu\rightarrow\infty}$ $n\gamma((\overline{X}/N)\sum_{i=1}^N A_i^2/X_i-\bar{A}^2)$. 
\par

Next, note that the limit in the expression of $\sigma^2_4$ exists in view of C$3$. Substituting $\pi_i$=$nX_i/\sum_{i=1}^N X_i$ in $\nabla g(\mu_0)\Sigma_1\nabla g(\mu_0)^T$ for any HE$\pi$PS sampling design, we get $\sigma_4^2$=$\lim_{\nu\rightarrow \infty}nN^{-2}[\sum_{i=1}^{N} A_i^2( \sum_{i=1}^N X_i/n X_i-1)-(\sum_{i=1}^{N}A_i(1-n X_{ i}/\sum_{i=1}^N X_i))^2/\sum_{i=1}^{N} (n X_i/$ $\sum_{i=1}^N X_i)(1-n X_i/\sum_{i=1}^N X_i)]$=  $\lim_{\nu\rightarrow\infty}\big\{(1/N)\sum_{i=1}^N A_i^2\big((\overline{X}/X_i)-(n/N)\big) -\phi^{-1} \overline{X}^{-1}\big((n/N)$ $\times\sum_{i=1}^N A_i X_i/N-\overline{A}$ $\overline{X}\big)^2 \big\}$.  Further, we can show that $\sigma^2_4$=$\lim_{\nu\rightarrow\infty}((\overline{X}/N)\sum_{i=1}^N A_i^2/X_i-\bar{A}^2)$, when C$1$ and C$2$ hold, and C$0$ holds with $\lambda$=$0$.  It also follows from Lemma S\ref{lem 6} that $n\gamma\rightarrow 1$ as $\nu\rightarrow\infty$, when C$0$ holds with $\lambda$=$0$. Thus we have $\sigma^2_3$=$\sigma^2_4$=$\lim_{\nu\rightarrow\infty}((\overline{X}/N)\sum_{i=1}^N A_i^2/$ $X_i-\bar{A}^2)$.  This completes the proof of $(ii)$ in Lemma \ref{thm 8}.
\end{proof}

\begin{lemma}\label{lem 5}
Suppose that $C0$ through C$2$ hold. Then under SRSWOR, LMS sampling design and any HE$\pi$PS sampling design, we have
$$
(i)\quad u^{*}=\max_{i\in s}|Z_i|=o_p(\sqrt{n}),\text{ and }\quad(ii)\quad \sum_{i\in s}\pi_i^{-1} Z_i/\sum_{i\in s}\pi_i^{-1} Z_i^2=O_p(1/\sqrt{n})
$$
as $\nu\rightarrow\infty$, where $Z_i$=$X_i-\overline{X}$ for $i$=$1,\ldots,N$
\end{lemma}

\begin{proof}
Let $P(s)$ be any sampling design and $E_P$ be the expectation with respect to $P(s)$. Then, $E_P(u^*/\sqrt{n})\leq (\max_{1\leq i\leq N} X_i+\overline{X})/\sqrt{n}\leq \overline{X}(\max_{1\leq i\leq N} X_i/$ $\min_{1\leq i\leq N} X_i+1)/\sqrt{n}$=$o(1)$ as $\nu\rightarrow\infty$ since C$1$ and C$2$ hold. Therefore, $(i)$ holds under $P(s)$ by Markov inequality. Thus $(i)$ holds under SRSWOR, LMS sampling design and any HE$\pi$PS sampling design. 
\par
\vspace{.2cm}

Using similar arguments as in the first paragraph of the proof of Lemma S\ref{lem 2}, it can be shown that $\sqrt{n}(\sum_{i\in s} Z_i/N\pi_i-\overline{Z})$=$\sqrt{n}\sum_{i\in s} Z_i/N\pi_i$=$O_p(1)$ and $\sum_{i\in s} Z_i^2/N\pi_i-\sum_{i=1}^N Z_i^2/N$=$o_p(1)$ as $\nu\rightarrow\infty$ under a high entropy sampling design $P(s)$ satisfying \eqref{eq 1} in Lemma S\ref{lem 1}. Therefore, $1/(\sum_{i\in s}Z_i^2/N\pi_i)$ =$O_p(1)$ as $\nu\rightarrow\infty$ under $P(s)$ since $\sum_{i=1}^N Z_i^2/N$ is bounded away from $0$ as $\nu\rightarrow\infty$ by C$1$. Thus under $P(s)$, $\sum_{i\in s}\pi_i^{-1} Z_i/$ $\sum_{i\in s}\pi_i^{-1} Z_i^2$=$O_p(1/\sqrt{n})$ as $\nu\rightarrow\infty$.
\par
\vspace{.2cm}

It follows from Lemma S\ref{lem 1} that SRSWOR and LMS sampling design are high entropy sampling designs and satisfy \eqref{eq 1} Also, any HE$\pi$PS sampling design satisfies \eqref{eq 1} since C$2$ holds. Therefore, the result in $(ii)$ holds under the above-mentioned sampling designs.
\end{proof}
\section{Proofs of Remark $1$ and Theorems $2$, $3$, $6$ and $7$}
In this section, we give the proofs of Remark $1$ and Theorems $2$, $3$, $6$ and $7$ of the main text.
\begin{proof}[\textbf{Proof of Theorem $2$}] 
 Let us first consider a HE$\pi$PS sampling design. Then, it can be shown in the same way as in the $1^{st}$ paragraph of the proof of Theorem $1$ that $\sqrt{n}(\hat{\overline{h}}_{PEML}-\hat{\overline{h}}_{GREG})$=$o_p(1)$ for $d(i,s)$=$(N\pi_i)^{-1}$ under this sampling design. It can also be shown in the same way as in the $1^{st}$ paragraph of the proof of Theorem $1$ that if $\hat{\overline{h}}$ is one of $\hat{\overline{h}}_{HT}$, $\hat{\overline{h}}_{H}$, and $\hat{\overline{h}}_{GREG}$ and $\hat{\overline{h}}_{PEML}$ with $d(i,s)$=$(N\pi_i)^{-1}$, then $(4)$ in the proof of Theorem $1$ holds under the above-mentioned sampling design. Here, we recall from Table $2$ in the main text that the HT, the ratio and the product estimators coincide under any HE$\pi$PS sampling design. Further, the asymptotic MSE of $\sqrt{n}(g(\hat{\overline{h}})-g(\overline{h}))$ is $\nabla g(\mu_0) \Gamma_1$ $(\nabla g(\mu_0))^T$, where $\mu_0$=$\lim_{\nu\rightarrow\infty}\overline{h}$, $\Gamma_1$=$\lim_{\nu\rightarrow\infty}nN^{-2}\sum_{i=1}^{N}(\textbf{V}_i-\textbf{T}\pi_i)^T(\textbf{V}_i-\textbf{T}\pi_i)(\pi_i^{-1}-1)$, and $\textbf{V}_i$ in $\Gamma_1$ is $h_i$ or $h_i-\overline{h}$ or $h_i-\overline{h}-S_{xh}(X_i-\overline{X})/S_x^2$ if $\hat{\overline{h}}$ is $\hat{\overline{h}}_{HT}$ or $\hat{\overline{h}}_{H}$, or $\hat{\overline{h}}_{GREG}$ with $d(i,s)$=$(N\pi_i)^{-1}$, respectively. Now, since $\sqrt{n}(\hat{\overline{h}}_{PEML}-\hat{\overline{h}}_{GREG})$=$o_p(1)$ for $\nu\rightarrow\infty$ under any HE$\pi$PS sampling design, $g(\hat{\overline{h}}_{GREG})$ and $g(\hat{\overline{h}}_{PEML})$ have the same asymptotic distribution under this sampling design. Thus under any HE$\pi$PS sampling design, $g(\hat{\overline{h}}_{GREG})$ and $g(\hat{\overline{h}}_{PEML})$ with $d(i,s)$=$(N\pi_i)^{-1}$ form class $5$, $g(\hat{\overline{h}}_{HT})$ forms class $6$, and $g(\hat{\overline{h}}_{H})$ forms class $7$ in Table $2$ of the main text. This completes the proof of $(i)$ in Theorem $2$.
\par

Let us  now  consider the RHC sampling design. We can show from  $(ii)$ in  Lemma S\ref{lem 3} that $\sqrt{n}(\hat{\overline{h}}-\overline{h})\xrightarrow{\mathcal{L}}N(0,\Gamma)$ as $\nu\rightarrow\infty$ for some p.d. matrix $\Gamma$, when $\hat{\overline{h}}$ is either $\hat{\overline{h}}_{RHC}$ or $\hat{\overline{h}}_{GREG}$ with $d(i,s)$=$G_i/NX_i$ under RHC sampling design. Further, $\sqrt{n}(\hat{\overline{h}}_{PEML}-\hat{\overline{h}}_{GREG})$=$o_p(1)$ as $\nu\rightarrow\infty$ for $d(i,s)$=$G_i/NX_i$ under RHC sampling design since  C$2$ holds, and $S^2_x$ is bounded away from $0$ as $\nu\rightarrow\infty$ (see A$2.2$ of Appendix $2$ in \cite{MR1707846}). Therefore, if $\hat{\overline{h}}$ is one of $\hat{\overline{h}}_{RHC}$, and $\hat{\overline{h}}_{GREG}$ and $\hat{\overline{h}}_{PEML}$ with $d(i,s)$=$G_i/N X_i$,  then we have
\begin{equation}\label{eq 7} 
\sqrt{n}(g(\hat{\overline{h}})-g(\overline{h}))\xrightarrow{\mathcal{L}}N(0,\Delta^2)\text{ as }\nu\rightarrow\infty
\end{equation}
for some $\Delta^2>0$ by the delta method and the condition $\nabla g(\mu_0)\neq 0$ at $\mu_0$=$\lim_{\nu\rightarrow\infty}\overline{h}$.  Moreover, it follows from the proof of  $(ii)$ in  Lemma S\ref{lem 3} that  $\Delta^2$= $\nabla g(\mu_0) \Gamma_2 (\nabla g(\mu_0))^T$, where $\Gamma_2$=$\lim_{\nu\rightarrow\infty}n\gamma\overline{X}N^{-1}\sum_{i=1}^{N}(\textbf{V}_i-X_i\overline{\textbf{V}}/\overline{X})^T (\textbf{V}_i-X_i\overline{\textbf{V}}/\overline{X})/ X_i$. It further follows from Table \ref{table B} in this supplement that $\textbf{V}_i$ in $\Gamma_2$ is $h_i$ if $\hat{\overline{h}}$ is $\hat{\overline{h}}_{RHC}$. Also, $\textbf{V}_i$ in $\Gamma_2$ is $h_i-\overline{h}-S_{xh}(X_i-\overline{X})/S_x^2$ if $\hat{\overline{h}}$ is $\hat{\overline{h}}_{GREG}$ with $d(i,s)$=$G_i/NX_i$.  Now, $g(\hat{\overline{h}}_{GREG})$ and $g(\hat{\overline{h}}_{PEML})$ have the same asymptotic distribution under RHC sampling design since $\sqrt{n}(\hat{\overline{h}}_{PEML}-\hat{\overline{h}}_{GREG})$=$o_p(1)$ for $\nu\rightarrow\infty$ under this sampling design as pointed out earlier in this paragraph. Thus $g(\hat{\overline{h}}_{GREG})$ and $g(\hat{\overline{h}}_{PEML})$  with $d(i,s)$=$G_i/NX_i$ under RHC sampling design form class $8$, and $g(\hat{\overline{h}}_{RHC})$ forms class $9$ in Table $2$ of the main article. This completes the proof of $(ii)$ in Theorem $2$. 
\end{proof}
\begin{proof}[\textbf{Proof of Remark $1$}]
It follows from \textsl{(ii)} in Lemma S\ref{thm 8} that in the case of $\lambda$=$0$, 
\begin{equation}\label{eq 5}
\sigma^2_{3}=\sigma^2_{4}=\lim_{\nu\rightarrow\infty}((\overline{X}/N)\sum_{i=1}^N A_i^2/X_i-\bar{A}^2),
\end{equation} 
where $\sigma^2_3$ and $\sigma^2_4$ are as defined in the statement of Lemma S\ref{thm 8}, and $A_i$=$ \nabla g(\mu_0) \textbf{V}_i^T$ for different choices of $\textbf{V}_i$ mentioned in the proof of Theorem $2$ above. Thus $g(\hat{\overline{h}}_{GREG})$ with $d(i,s)$=$(N\pi_i)^{-1}$ under any HE$\pi$PS sampling design, and with $d(i,s)$=$G_i/NX_i$ under RHC sampling design have the same asymptotic MSE. Therefore, class $8$ is merged with class $5$ in Table $2$ of the main text. Further, \eqref{eq 5} implies that $g(\hat{\overline{h}}_{HT})$ under any HE$\pi$PS sampling design and $g(\hat{\overline{h}}_{RHC})$ have the same asymptotic MSE. Therefore, class $9$ is merged with class $6$ in Table $2$ of the main text. This completes the proof of Remark $1$. 
\end{proof}
\begin{proof}[\textbf{Proof of Theorem $3$}] Recall the expression of $A_i$'s from the proofs of Theorem $1$ and Remark $1$.  Note that $\lim_{\nu\rightarrow\infty}\sum (A_i-\bar{A})^2/N$=$\lim_{\nu\rightarrow\infty}\sum (B_i-\bar{B})^2/N$, $\lim_{\nu\rightarrow\infty} n\gamma  \big((\overline{X}/N)$ $\times\sum_{i=1}^N A_i^2/X_i-\bar{A}^2\big)$=$\lim_{\nu\rightarrow\infty} n\gamma  \big((\overline{X}/N)\sum_{i=1}^N B_i^2/X_i$ $-\bar{B}^2\big)$  and $\lim_{\nu\rightarrow\infty}\big\{(1/N)\sum_{i=1}^N A_i^2\times$ $\big((\overline{X}/X_i)-(n/N)\big) -\phi^{-1} \overline{X}^{-1}\big((n/N)\sum_{i=1}^N A_i X_i/N-\bar{A} \overline{X}\big)^2 \big\}$=$\lim_{\nu\rightarrow\infty}\big\{(1/N)\sum_{i=1}^N B_i^2\times$ $\big((\overline{X}/X_i)-(n/N)\big) -\phi^{-1} \overline{X}^{-1} \big((n/N)\sum_{i=1}^N B_i X_i/N-\bar{B} \overline{X}\big)^2 \big\}$  for $B_i$=$\nabla g(\overline{h})\textbf{V}_i^T$ and $\textbf{V}_i$ as in Table \ref{table B} in this supplement since $\nabla g(\overline{h})\rightarrow \nabla g(\mu_0)$ as $\nu\rightarrow\infty$.  Here, $\phi$=$\overline{X}-(n/N)\sum_{i=1}^N X_i^2/N\overline{X}$.  Then, from  Lemma S\ref{thm 8} and the expressions of asymptotic MSEs of $\sqrt{n}(g(\hat{\overline{h}})-g(\overline{h}))$ discussed  in the proofs of Theorems $1$ and $2$ , the results in Table $3$ of the main text follow.  This completes the proof of Theorem $3$.
\end{proof}
\begin{proof}[\textbf{Proof of Theorem $6$}] 
 Using similar arguments as in the $1^{st}$ paragraph of the proof of Theorem $4$, we can say that under SRSWOR and LMS sampling design, conclusions of  Theorems $1$ and $3$  hold \textit{a.s.} $[\mathbb{P}]$ for $d$=$1$, $p$=$2$, $h(y)$=$(y,y^2)$ and $g(s_1,s_2)$=$s_2-s_1^2$  in the same way as conclusions of Theorems $1$ and $3$ hold \textit{a.s.} $[\mathbb{P}]$ for $d$=$p$=$1$, $h(y)$=$y$ and $g(s)$=$s$ in the $1^{st}$ paragraph of the proof of Theorem $4$.  Note that $W_i$=$Y_i^2-2Y_i\overline{Y}$ for the above choices of $h$ and $g$. Further, it follows from SLLN and the condition $E_{\mathbb{P}}(\epsilon_i)^8<\infty$ that the $\Delta^2_i$'s in Table $3$ in the main text can be expressed in terms of superpopulation moments of $(Y_i,X_i)$ \textit{a.s.} $[\mathbb{P}]$. Note that $\Delta^2_2-\Delta^2_1$=$cov_{\mathbb{P}}^2(\tilde{W_i},X_i)$ \textit{a.s.} $[\mathbb{P}]$, where $\tilde{W}_i$=$Y_i^2-2Y_i E_{\mathbb{P}}(Y_i)$. Then, $\Delta^2_1< \Delta^2_2$ \textit{a.s.} $[\mathbb{P}]$.  This completes the proof of $(i)$ in Theorem $6$.
\par

 Next consider the case of  $0\leq \lambda< E_{\mathbb{P}}(X_i)/b$.  Using the same line of arguments as in the $2^{nd}$ paragraph of the proof of Theorem $4$, it can be shown that under RHC and any HE$\pi$PS sampling designs, conclusions of  Theorems $2$ and $3$  hold \textit{a.s.} $[\mathbb{P}]$ for $d$=$1$, $p$=$2$, $h(y)$=$(y,y^2)$ and $g(s_1,s_2)$=$s_2-s_1^2$  in the same way as conclusions of Theorems $2$ and $3$ hold \textit{a.s.} $[\mathbb{P}]$ for $d$=$p$=$1$, $h(y)$=$y$ and $g(s)$=$s$ in the $2^{nd}$ paragraph of the proof of Theorem $4$.  Note that  $\Delta^2_7-\Delta^2_5$=$\big\{\mu_1^2 cov_{\mathbb{P}}(\tilde{W}_i,X_i)\big(cov_{\mathbb{P}}(\tilde{W}_i,X_i)cov_{\mathbb{P}}(X_i,$ $1/X_i)-2cov_{\mathbb{P}}(\tilde{W}_i,1/X_i)\big)\big\}-\lambda^2 cov^2_{\mathbb{P}}(\tilde{W}_i,X_i)/\chi\mu_1-\lambda cov^2_{\mathbb{P}}(\tilde{W}_i,X_i)\leq \big\{\mu_1^2\times$ $cov_{\mathbb{P}}(\tilde{W}_i,X_i)\big(cov_{\mathbb{P}}(\tilde{W}_i,X_i)cov_{\mathbb{P}}(X_i,1/X_i)-2cov_{\mathbb{P}}(\tilde{W}_i,1/X_i)\big)\big\}$ \textit{a.s.} $[\mathbb{P}]$ because $\chi>0$.  Recall from C$6$ that $\xi$=$\mu_3-\mu_2\mu_1$ and $\mu_j$=$E_{\mathbb{P}}(X_i)^j$ for $j$=$-1,1,2,3$. Then, from the linear model set up,  we have $\big\{\mu_1^2 cov_{\mathbb{P}}(\tilde{W}_i,X_i)\times$ $\big(cov_{\mathbb{P}}(\tilde{W}_i,X_i)cov_{\mathbb{P}}(X_i,$ $1/X_i)-2cov_{\mathbb{P}}(\tilde{W}_i,1/X_i)\big)\big\}$=$(\beta^2\mu_1)^2(\xi-2\mu_1)((\xi+2\mu_1)\zeta_1-2\zeta_2)$.  Here, $\zeta_1$=$1-\mu_1\mu_{-1}$ and $\zeta_2$=$\mu_1-\mu_2\mu_{-1}$. Note that $(\xi+2\mu_1)\zeta_1-2\zeta_2$=$\xi\zeta_1+2\mu_{-1}$ and $\zeta_1<0$. Therefore,  $\big\{\mu_1^2 cov_{\mathbb{P}}(\tilde{W}_i,X_i) \big(cov_{\mathbb{P}}(\tilde{W}_i,X_i)cov_{\mathbb{P}}(X_i,$ $1/X_i)-2cov_{\mathbb{P}}(\tilde{W}_i,1/X_i)\big)\big\}<0$ if $\xi>2\max\{\mu_1,\mu_{-1}/(\mu_1\mu_{-1}-1)\}$. Hence, $\Delta^2_7-\Delta^2_5<0$ \textit{a.s.} $[\mathbb{P}]$. This completes the proof of $(ii)$ in Theorem $6$. 
\end{proof}
\begin{proof}[\textbf{Proof of Theorem $7$}] Using the same line of arguments as in the $1^{st}$ paragraph of the proof of Theorem $4$, it can be shown that under SRSWOR and LMS sampling design, conclusions of Theorems $1$ and $3$ hold \textit{a.s.} $[\mathbb{P}]$ for $d$=$2$, $p$=$5$, $h(z_1,z_2)$=$(z_1,z_2,z_1^2$ $,z_2^2,$ $z_1z_2)$ and $g(s_1,s_2,s_3,s_4,s_5)$=$(s_5-s_1s_2)/((s_3-s_1^2)(s_4-s_2^2))^{1/2}$ in the case of the correlation coefficient between $z_1$ and $z_2$, and for $d$=$2$, $p$=$4$, $h(z_1,z_2)$=$(z_1, z_2, z^2_2, z_1 z_2)$ and $g(s_1,s_2,s_3, s_4)$= $(s_4-s_1 s_2)/(s_3-s^2_2)$ in the case of the regression coefficient of $z_1$ on $z_2$  in the same way as conclusions of Theorems $1$ and $3$ hold \textit{a.s.} $[\mathbb{P}]$ for $d$=$p$=$1$, $h(y)$=$y$ and $g(s)$=$s$ in the case of the mean of $y$ in the $1^{st}$ paragraph of the proof of Theorem $4$.  Further, if C$0$ holds with  $0\leq \lambda< E_{\mathbb{P}}(X_i)/b$,  then using similar arguments as in the $2^{nd}$ paragraph of the proof of Theorem $4$, it can also be shown that under RHC and any HE$\pi$PS sampling designs, conclusions of  Theorems $2$ and $3$  hold \textit{a.s.} $[\mathbb{P}]$ for  $d$=$2$, $p$=$5$, $h(z_1,z_2)$=$(z_1,z_2,z_1^2$ $,z_2^2,$ $z_1z_2)$ and $g(s_1,s_2,s_3,s_4,s_5)$=$(s_5-s_1s_2)/((s_3-s_1^2)(s_4-s_2^2))^{1/2}$ in the case of the correlation coefficient between $z_1$ and $z_2$, and for $d$=$2$, $p$=$4$, $h(z_1,z_2)$=$(z_1, z_2, z^2_2, z_1 z_2)$ and $g(s_1,s_2,s_3, s_4)$=$(s_4-s_1 s_2)/(s_3-s^2_2)$ in the case of the regression coefficient of $z_1$ on $z_2$ in the same way as conclusions of Theorems $2$ and $3$ hold \textit{a.s.} $[\mathbb{P}]$ for $d$=$p$=$1$, $h(y)$=$y$ and $g(s)$=$s$ in the case of the mean of $y$ in the $2^{nd}$ paragraph of the proof of Theorem $4$.  Note that $W_i$=$R_{12}[(\overline{Z}_1/S_1^2-\overline{Z}_2/S_{12})Z_{1i}+(\overline{Z}_2/S_2^2-\overline{Z}_1/S_{12})Z_{2i}-Z_{1i}^2/2S_1^2-Z_{2i}^2/2S_2^2+Z_{1i}Z_{2i}/S_{12}]$ for the correlation coefficient, and $W_i$=$(1/S_2^2)[-\overline{Z}_2 Z_{1i}-(\overline{Z}_1-2S_{12}\overline{Z}_2/S_2^2)Z_{2i}-S_{12}Z_{2i}^2/S_2^2+Z_{1i}Z_{2i}]$ for the regression coefficient. Here, $\overline{Z}_1$=$\sum_{i=1}^N Z_{1i}/N$, $\overline{Z}_2$=$\sum_{i=1}^N Z_{2i}/N$, $S_1^2$=$\sum_{i=1}^N Z_{1i}^2$ $/N-\overline{Z}_1^2$, $S_2^2$=$\sum_{i=1}^N Z_{2i}^2/N-\overline{Z}_2^2$, $S_{12}$=$\sum_{i=1}^N Z_{1i}$ $Z_{2i}/N-\overline{Z}_1\overline{Z}_2$ and $R_{12}$=$S_{12}/S_1S_2$. Also, note that since $E_{\mathbb{P}}||\epsilon_i||^8<\infty$, the $\Delta^2_i$'s in Table $3$ in the main text can be expressed in terms of superpopulation moments of $(h(Z_{1i},Z_{2i}),X_i)$ \textit{a.s.} $[\mathbb{P}]$ for both the parameters by SLLN. Further, for the above parameters, we have $\Delta^2_2-\Delta^2_1$=$cov_{\mathbb{P}}^2(\tilde{W_i},X_i)>0$ and  $\Delta^2_7-\Delta^2_5$=$\big\{\mu_1^2 cov_{\mathbb{P}}(\tilde{W}_i,X_i)\big(cov_{\mathbb{P}}(\tilde{W}_i,X_i)cov_{\mathbb{P}}(X_i,$ $1/X_i)-2cov_{\mathbb{P}}(\tilde{W}_i,1/X_i)\big)\big\}-\lambda^2 cov^2_{\mathbb{P}}(\tilde{W}_i,X_i)/\chi\mu_1-\lambda cov^2_{\mathbb{P}}(\tilde{W}_i,X_i)\leq \big\{\mu_1^2 cov_{\mathbb{P}}(\tilde{W}_i,X_i)\big(cov_{\mathbb{P}}(\tilde{W}_i,X_i)\times$ $cov_{\mathbb{P}}(X_i,1/X_i)-2cov_{\mathbb{P}}(\tilde{W}_i,1/X_i)\big)\big\}$  \textit{a.s.} $[\mathbb{P}]$, where $\tilde{W}_i$ is the same as $W_i$ with all finite population moments in the expression of $W_i$ replaced by their corresponding superpopulation moments. Also, from the linear model set up, we have  $\big\{\mu_1^2 cov_{\mathbb{P}}(\tilde{W}_i,X_i)\big(cov_{\mathbb{P}}(\tilde{W}_i,X_i)cov_{\mathbb{P}}(X_i,$ $1/X_i)-2cov_{\mathbb{P}}(\tilde{W}_i,1/X_i)\big)\big\}$=$K(\xi-2\mu_1)((\xi+2\mu_1)\zeta_1-2\zeta_2)$  for some constant $K> 0$ in the case of the correlation coefficient, and  $\big\{\mu_1^2 cov_{\mathbb{P}}(\tilde{W}_i,X_i)\times$ $\big(cov_{\mathbb{P}}(\tilde{W}_i,X_i)cov_{\mathbb{P}}(X_i,1/X_i)-2cov_{\mathbb{P}}(\tilde{W}_i,1/X_i)\big)\big\}$ =$K^{\prime}(\xi-2\mu_1)((\xi+2\mu_1)\zeta_1-2\zeta_2)$  for some constant $K^{\prime}> 0$ in the case of the regression coefficient. Thus proofs of both the parts of the theorem follow in the same way as the proof of Theorem $6$.
\end{proof}
\section{Comparison of estimators with their bias-corrected versions}\label{sec S2}
In this section, we empirically compare the biased estimators considered in Table $5$ in Section $4$ of the main text with their bias-corrected versions based on both synthetic and real data used in Section $4$. Following the idea in \cite{stefan2022jackknife}, we consider the bias-corrected jackknife estimator corresponding to each of the biased estimators considered in Table $5$ of the main article. For the mean, we consider the bias-corrected jackknife estimators corresponding to the GREG and the PEML estimators under each of SRSWOR, RS and RHC sampling designs, and the H\'ajek estimator under RS sampling design. On the other hand, for each of the variance, the correlation coefficient and the regression coefficient, we consider the bias-corrected jackknife estimators corresponding to the estimators that are obtained by plugging in the H\'ajek and the PEML estimators under each of SRSWOR and RS sampling design, and the PEML estimator under RHC sampling design.
\par

Suppose that $s$ is a sample of size $n$ drawn using one of the sampling designs given in Table $5$ of the main text. Further, suppose that $s_{-i}$ is the subset of $s$, which excludes the $i^{th}$ unit for any given $i\in s$. Now, for any $i\in s$, let us denote the estimator $g(\hat{\overline{h}})$ constructed based on $s_{-i}$ by $g(\hat{\overline{h}}_{-i})$. Then, we compute the bias-corrected jackknife estimator of $g(\overline{h})$ corresponding to $g(\hat{\overline{h}})$ as $n g(\hat{\overline{h}})-(n-1) \sum_{i\in s} g(\hat{\overline{h}}_{-i})/n$ (cf. \cite{stefan2022jackknife}). Recall from Section $4$ in the main article that we draw $I$=$1000$ samples each of sizes $n$=$75$, $100$ and $125$ from some synthetic as well as real datasets using sampling designs mentioned in Table $5$ and compute MSEs of the estimators considered in Table $5$ based on these samples. Here, we compute MSEs of the above-mentioned bias-corrected jackknife estimators using the same procedure and compare them with the original biased estimators in terms of their MSEs. We observe from the above analyses that for all the parameters considered in Section $4$ of the main text, the bias-corrected jackknife estimators become worse than the original biased estimators in the cases of both the synthetic and the real data (see Tables \ref{table 1} through \ref{table 5} and \ref{table 11} through \ref{table 20} in Sections \ref{sec S3} and \ref{sec S4} below). Despite reducing the biases of the original biased estimators, bias-correction increases the variances of these estimators significantly. This is the reason why the bias-corrected jackknife estimators have larger MSEs than the original biased estimators in the cases of both the synthetic and the real data. 
\section{ Analysis based on synthetic data}\label{sec S3}
The results obtained from the analysis carried out in Section $4.1$ of the main paper and Section \ref{sec S2} in this supplement are summarized in these sections.  Here, we provide some tables that were mentioned in these sections.  Tables \ref{table 1} through \ref{table 5} contain relative efficiencies of estimators for the mean, the variance, the correlation coefficient and the regression coefficient in the population.  Tables \ref{table 6} through \ref{table 10} contain the average and the standard deviation of lengths of asymptotically $95\%$ CIs of the above parameters.  
  
\begin{table}[h]
\caption{Relative efficiencies of estimators for mean of $y$.}
\label{table 1}
\begin{threeparttable}[b]
\centering
\renewcommand{\arraystretch}{0.9}
\begin{tabular}{|c|c|c|c|} 
\hline
\backslashbox{Relative efficiency}{Sample size}
&  $n$=$75$ & $n$=$100$ & $n$=$125$  \\ 
 \hline
 RE($\hat{\overline{Y}}_{PEML}$, SRSWOR $|$ $\hat{\overline{Y}}_{GREG}$, SRSWOR)& $1.049985$& $1.020252$& $1.035038$\\
 RE($\hat{\overline{Y}}_{PEML}$, SRSWOR $|$ $\hat{\overline{Y}}_{H}$, RS)& $4.870516$& $5.370899$& $4.987635$\\
 RE($\hat{\overline{Y}}_{PEML}$, SRSWOR $|$ $\hat{\overline{Y}}_{HT}$, RS )& $2.026734$& $2.061607$& $2.027386$\\
 RE($\hat{\overline{Y}}_{PEML}$, SRSWOR $|$ $\hat{\overline{Y}}_{PEML}$, RS)&  $1.144439$& $1.124697$& $1.170224$\\
 RE($\hat{\overline{Y}}_{PEML}$, SRSWOR $|$ $\hat{\overline{Y}}_{GREG}$, RS)& $1.144455$& $1.124975$& $1.170267$\\ 
 RE($\hat{\overline{Y}}_{PEML}$, SRSWOR$|$ $\hat{\overline{Y}}_{RHC}$, RHC )&  $2.022378$& $1.978623$& $2.143015$\\
 RE($\hat{\overline{Y}}_{PEML}$, SRSWOR $|$ $\hat{\overline{Y}}_{PEML}$, RHC)&   $1.089837$& $1.030332$& $1.094067$\\
 RE($\hat{\overline{Y}}_{PEML}$, SRSWOR $|$ $\hat{\overline{Y}}_{GREG}$, RHC)& $1.089853$& $1.030587$&$1.094108$\\
\hline
 RE($\hat{\overline{Y}}_{PEML}$, SRSWOR $|$ \tnote{1} $\hat{\overline{Y}}_{BCPEML}$, SRSWOR)& $1.050461$& $1.021275$& $1.038282$\\
RE($\hat{\overline{Y}}_{GREG}$, SRSWOR $|$ \tnote{1} $\hat{\overline{Y}}_{BCGREG}$, SRSWOR)& $1.002649$& $1.003156$& $1.005397$\\
RE($\hat{\overline{Y}}_{H}$, RS $|$ \tnote{1} $\hat{\overline{Y}}_{BCH}$, RS)&$1.036379$& $1.006945$& $1.12841$\\
RE($\hat{\overline{Y}}_{PEML}$, RS $|$ \tnote{1} $\hat{\overline{Y}}_{BCPEML}$, RS)& $1.016953$& $1.013402$& $1.011762$\\
RE($\hat{\overline{Y}}_{GREG}$, RS $|$ \tnote{1} $\hat{\overline{Y}}_{BCGREG}$, RS)& $1.016692$& $1.011597$& $1.011493$\\
RE($\hat{\overline{Y}}_{PEML}$, RHC $|$ \tnote{1} $\hat{\overline{Y}}_{BCPEML}$, RHC)&$1.01914$ & $1.02292$& $1.024689$ \\
RE($\hat{\overline{Y}}_{GREG}$, RHC $|$ \tnote{1} $\hat{\overline{Y}}_{BCGREG}$, RHC)& $1.011583$& $1.052311$& $1.023058$\\
\hline
\end{tabular}
\begin{tablenotes}
\item[1] BCPEML=Bias-corrected PEML estimator, BCH=Bias-corrected H\'ajek estimator, and  BCGREG=Bias-corrected GREG estimator. 
\end{tablenotes}
\end{threeparttable}

\footnotetext[1]{}
\end{table}

\begin{table}[h!]
\renewcommand{\arraystretch}{0.7}
\caption{Relative efficiencies of estimators for variance of $y$. Recall from Table $4$ in Section $2$ that for variance of $y$, $h(y)$=$(y^2,y)$ and $g(s_1,s_2)$=$s_1-s_2^2$.}
\label{table 2}
\begin{threeparttable}[b]
\centering
\begin{tabular}{|c|c|c|c|} 
\hline
\backslashbox{Relative efficiency}{Sample size}
&  $n$=$75$ & $n$=$100$ & $n$=$125$  \\ 
 \hline
 RE($g(\hat{\overline{h}}_{PEML})$, SRSWOR $|$ $g(\hat{\overline{h}}_{H})$, SRSWOR)  &  $1.0926$& $1.0848$&$1.0419$\\ 
 RE($g(\hat{\overline{h}}_{PEML})$, SRSWOR $|$ $g(\hat{\overline{h}}_{H})$, RS) & $1.0367$& $1.0435$& $1.0226$\\
 RE($g(\hat{\overline{h}}_{PEML})$, SRSWOR $|$ $g(\hat{\overline{h}}_{PEML})$, RS) &  $ 1.15067$& $1.136$& $1.1635$\\
 RE($g(\hat{\overline{h}}_{PEML})$, SRSWOR $|$ $g(\hat{\overline{h}}_{PEML})$, RHC)  &  $1.141$& $1.1849$& $1.1631$\\
 \hline
RE($g(\hat{\overline{h}}_{PEML})$, SRSWOR $|$ \tnote{2} \  BC $g(\hat{\overline{h}}_{PEML})$, SRSWOR)  &  $1.0208$& $1.01$& $1.0669$\\
 RE($g(\hat{\overline{h}}_{H})$, SRSWOR $|$ \tnote{2} \  BC $g(\hat{\overline{h}}_{H})$, SRSWOR)  &$ 38.642$&$ 50.009$&$ 65.398$\\
 RE($g(\hat{\overline{h}}_{H})$, RS $|$ \tnote{2} \  BC $g(\hat{\overline{h}}_{H})$, RS)  &  $ 1.0029$&$ 1.0117$&$ 1.074$\\ 
 RE($g(\hat{\overline{h}}_{PEML})$, RS $|$ \tnote{2} \  BC $g(\hat{\overline{h}}_{PEML})$, RS)  & $ 1.0112$&$ 1.023$&$ 1.0377
$ \\  
 RE($g(\hat{\overline{h}}_{PEML})$, RHC $|$ \tnote{2} \  BC $g(\hat{\overline{h}}_{PEML})$, RHC)  & $1.0141$&$ 1.015$&$ 1.0126$  \\ 
\hline
\end{tabular}
\begin{tablenotes}
\item[2] BC=Bias-corrected. 
\end{tablenotes}
\end{threeparttable}
\end{table}

\begin{table}[h]
\renewcommand{\arraystretch}{0.7}
\caption{Relative efficiencies of estimators for correlation coefficient between $z_1$ and $z_2$. Recall from Table $4$ in Section $2$ that for correlation coefficient between $z_1$ and $z_2$, $h(z_1,z_2)$=$(z_1,z_2,z_1^2,z_2^2, z_1 z_2)$ and $g(s_1,s_2,s_3,s_4,s_5)$=$(s_5-s_1s_2)/((s_3-s_1^2)(s_4-s_2^2))^{1/2}$.}
\label{table 3}
\begin{threeparttable}
\centering
\begin{tabular}{|c|c|c|c|}
\hline 
\backslashbox{Relative efficiency}{Sample size}
&  $n$=$75$ & $n$=$100$ & $n$=$125$  \\
\hline
 RE($g(\hat{\overline{h}}_{PEML})$, SRSWOR $|$ $g(\hat{\overline{h}}_{H})$, SRSWOR)& $1.0304$& $1.0274$& $1.0385$ \\
 RE($g(\hat{\overline{h}}_{PEML})$, SRSWOR $|$ $g(\hat{\overline{h}}_{H})$, RS)& $1.0307$& $1.0838$& $1.0515$ \\
 RE($g(\hat{\overline{h}}_{PEML})$, SRSWOR $|$ $g(\hat{\overline{h}}_{PEML})$, RS)&  $1.0573$& $1.1862$& $1.1081$ \\ 
 RE($g(\hat{\overline{h}}_{PEML})$, SRSWOR $|$ $g(\hat{\overline{h}}_{PEML})$, RHC)& $1.0847$& $1.1459$& $1.0911$ \\
 \hline
 RE($g(\hat{\overline{h}}_{PEML})$, SRSWOR $|$ \tnote{2} \  BC $g(\hat{\overline{h}}_{PEML})$, SRSWOR)  &  $89.989$& $95.299$& $123.89$\\
 RE($g(\hat{\overline{h}}_{H})$, SRSWOR $|$ \tnote{2} \  BC $g(\hat{\overline{h}}_{H})$, SRSWOR)  &$90.407$&$  96.79$&$ 141.989$\\
 RE($g(\hat{\overline{h}}_{H})$, RS $|$ \tnote{2} \  BC $g(\hat{\overline{h}}_{H})$, RS)  &  $90.037$&$ 102.914$&$ 152.993$\\ 
 RE($g(\hat{\overline{h}}_{PEML})$, RS $|$ \tnote{2} \  BC $g(\hat{\overline{h}}_{PEML})$, RS)  & $95.68$&$  98.758$&$ 158.832$ \\  
 RE($g(\hat{\overline{h}}_{PEML})$, RHC $|$ \tnote{2} \  BC $g(\hat{\overline{h}}_{PEML})$, RHC)  & $86.27$&$ 120.582$&$ 125.374$  \\
 \hline 
\end{tabular}
\end{threeparttable}
\end{table}

\begin{table}[h]
\renewcommand{\arraystretch}{0.7}
\caption{Relative efficiencies of estimators for regression coefficient of $z_1$ on $z_2$. Recall from Table $4$ in Section $2$ that for regression coefficient of $z_1$ on $z_2$, $h(z_1,z_2)$=$(z_1, z_2, z^2_2, z_1 z_2)$ and $g(s_1,s_2,s_3, s_4)$=$(s_4-s_1 s_2)/(s_3-s^2_2)$.}
\label{table 4}
\begin{threeparttable}
\centering
\begin{tabular}{|c|c|c|c|}
\hline 
\backslashbox{Relative efficiency}{Sample size}
&  $n$=$75$ & $n$=$100$ & $n$=$125$  \\
\hline
 RE($g(\hat{\overline{h}}_{PEML})$, SRSWOR $|$ $g(\hat{\overline{h}}_{H})$, SRSWOR)& $1.0389$& $1.0473$& $1.0218$ \\ 
 RE($g(\hat{\overline{h}}_{PEML})$, SRSWOR $|$ $g(\hat{\overline{h}}_{H})$, RS)& $1.0589$& $1.0829$& $1.0827$ \\
 RE($g(\hat{\overline{h}}_{PEML})$, SRSWOR $|$ $g(\hat{\overline{h}}_{PEML})$, RS)&  $1.1219$& $1.1334$& $1.2137$ \\ 
 RE($g(\hat{\overline{h}}_{PEML})$, SRSWOR $|$ $g(\hat{\overline{h}}_{PEML})$, RHC)& $1.2037$& $1.1307$& $1.1399$ \\
\hline
 RE($g(\hat{\overline{h}}_{PEML})$, SRSWOR $|$ \tnote{2} \  BC $g(\hat{\overline{h}}_{PEML})$, SRSWOR)  &  $80.64$& $91.707$& $124.476$\\ 
 RE($g(\hat{\overline{h}}_{H})$, SRSWOR $|$ \tnote{2} \  BC $g(\hat{\overline{h}}_{H})$, SRSWOR)  &$79.298$& $ 89.105$&$ 123.042$\\
 RE($g(\hat{\overline{h}}_{RS})$, RS $|$ \tnote{2} \  BC $g(\hat{\overline{h}}_{H})$, RS)  &  $85.97$&$  96.22$&$ 135.449$\\
 RE($g(\hat{\overline{h}}_{PEML})$, RS $|$ \tnote{2} \  BC $g(\hat{\overline{h}}_{PEML})$, RS)  & $ 83.331$&$  97.583$&$ 125.657$ \\  
 RE($g(\hat{\overline{h}}_{PEML})$, RHC $|$ \tnote{2} \  BC $g(\hat{\overline{h}}_{PEML})$, RHC)  & $75.343$&$ 112.619$&$ 115.594$  \\
 \hline
\end{tabular}
\end{threeparttable}
\end{table}

\begin{table}[h]
\renewcommand{\arraystretch}{0.7}
\caption{Relative efficiencies of estimators for regression coefficient of $z_2$ on $z_1$. Recall from Table $4$ in Section $2$ that for regression coefficient of $z_2$ on $z_1$, $h(z_1,z_2)$=$(z_2, z_1, z^2_1, z_1 z_2)$ and $g(s_1,s_2,s_3, s_4)$=$(s_4-s_1 s_2)/(s_3-s^2_2)$.}
\label{table 5}
\begin{threeparttable}
\centering
\begin{tabular}{|c|c|c|c|}
\hline 
\backslashbox{Relative efficiency}{Sample size}
&  $n$=$75$ & $n$=$100$ & $n$=$125$  \\
\hline
 RE($g(\hat{\overline{h}}_{PEML})$, SRSWOR $|$ $g(\hat{\overline{h}}_{H})$, SRSWOR)& $1.0498$& $1.04$& $1.0301$ \\ 
 RE($g(\hat{\overline{h}}_{PEML})$, SRSWOR $|$ $g(\hat{\overline{h}}_{H})$, RS)& $1.0655$& $1.0652$& $1.0548$ \\
 RE($g(\hat{\overline{h}}_{PEML})$, SRSWOR $|$ $g(\hat{\overline{h}}_{PEML})$, RS)&  $ 1.1073$& $1.1153$& $1.1135$ \\ 
 RE($g(\hat{\overline{h}}_{PEML})$, SRSWOR $|$ $g(\hat{\overline{h}}_{PEML})$, RHC)& $ 1.0762$& $1.0905$& $1.1108$ \\
 \hline
  RE($g(\hat{\overline{h}}_{PEML})$, SRSWOR $|$ \tnote{2} \  BC $g(\hat{\overline{h}}_{PEML})$, SRSWOR)  &  $72.061$& $105.389$& $111.124$\\ 
  RE($g(\hat{\overline{h}}_{H})$, SRSWOR $|$ \tnote{2} \  BC $g(\hat{\overline{h}}_{H})$, SRSWOR)  &$69.114$& $108.837$&$ 118.675$\\
 RE($g(\hat{\overline{h}}_{H})$, RS $|$ \tnote{2} \  BC $g(\hat{\overline{h}}_{H})$, RS)  &  $ 69.16$& $115.113$&$ 144.811$\\
 RE($g(\hat{\overline{h}}_{PEML})$, RS $|$ \tnote{2} \  BC $g(\hat{\overline{h}}_{PEML})$, RS)  & $72.448$& $127.387$& $131.558$
 \\  
 RE($g(\hat{\overline{h}}_{PEML})$, RHC $|$ \tnote{2} \  BC $g(\hat{\overline{h}}_{PEML})$, RHC)  & $90.132$& $104.121$&$ 148.139$  \\
 \hline
\end{tabular}
\end{threeparttable}
\end{table}

\begin{table}[h!]
\caption{Average and standard deviation of lengths of asymptotically $95\%$ CIs for mean of $y$.}
\label{table 6}
\begin{threeparttable}
\centering
\begin{tabular}{|c|c|c|c|} 
\hline
& \multicolumn{3}{c|}{Average length}\\
&\multicolumn{3}{c|}{(Standard deviation)}\\
 \hline
\backslashbox{Estimator and \\sampling  design \\based on which CI is constructed}{Sample size}
&  $n$=$75$ & $n$=$100$ & $n$=$125$  \\ 
 \hline
\multirow{2}{*}{ $\hat{\overline{Y}}_{H}$, SRSWOR}& $536.821$& $538.177$& $ 539.218$\\
&$(11.357)$& $(9.0784)$&$(6.8211)$\\
\multirow{2}{*}{\tnote{3}  \  $\hat{\overline{Y}}_{PEML}$, SRSWOR} & $44.824$& $38.81$ & $34.648$\\
&$(3.7002)$& $(2.7727)$&$(2.2055)$\\
\multirow{2}{*}{$\hat{\overline{Y}}_{HT}$, RS}& $689.123$& $597.999$& $535.951$\\
&$(7.8452)$& $(5.7176)$&$(4.8422)$\\
\multirow{2}{*}{$\hat{\overline{Y}}_{H}$, RS}& $102.611$& $87.915$& $59.98307$\\
&$(10.969)$& $(8.453)$&$(6.5828)$\\
\multirow{2}{*}{\tnote{3} \ $\hat{\overline{Y}}_{PEML}$, RS}& $345.956$& $115.944$& $78.711$\\
&$(654.77)$& $(265.93)$&$(1041.2)$\\
\multirow{2}{*}{$\hat{\overline{Y}}_{RHC}$, RHC} & $848.033$& $624.881$& $541.421$ \\
&$(6.8489)$& $(4.9609)$&$(4.0927)$\\
\multirow{2}{*}{\tnote{3} \ $\hat{\overline{Y}}_{PEML}$, RHC}& $64.573$& $56.531$& $50.601$\\
&$(715.16)$& $(275.11)$&$(651.31)$\\
\hline
\end{tabular}
\begin{tablenotes}
\item[3] It is to be noted that in the cases of PEML and GREG estimators under any given sampling design, we have the same asymptotic MSE and hence the same asymptotic CI. Therefore, the average and the standard deviation of CIs are not reported for the GREG estimator. 
\end{tablenotes}
\end{threeparttable}
\end{table}

\begin{table}[h!]
\caption{Average and standard deviation of lengths of asymptotically $95\%$ CIs for variance of $y$. Recall from Table $4$ in Section $2$ that for variance of $y$, $h(y_1)$=$(y^2,y)$ and $g(s_1,s_2)$=$s_1-s_2^2$.}
\label{table 7}
\centering
\begin{tabular}{|c|c|c|c|} 
\hline
& \multicolumn{3}{|c|}{Average length}\\
&\multicolumn{3}{c|}{(Standard deviation)}\\
 \hline
\backslashbox{Estimator and \\ sampling design \\based  on which CI\\ is constructed}{Sample size}
&  $n$=$75$ & $n$=$100$ & $n$=$125$  \\ 
 \hline
\multirow{2}{*}{$g(\hat{\overline{h}}_{H})$, SRSWOR}& $1010775$& $878689.4$& $786228$\\
&$(34245.5)$& $(26373.9)$&$(20414.5)$\\
\multirow{2}{*}{$g(\hat{\overline{h}}_{PEML})$, SRSWOR}& $29432.4$& $25929$& $23422$\\
&$(6076.97)$& $(4441.2)$&$(3526.8)$\\
\multirow{2}{*}{$g(\hat{\overline{h}}_{H})$, RS}& $444594.4$& $434160.7$& $239065$\\
&$(44701.7)$& $(31965.2)$&$(26739.6)$\\
\multirow{2}{*}{$g(\hat{\overline{h}}_{PEML})$, RS}& $1152403$& $1290084$& $235909.1$\\
&$(9083944)$& $(869339.1)$&$(1183961)$\\
\multirow{2}{*}{$g(\hat{\overline{h}}_{PEML})$, RHC}& $1031407$& $895639$& $801178.9$\\
&$(7311193)$& $(1530759)$&$(417582.9)$\\
\hline
\end{tabular}
\end{table}

\begin{table}[h!]
\caption{Average and standard deviation of lengths of asymptotically $95\%$ CIs for correlation coefficient between $z_1$ and $z_2$. Recall from Table $4$ in Section $2$ that for correlation coefficient between $z_1$ and $z_2$, $h(z_1,z_2)$=$(z_1,z_2,z_1^2,z_2^2,z_1 z_2)$ and $g(s_1,s_2,s_3,s_4,s_5)$=$(s_5-s_1s_2)/((s_3-s_1^2)(s_4-s_2^2))^{1/2}$.}
\label{table 8}
\centering
\begin{tabular}{|c|c|c|c|} 
\hline
& \multicolumn{3}{c|}{Average length}\\
&\multicolumn{3}{c|}{(Standard deviation)}\\
 \hline
\backslashbox{Estimator and \\sampling  design \\based on which CI is constructed}{Sample size}
&  $n$=$75$ & $n$=$100$ & $n$=$125$  \\ 
 \hline
\multirow{2}{*}{$g(\hat{\overline{h}}_{H})$, SRSWOR}& $8.2191$& $8.0909$& $8.0897$\\
&$(2.429)$& $(1.889)$&$(1.449)$\\
\multirow{2}{*}{$g(\hat{\overline{h}}_{PEML})$, SRSWOR}& $0.2542$& $0.2575$& $0.2583$ \\
&$(0.0467)$& $(0.0365)$&$(0.0294)$\\
\multirow{2}{*}{$g(\hat{\overline{h}}_{H})$, RS}& $4.6847$& $3.3135$& $1.3942$\\
&$(2.555)$& $(1.884)$&$(1.421)$\\
\multirow{2}{*}{$g(\hat{\overline{h}}_{PEML})$, RS}& $5.0473$& $4.3229$& $3.1306$\\
&$(162.9)$& $(17.19)$&$(21.04)$\\
\multirow{2}{*}{$g(\hat{\overline{h}}_{PEML})$, RHC}& $8.3174$& $8.3898$& $8.3514$\\
&$(15.82)$& $(41.88)$&$(19.62)$\\
\hline
\end{tabular}
\end{table}

\begin{table}[h!]
\caption{Average and standard deviation of lengths of asymptotically $95\%$ CIs for regression coefficient of $z_1$ on $z_2$. Recall from Table $4$ in Section $2$ that for regression coefficient of $z_1$ on $z_2$, $h(z_1,z_2)$=$(z_1, z_2, z^2_2, z_1 z_2)$ and $g(s_1,s_2,s_3, s_4)$=$(s_4-s_1 s_2)/(s_3-s^2_2)$.}
\label{table 9}
\centering
\begin{tabular}{|c|c|c|c|} 
\hline
& \multicolumn{3}{c|}{Average length}\\
&\multicolumn{3}{c|}{(Standard deviation)}\\
 \hline
\backslashbox{Estimator and \\sampling  design \\based on which CI is constructed}{Sample size}
&  $n$=$75$ & $n$=$100$ & $n$=$125$  \\ 
 \hline
\multirow{2}{*}{$g(\hat{\overline{h}}_{H})$, SRSWOR}& $5.9565$& $5.068$& $4.4818$\\
&$(2.013)$& $(1.514)$&$(1.135)$\\
\multirow{2}{*}{$g(\hat{\overline{h}}_{PEML})$, SRSWOR}& $0.2596$& $0.2251$& $0.2032$\\
&$(0.0429)$& $(0.0324)$&$(0.025)$\\
\multirow{2}{*}{$g(\hat{\overline{h}}_{H})$, RS}& $3.0488$& $1.469$& $1.1532$\\
&$(2.178)$& $(1.517)$&$(1.171)$\\
\multirow{2}{*}{$g(\hat{\overline{h}}_{PEML})$, RS}& $3.6477$& $1.8558$& $1.4023$\\
&$(19.09)$& $(4.697)$&$(4.672)$\\
\multirow{2}{*}{$g(\hat{\overline{h}}_{PEML})$, RHC}& $6.111$& $5.1324$& $4.6658$\\
&$(25.16)$& $(38.36)$&$(11.17)$\\
\hline
\end{tabular}
\end{table}

\begin{table}[h]
\caption{Average and standard deviation of lengths of asymptotically $95\%$ CIs for regression coefficient of $z_2$ on $z_1$. Recall from Table $4$ in Section $2$ that for regression coefficient of $z_2$ on $z_1$, $h(z_1,z_2)$=$(z_2, z_1, z^2_1, z_1 z_2)$ and $g(s_1,s_2,s_3, s_4)$=$(s_4-s_1 s_2)/(s_3-s^2_2)$.}
\label{table 10}
\centering
\begin{tabular}{|c|c|c|c|} 
\hline
& \multicolumn{3}{c|}{Average length}\\
&\multicolumn{3}{c|}{(Standard deviation)}\\
 \hline
\backslashbox{Estimator and \\sampling  design \\based on which CI is constructed}{Sample size}
&  $n$=$75$ & $n$=$100$ & $n$=$125$  \\ 
 \hline
\multirow{2}{*}{$g(\hat{\overline{h}}_{H})$, SRSWOR}& $11.2173$& $9.6463$& $8.5885$\\
&$(3.238)$& $(2.418)$&$(1.877)$\\
\multirow{2}{*}{$g(\hat{\overline{h}}_{PEML})$, SRSWOR}& $0.4198$& $0.3652$& $0.3307$\\
&$(0.0661)$& $(0.0531)$&$(0.0405)$\\
\multirow{2}{*}{$g(\hat{\overline{h}}_{H})$, RS}& $6.7247$&$3.3547$ &$1.7421$\\
&$(3.546)$& $(2.539)$&$(1.921)$\\
\multirow{2}{*}{$g(\hat{\overline{h}}_{PEML})$, RS}& $11.3373$& $9.988$& $8.7889$\\
&$(151.9)$& $(31.83)$&$(7.405)$\\
\multirow{2}{*}{$g(\hat{\overline{h}}_{PEML})$, RHC}& $19.9049$& $3.5595$&  $1.8327$\\
&$(28.77)$& $(321.7)$&$(8.164)$\\
\hline
\end{tabular}
\end{table}
\clearpage

\section{Analysis based on real data}\label{sec S4}
The results obtained from the analyses carried out in Section $4.2$ of the main paper and Section \ref{sec S2} in this supplement are summarized in these sections. Here, we provide some scatter plots and tables that were mentioned in these sections.  Figures \ref{Fig 1} through \ref{Fig 4} present scatter plots  and least square regression lines  between different study and size variables drawn based on all the population values. Tables \ref{table 11} through \ref{table 20} contain relative efficiencies of estimators for the mean, the variance, the correlation coefficient and the regression coefficient in the population.  Tables \ref{table 21} through \ref{table 30} contain the average and the standard deviation of lengths of asymptotically $95\%$ CIs of the above parameters.   
\begin{figure}[h!]
\begin{center}
\includegraphics[height=15cm,width=14cm]{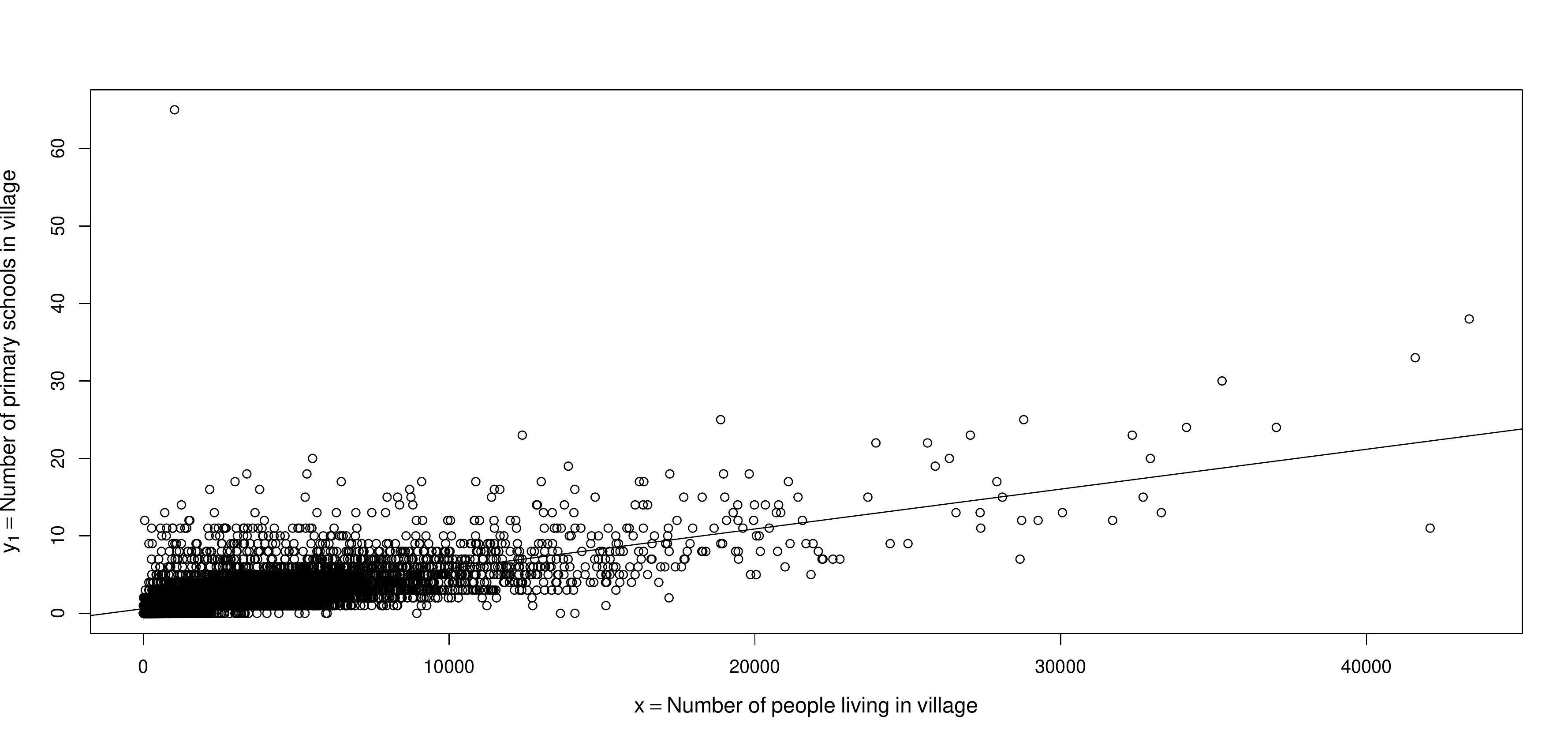}
\caption{Scatter plot  and least square regression line  for variables $y_1$ and $x$}
\label{Fig 1}
\end{center}
\end{figure}

\begin{figure}[h!]
\begin{center}
\includegraphics[height=16cm,width=14cm]{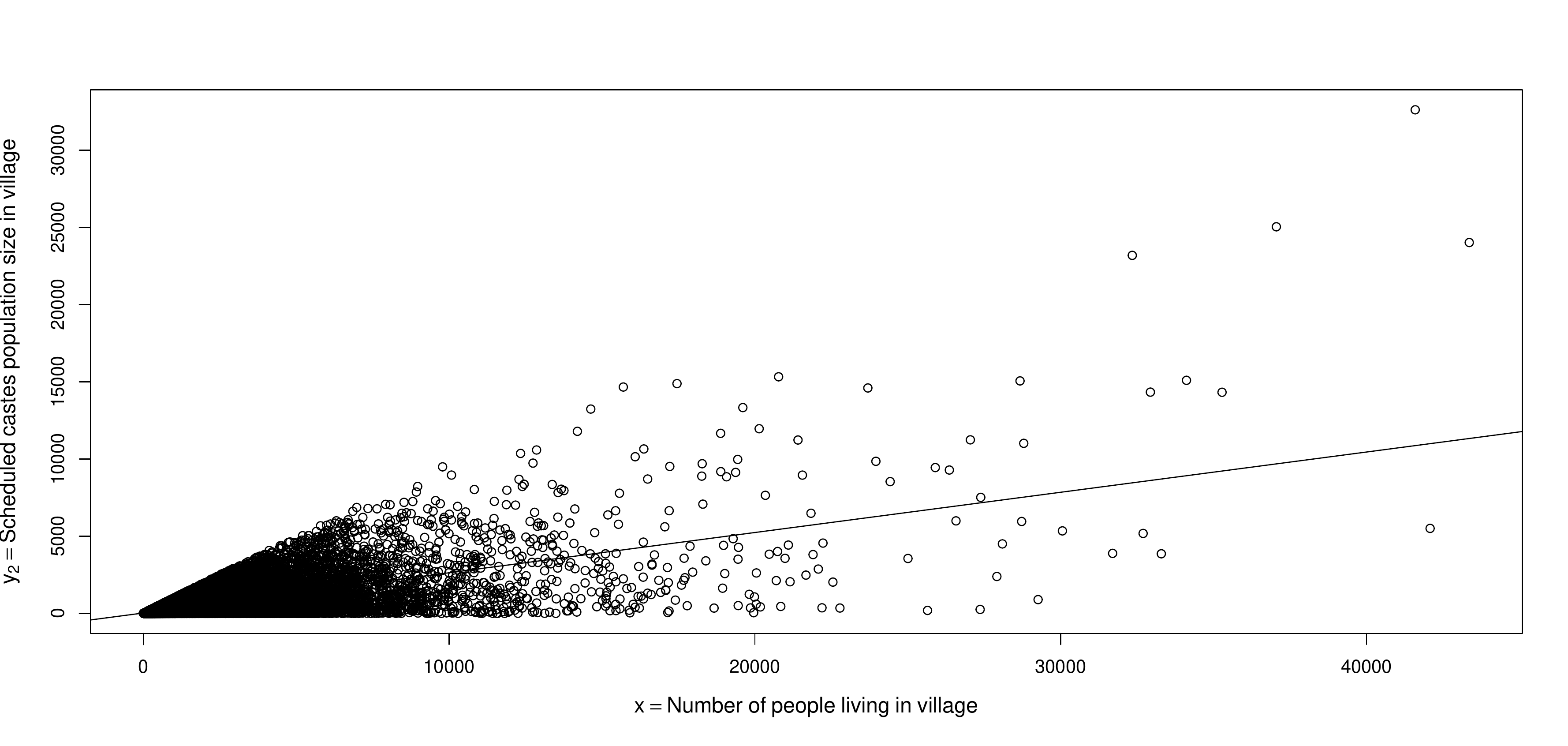}
\caption{Scatter plot  and least square regression line  for variables $y_2$ and $x$}
\label{Fig 2}
\end{center}
\end{figure}

\begin{figure}[h!]
\begin{center}
\includegraphics[height=16cm,width=14cm]{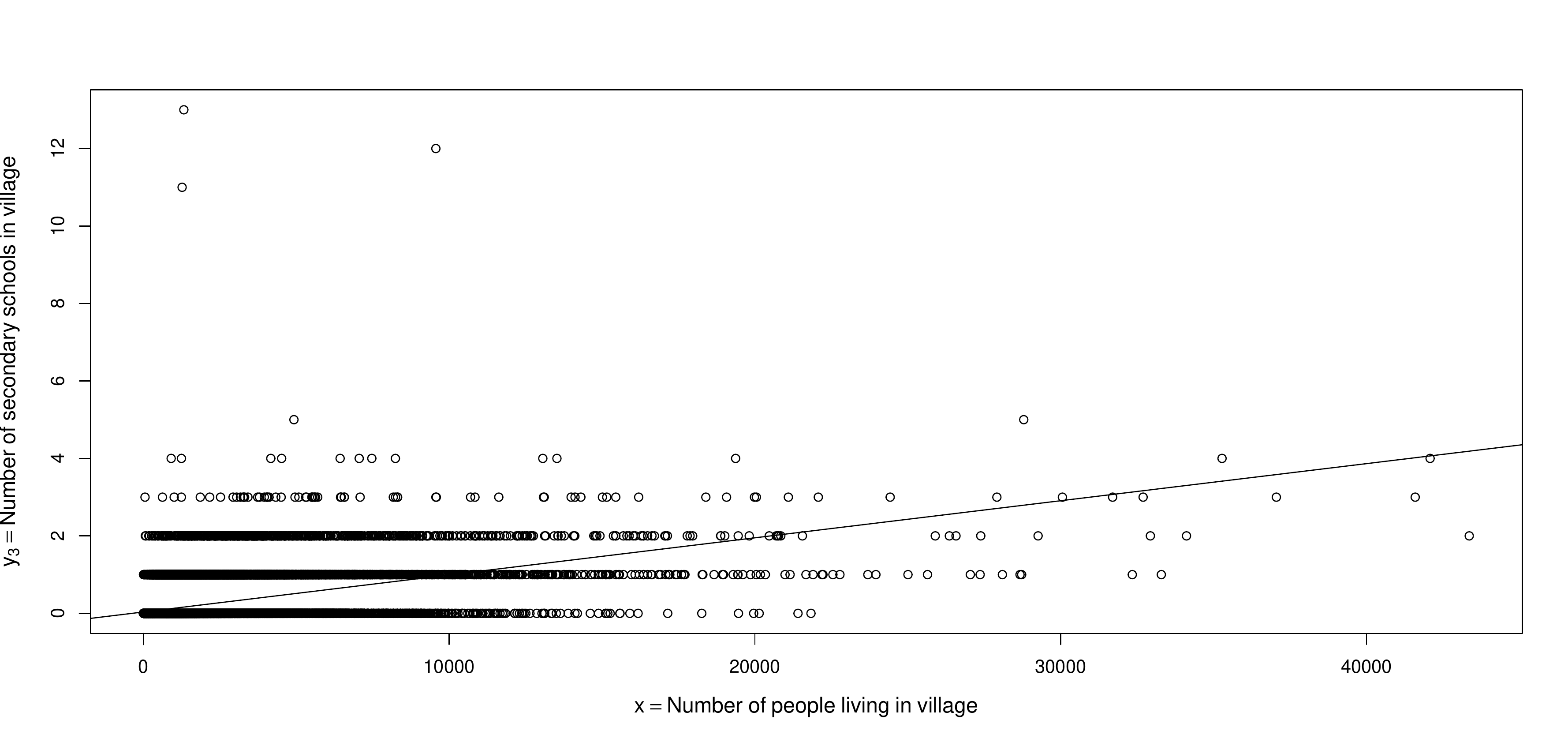}
\caption{Scatter plot  and least square regression line  for variables $y_3$ and $x$}
\label{Fig 3}
\end{center}
\end{figure}

\begin{figure}[h!]
\begin{center}
\includegraphics[height=16cm,width=14cm]{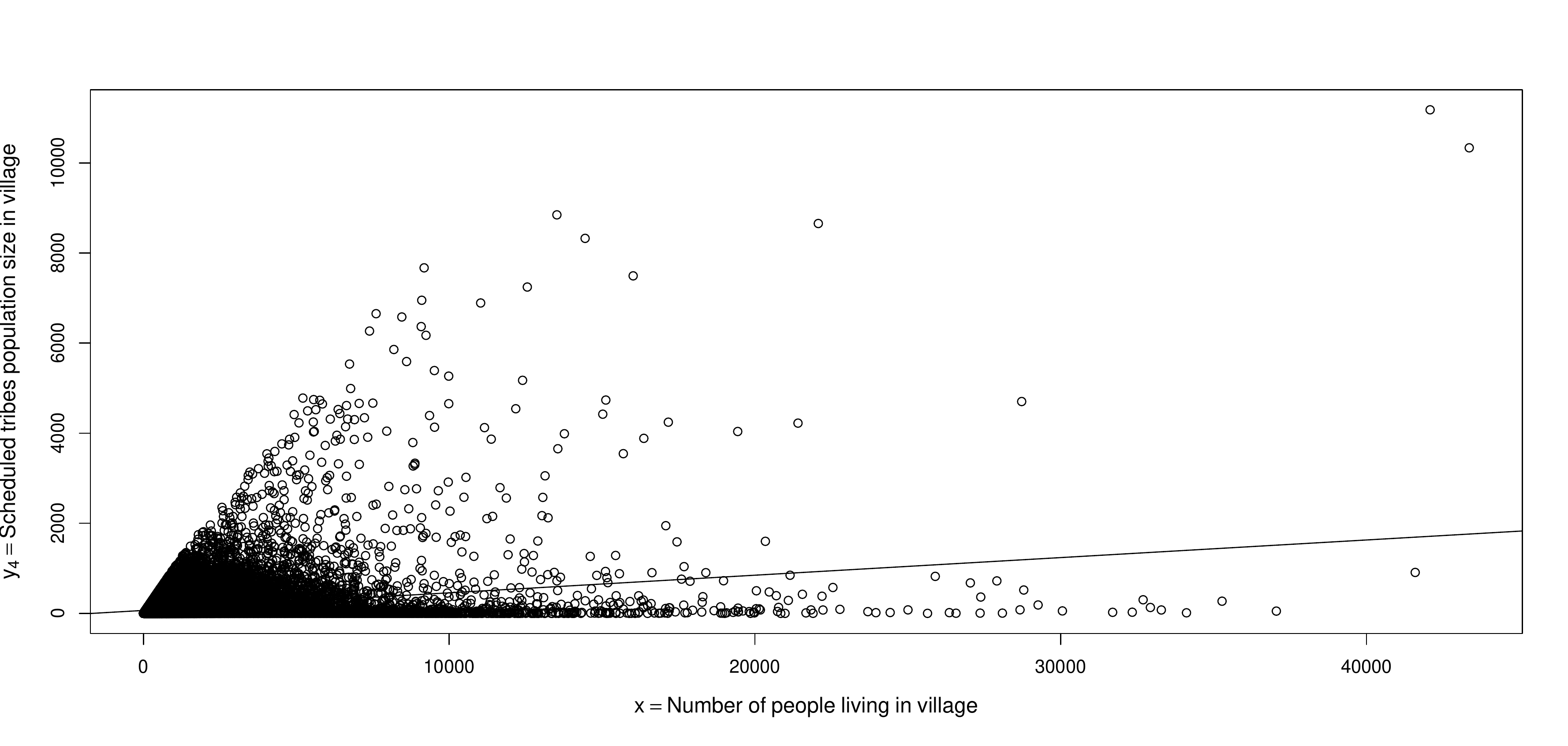}
\caption{Scatter plot  and least square regression line  for variables $y_4$ and $x$}
\label{Fig 4}
\end{center}
\end{figure}

\begin{table}[h!]
\caption{Relative efficiencies of estimators for mean of $y_1$.}
\renewcommand{\arraystretch}{1.2}
\label{table 11}
\begin{threeparttable}
\centering
\begin{tabular}{|c|c|c|c|} 
\hline
\backslashbox{Relative efficiency}{Sample size}
&  $n$=$75$ & $n$=$100$ & $n$=$125$  \\ 
 \hline
 RE($\hat{\overline{Y}}_{PEML}$, SRSWOR $|$ $\hat{\overline{Y}}_{GREG}$, SRSWOR)& $1.008215$& $1.005233$& $1.020408$\\
 RE($\hat{\overline{Y}}_{PEML}$, SRSWOR $|$ $\hat{\overline{Y}}_{H}$, RS)& $3.503939$& $3.880443$& $4.175886$\\
 RE($\hat{\overline{Y}}_{PEML}$, SRSWOR $|$ $\hat{\overline{Y}}_{HT}$, RS)& $1.796937$& $2.182675$& $1.8311$\\
 RE($\hat{\overline{Y}}_{PEML}$, SRSWOR $|$ $\hat{\overline{Y}}_{PEML}$, RS)&  $1.20961$& $1.228022$& $1.50233$\\
 RE($\hat{\overline{Y}}_{PEML}$, SRSWOR $|$ $\hat{\overline{Y}}_{GREG}$, RS)& $1.21831$& $1.237737$& $1.553863$\\
 RE($\hat{\overline{Y}}_{PEML}$, SRSWOR $|$ $\hat{\overline{Y}}_{RHC}$, RHC)&  $3.274031$& $2.059141$& $2.030995$\\
 RE($\hat{\overline{Y}}_{PEML}$, SRSWOR $|$ $\hat{\overline{Y}}_{PEML}$, RHC)&  $1.088166$& $1.388563$ & $1.51547$\\
 RE($\hat{\overline{Y}}_{PEML}$, SRSWOR $|$ $\hat{\overline{Y}}_{GREG}$, RHC)&$1.097934$& $1.398241$& $1.567545$\\
\hline
RE($\hat{\overline{Y}}_{PEML}$, SRSWOR $|$ \tnote{1} $\hat{\overline{Y}}_{BCPEML}$, SRSWOR)& $1.070226$& $1.019958$& $1.007533$\\
RE($\hat{\overline{Y}}_{GREG}$, SRSWOR $|$ \tnote{1} $\hat{\overline{Y}}_{BCGREG}$, SRSWOR)& $1.146007$&$ 1.116225$&$ 1.117507$\\
RE($\hat{\overline{Y}}_{H}$, RS $|$ \tnote{1} $\hat{\overline{Y}}_{BCH}$, RS)&$ 1.240493$&$ 1.012969$&$ 1.155246$\\
RE($\hat{\overline{Y}}_{PEML}$, RS $|$ \tnote{1} $\hat{\overline{Y}}_{BCPEML}$, RS)& $1.374578$&$ 1.046986$&$ 1.055930$ \\
RE($\hat{\overline{Y}}_{GREG}$, RS $|$ \tnote{1} $\hat{\overline{Y}}_{BCGREG}$, RS)&$1.466647$&$ 1.138300$&$ 1.205053$\\
RE($\hat{\overline{Y}}_{PEML}$, RHC $|$ \tnote{1} $\hat{\overline{Y}}_{BCPEML}$, RHC)&$1.566827$&$ 1.083589$&$ 1.132790$ \\
RE($\hat{\overline{Y}}_{GREG}$, RHC $|$ \tnote{1} $\hat{\overline{Y}}_{BCGREG}$, RHC)& $1.460886$&$ 1.037045$&$ 1.028358$\\
\hline
\end{tabular}
\end{threeparttable}

\footnotetext[1]{}
\end{table}

\begin{table}[h!]
\renewcommand{\arraystretch}{1.2}
\caption{Relative efficiencies of estimators for variance of $y_1$. Recall from Table $4$ in Section $2$ that for variance of $y_1$, $h(y_1)$=$(y_1^2,y_1)$ and $g(s_1,s_2)$=$s_1-s_2^2$.}
\label{table 12}
\begin{threeparttable}
\centering
\begin{tabular}{|c|c|c|c|} 
\hline
\backslashbox{Relative efficiency}{Sample size}
&  $n$=$75$ & $n$=$100$ & $n$=$125$  \\ 
 \hline 
  RE($g(\hat{\overline{h}}_{PEML})$, SRSWOR $|$ $g(\hat{\overline{h}}_{H})$, SRSWOR)  &  $1.3294$& $1.2413$& $1.1476$\\
 RE($g(\hat{\overline{h}}_{PEML})$, SRSWOR $|$ $g(\hat{\overline{h}}_{H})$, RS) & $2.5303$& $1.6656$& $1.5374$\\
 RE($g(\hat{\overline{h}}_{PEML})$, SRSWOR $|$ $g(\hat{\overline{h}}_{PEML})$, RS) &  $3.1642$& $2.4051$& $2.5831$\\
 RE($g(\hat{\overline{h}}_{PEML})$, SRSWOR $|$ $g(\hat{\overline{h}}_{PEML})$, RHC)  &  $ 2.5499$& $4.7704$& $3.0985$\\
\hline
 RE($g(\hat{\overline{h}}_{PEML})$, SRSWOR $|$ \tnote{2} \  BC $g(\hat{\overline{h}}_{PEML})$, SRSWOR)  &  $1.1812$& $1.2736$& $1.8669$\\ 
  RE($g(\hat{\overline{h}}_{H})$, SRSWOR $|$ \tnote{2} \  BC $g(\hat{\overline{h}}_{H})$, SRSWOR)  &$ 4.3526$&$ 4.8948$&$ 6.0349$\\
 RE($g(\hat{\overline{h}}_{H})$, RS $|$ \tnote{2} \  BC $g(\hat{\overline{h}}_{H})$, RS)  &  $1.115$&$ 1.1239$&$ 1.2269$\\ 
 RE($g(\hat{\overline{h}}_{PEML})$, RS $|$ \tnote{2} \  BC $g(\hat{\overline{h}}_{PEML})$, RS)  & $1.4373$& $1.1739$& $1.6481$ \\  
 RE($g(\hat{\overline{h}}_{PEML})$, RHC $|$ \tnote{2} \  BC $g(\hat{\overline{h}}_{PEML})$, RHC)  & $1.8502$&$ 1.0186$&$ 1.0384$  \\ 
\hline
\end{tabular}
\end{threeparttable}
\end{table}
\clearpage

\begin{table}[h!]
\caption{Relative efficiencies of estimators for mean of $y_2$.}
\renewcommand{\arraystretch}{1.2}
\label{table 13}
\begin{threeparttable}
\centering
\begin{tabular}{|c|c|c|c|} 
\hline
\backslashbox{Relative efficiency}{Sample size}
&  $n$=$75$ & $n$=$100$ & $n$=$125$  \\ 
 \hline
 RE($\hat{\overline{Y}}_{HT}$, RS $|$ $\hat{\overline{Y}}_{H}$, RS)& $4.367712$ &$4.008655$ & $4.463214$\\
 RE($\hat{\overline{Y}}_{HT}$, RS $|$ $\hat{\overline{Y}}_{PEML}$, RS)& $1.148074$  &$1.082488$ & $1.088804$\\
 RE($\hat{\overline{Y}}_{HT}$, RS $|$ $\hat{\overline{Y}}_{GREG}$, RS)& $1.216958$ &$1.115967$ & $1.154132$\\
 RE($\hat{\overline{Y}}_{HT}$, RS $|$ $\hat{\overline{Y}}_{RHC}$, RHC)&  $ 1.073138$& $1.03213$& $1.07484$\\
 RE($\hat{\overline{Y}}_{HT}$, RS $|$ $\hat{\overline{Y}}_{PEML}$, RHC)& $1.230884$ & $1.0937$& $1.207308$\\
 RE($\hat{\overline{Y}}_{HT}$, RS $|$ $\hat{\overline{Y}}_{GREG}$, RHC)& $1.304737$& $1.127526$& $1.279746$\\
 RE($\hat{\overline{Y}}_{HT}$, RS $|$ $\hat{\overline{Y}}_{PEML}$, SRSWOR)&   $2.440441$&$2.305339$ & $2.350916$\\
 RE($\hat{\overline{Y}}_{HT}$, RS $|$ $\hat{\overline{Y}}_{GREG}$, SRSWOR)& $2.58687$ &$2.376638$ &$2.49197$\\
\hline
RE($\hat{\overline{Y}}_{H}$, RS $|$ \tnote{1} $\hat{\overline{Y}}_{BCH}$, RS)&$1.252123$& $1.325047$&$ 1.241809$\\
RE($\hat{\overline{Y}}_{PEML}$, RS $|$ \tnote{1} $\hat{\overline{Y}}_{BCPEML}$, RS)& $1.988105$&$ 2.146357$&$ 2.260343$\\
RE($\hat{\overline{Y}}_{GREG}$, RS $|$ \tnote{1} $\hat{\overline{Y}}_{BCGREG}$, RS)& $2.055588$&$ 2.018015$&$ 2.287817$\\
RE($\hat{\overline{Y}}_{PEML}$, RHC $|$ \tnote{1} $\hat{\overline{Y}}_{BCPEML}$, RHC)&$ 1.831377$&$ 2.083210$&$ 2.006134$ \\
RE($\hat{\overline{Y}}_{GREG}$, RHC $|$ \tnote{1} $\hat{\overline{Y}}_{BCGREG}$, RHC)& $1.925938$&$ 1.983984$&$ 2.091003$\\
RE($\hat{\overline{Y}}_{PEML}$, SRSWOR $|$ \tnote{1} $\hat{\overline{Y}}_{BCPEML}$, SRSWOR)& $1.001786$&$ 1.004973$&$ 1.060588$\\
RE($\hat{\overline{Y}}_{GREG}$, SRSWOR $|$ \tnote{1} $\hat{\overline{Y}}_{BCGREG}$, SRSWOR)&$1.021103$&$ 1.008525$&$ 1.003390$\\
\hline
\end{tabular}
\end{threeparttable}
\end{table}

\begin{table}[h!]
\renewcommand{\arraystretch}{0.75}
\caption{Relative efficiencies of estimators for variance of $y_2$. Recall from Table $4$ in Section $2$ that for variance of $y_2$, $h(y_2)$=$(y_2^2,y_2)$ and $g(s_1,s_2)$=$s_1-s_2^2$.}
\label{table 14}
\begin{threeparttable}
\centering
\begin{tabular}{|c|c|c|c|} 
\hline
\backslashbox{Relative efficiency}{Sample size}
&  $n$=$75$ & $n$=$100$ & $n$=$125$  \\ 
 \hline 
 RE($g(\hat{\overline{h}}_{H})$, RS $|$ $g(\hat{\overline{h}}_{PEML})$, RS) &  $11.893$& $6.967$& $34.691$\\
 RE($g(\hat{\overline{h}}_{H})$, RS $|$ $g(\hat{\overline{h}}_{PEML})$, RHC)  &  $5.0093$& $19.456$& $21.919$\\
 RE($g(\hat{\overline{h}}_{H})$, RS $|$ $g(\hat{\overline{h}}_{H})$, SRSWOR)  &  $ 9.8232$& $10.27$& $16.763$\\
 RE($g(\hat{\overline{h}}_{H})$, RS $|$ $g(\hat{\overline{h}}_{PEML})$, SRSWOR) & $ 2.4768$ & $4.8093$ &$6.2264$\\
\hline
 RE($g(\hat{\overline{h}}_{H})$, RS $|$ \tnote{2} \  BC $g(\hat{\overline{h}}_{H})$, RS)  &  $13.301$& $6.3589$& $33.579$\\
  RE($g(\hat{\overline{h}}_{PEML})$, RS $|$ \tnote{2} \  BC $g(\hat{\overline{h}}_{PEML})$, RS)  & $4.448$&$ 7.4621$&$ 7.989$ \\  
 RE($g(\hat{\overline{h}}_{PEML})$, RHC $|$ \tnote{2} \  BC $g(\hat{\overline{h}}_{PEML})$, RHC)  & $21.855$&$  3.0076$&$ 11.368$  \\ 
  RE($g(\hat{\overline{h}}_{H})$, SRSWOR $|$ \tnote{2} \  BC $g(\hat{\overline{h}}_{H})$, SRSWOR)  &$8.7641$&$  5.6119$&$ 13.7$\\
 RE($g(\hat{\overline{h}}_{PEML})$, SRSWOR $|$ \tnote{2} \  BC $g(\hat{\overline{h}}_{PEML})$, SRSWOR)  &  $6.2655$&$ 2.0015$&$ 6.959$\\ 

\hline
\end{tabular}
\end{threeparttable}
\end{table}

\begin{table}[h!]
\renewcommand{\arraystretch}{0.75}
\caption{Relative efficiencies of estimators for correlation coefficient between $y_1$ and $y_3$. Recall from Table $4$ in Section $2$ that for correlation coefficient between $y_1$ and $y_3$, $h(y_1,y_3)$=$(y_1,y_3,y_1^2,y_3^2,y_1 y_3)$ and $g(s_1,s_2,s_3,s_4,s_5)$=$(s_5-s_1s_2)/((s_3-s_1^2)(s_4-s_2^2))^{1/2}$.}
\label{table 15}
\begin{threeparttable}
\centering
\begin{tabular}{|c|c|c|c|}
\hline 
\backslashbox{Relative efficiency}{Sample size}
&  $n$=$75$ & $n$=$100$ & $n$=$125$  \\
\hline
 RE($g(\hat{\overline{h}}_{PEML})$, SRSWOR $|$ $g(\hat{\overline{h}}_{H})$, SRSWOR)& $1.0967$& $1.0369$& $1.0374$ \\ 
 RE($g(\hat{\overline{h}}_{PEML})$, SRSWOR $|$ $g(\hat{\overline{h}}_{H})$, RS)& $1.317$& $1.4831$& $1.2561$ \\
 RE($g(\hat{\overline{h}}_{PEML})$, SRSWOR $|$ $g(\hat{\overline{h}}_{PEML})$, RS)&  $1.9803$& $1.9874$& $1.8441$ \\ 
 RE($g(\hat{\overline{h}}_{PEML})$, SRSWOR $|$ $g(\hat{\overline{h}}_{PEML})$, RHC)& $2.0562$& $1.9651$& $1.8541$ \\
\hline
 RE($g(\hat{\overline{h}}_{PEML})$, SRSWOR $|$ \tnote{2} \  BC $g(\hat{\overline{h}}_{PEML})$, SRSWOR)  &  $23.149$& $51.887$& $45.976$\\
  RE($g(\hat{\overline{h}}_{H})$, SRSWOR $|$ \tnote{2} \  BC $g(\hat{\overline{h}}_{H})$, SRSWOR)  &$90.769$&$ 163.74$&$ 154.97$\\
 RE($g(\hat{\overline{h}}_{H})$, RS $|$ \tnote{2} \  BC $g(\hat{\overline{h}}_{H})$, RS)  &  $ 72.604$&$  79.355$&$ 163.03$\\ 
 RE($g(\hat{\overline{h}}_{PEML})$, RS $|$ \tnote{2} \  BC $g(\hat{\overline{h}}_{PEML})$, RS)  & $24.483$&$ 35.874$&$ 43.164$ \\  
 RE($g(\hat{\overline{h}}_{PEML})$, RHC $|$ \tnote{2} \  BC $g(\hat{\overline{h}}_{PEML})$, RHC)  & $29.189$&$ 65.949$&$ 43.13$\\
 \hline
\end{tabular}
\end{threeparttable}
\end{table}
\clearpage

\begin{table}[h!]
\renewcommand{\arraystretch}{0.6}
\caption{Relative efficiencies of estimators for regression coefficient of $y_1$ on $y_3$. Recall from Table $4$ in Section $2$ that for regression coefficient of $y_1$ on $y_3$, $h(y_1,y_3)$=$(y_1, y_3, y^2_3, y_1 y_3)$ and $g(s_1,s_2,s_3, s_4)$=$(s_4-s_1 s_2)/(s_3-s^2_2)$.}
\label{table 16}
\begin{threeparttable}
\centering
\begin{tabular}{|c|c|c|c|}
\hline 
\backslashbox{Relative efficiency}{Sample size}
&  $n$=$75$ & $n$=$100$ & $n$=$125$  \\
\hline
 RE($g(\hat{\overline{h}}_{PEML})$, SRSWOR $|$ $g(\hat{\overline{h}}_{H})$, SRSWOR)& $1.0298$& $1.0504$& $1.0423$ \\
 RE($g(\hat{\overline{h}}_{PEML})$, SRSWOR $|$ $g(\hat{\overline{h}}_{H})$, RS)& $ 1.8046$& $1.2304$& $1.3482$ \\
 RE($g(\hat{\overline{h}}_{PEML})$, SRSWOR $|$ $g(\hat{\overline{h}}_{PEML})$, RS)&  $2.2709$& $1.5949$& $1.854$ \\ 
 RE($g(\hat{\overline{h}}_{PEML})$, SRSWOR $|$ $g(\hat{\overline{h}}_{PEML})$, RHC)& $1.8719$& $1.5069$& $1.5626$ \\
\hline
 RE($g(\hat{\overline{h}}_{PEML})$, SRSWOR $|$ \tnote{2} \  BC $g(\hat{\overline{h}}_{PEML})$, SRSWOR)  &  $31.789$& $50.26$& $50.107$\\
RE($g(\hat{\overline{h}}_{H})$, SRSWOR $|$ \tnote{2} \  BC $g(\hat{\overline{h}}_{H})$, SRSWOR)  &$236.49$&$ 119.88$&$ 222.23$\\
 RE($g(\hat{\overline{h}}_{H})$, RS $|$ \tnote{2} \  BC $g(\hat{\overline{h}}_{H})$, RS)  &  $63.933$&$  77.049$&$ 184.45$\\ 
 RE($g(\hat{\overline{h}}_{PEML})$, RS $|$ \tnote{2} \  BC $g(\hat{\overline{h}}_{PEML})$, RS)  & $ 31.503$&$  44.945$&$ 263.5$ \\  
 RE($g(\hat{\overline{h}}_{PEML})$, RHC $|$ \tnote{2} \  BC $g(\hat{\overline{h}}_{PEML})$, RHC)  & $65.145$&$ 76.533$&$ 90.413$\\ 
 \hline
\end{tabular}
\end{threeparttable}
\end{table}

\begin{table}[h!]
\renewcommand{\arraystretch}{0.5}
\caption{Relative efficiencies of estimators for regression coefficient of $y_3$ on $y_1$. Recall from Table $4$ in Section $2$ that for regression coefficient of $y_3$ on $y_1$, $h(y_1,y_3)$=$(y_3, y_1, y^2_1, y_1 y_3)$ and $g(s_1,s_2,s_3, s_4)$=$(s_4-s_1 s_2)/(s_3-s^2_2)$.}
\label{table 17}
\centering
\begin{threeparttable}
\begin{tabular}{|c|c|c|c|}
\hline 
\backslashbox{Relative efficiency}{Sample size}
&  $n$=$75$ & $n$=$100$ & $n$=$125$  \\
\hline
  RE($g(\hat{\overline{h}}_{PEML})$, SRSWOR $|$ $g(\hat{\overline{h}}_{H})$, SRSWOR)& $1.0997$& $1.2329$& $1.1529$ \\
 RE($g(\hat{\overline{h}}_{PEML})$, SRSWOR $|$ $g(\hat{\overline{h}}_{H})$, RS)&  $1.3948$& $1.3329$& $1.368$ \\
 RE($g(\hat{\overline{h}}_{PEML})$, SRSWOR $|$ $g(\hat{\overline{h}}_{PEML})$, RS)&  $3.6069$& $1.5532$& $1.8035$ \\ 
 RE($g(\hat{\overline{h}}_{PEML})$, SRSWOR $|$ $g(\hat{\overline{h}}_{PEML})$, RHC)& $2.5567$& $1.4867$& $1.5335$ \\
\hline
 RE($g(\hat{\overline{h}}_{PEML})$, SRSWOR $|$ \tnote{2} \  BC $g(\hat{\overline{h}}_{PEML})$, SRSWOR)  & $26.09$&$29.557$& $32.345$\\
 RE($g(\hat{\overline{h}}_{H})$, SRSWOR $|$ \tnote{2} \  BC $g(\hat{\overline{h}}_{H})$, SRSWOR)  &$ 98.43$&$ 104.19$&$ 165.95$\\
 RE($g(\hat{\overline{h}}_{H})$, RS $|$ \tnote{2} \  BC $g(\hat{\overline{h}}_{H})$, RS)  &  $100.3$&$ 110.15$&$ 196.34$ \\ 
 RE($g(\hat{\overline{h}}_{PEML})$, RS $|$ \tnote{2} \  BC $g(\hat{\overline{h}}_{PEML})$, RS)  & $ 11.416$&$ 71.664$&$ 23.433$ \\  
 RE($g(\hat{\overline{h}}_{PEML})$, RHC $|$ \tnote{2} \  BC $g(\hat{\overline{h}}_{PEML})$, RHC)  & $13.268$&$ 28.198$&$ 50.571$\\
 \hline
\end{tabular}
\end{threeparttable}
\end{table}
\clearpage

\begin{table}[h!]
\renewcommand{\arraystretch}{0.58}
\caption{Relative efficiencies of estimators for correlation coefficient between $y_2$ and $y_4$. Recall from Table $4$ in Section $2$ that for correlation coefficient between $y_2$ and $y_4$, $h(y_2,y_4)$=$(y_2,y_4,y_2^2,y_4^2, y_2 y_4)$ and $g(s_1,s_2,s_3,s_4,s_5)$=$(s_5-s_1s_2)/((s_3-s_1^2)(s_4-s_2^2))^{1/2}$.}
\label{table 18}
\begin{threeparttable}
\centering
\begin{tabular}{|c|c|c|c|}
\hline 
\backslashbox{Relative efficiency}{Sample size}
&  $n$=$75$ & $n$=$100$ & $n$=$125$  \\
\hline
 RE($g(\hat{\overline{h}}_{H})$, RS $|$ $g(\hat{\overline{h}}_{PEML})$, RS)&  $1.448$& $1.696$& $2.027$ \\ 
 RE($g(\hat{\overline{h}}_{H})$, RS $|$ $g(\hat{\overline{h}}_{PEML})$, RHC)& $1.491$& $2.135$& $2.27$ \\
 RE($g(\hat{\overline{h}}_{H})$, RS $|$ $g(\hat{\overline{h}}_{H})$, SRSWOR)& $2.39$& $2.521$& $2.849$ \\
 RE($g(\hat{\overline{h}}_{H})$, RS $|$ $g(\hat{\overline{h}}_{PEML})$, SRSWOR)& $2.185$ & $2.396$& $2.594$\\
\hline
 RE($g(\hat{\overline{h}}_{H})$, RS $|$ \tnote{2} \  BC $g(\hat{\overline{h}}_{PEML})$, RS)  &  $79.092$& $58.241$& $120.229$\\
  RE($g(\hat{\overline{h}}_{PEML})$, RS $|$ \tnote{2} \  BC $g(\hat{\overline{h}}_{PEML})$, RS)  & $82.309$&$  61.995$&$ 316.929$ \\  
 RE($g(\hat{\overline{h}}_{PEML})$, RHC $|$ \tnote{2} \  BC $g(\hat{\overline{h}}_{PEML})$, RHC)  & $1 75.22$&$  74.847$&$ 220.74$\\ 
RE($g(\hat{\overline{h}}_{H})$, SRSWOR $|$ \tnote{2} \  BC $g(\hat{\overline{h}}_{H})$, SRSWOR)  &$87.942$&$ 36.363$&$ 97.432$\\ 
 RE($g(\hat{\overline{h}}_{PEML})$, SRSWOR $|$ \tnote{2} \  BC $g(\hat{\overline{h}}_{PEML})$, SRSWOR)  &  $ 120.02$&$  51.959$&$ 121.42$\\
\hline
\end{tabular}
\end{threeparttable}
\end{table}

\begin{table}[h!]
\renewcommand{\arraystretch}{0.6}
\caption{Relative efficiencies of estimators for regression coefficient of $y_2$ on $y_4$. Recall from Table $4$ in Section $2$ that for regression coefficient of $y_2$ on $y_4$, $h(y_2,y_4)$=$(y_2, y_4, y^2_4, y_2 y_4)$ and $g(s_1,s_2,s_3, s_4)$=$(s_4-s_1 s_2)/(s_3-s^2_2)$.}
\label{table 19}
\centering
\begin{threeparttable}
\begin{tabular}{|c|c|c|c|}
\hline 
\backslashbox{Relative efficiency}{Sample size}
&  $n$=$75$ & $n$=$100$ & $n$=$125$  \\
\hline
 RE($g(\hat{\overline{h}}_{H})$, RS $|$ $g(\hat{\overline{h}}_{PEML})$, RS)&  $1.8158$& $2.3771$& $3.2021$ \\ 
 RE($g(\hat{\overline{h}}_{H})$, RS $|$ $g(\hat{\overline{h}}_{PEML})$, RHC)& $2.5985$& $2.6002$& $3.4744$ \\
 RE($g(\hat{\overline{h}}_{H})$, RS $|$ $g(\hat{\overline{h}}_{H})$, SRSWOR)& $3.3278$& $4.5041$& $6.312$ \\
 RE($g(\hat{\overline{h}}_{H})$, RS $|$ $g(\hat{\overline{h}}_{PEML})$, SRSWOR)& $2.9788$ & $3.9417$& $6.0391$\\
\hline
 RE($g(\hat{\overline{h}}_{H})$, RS $|$ \tnote{2} \  BC $g(\hat{\overline{h}}_{H})$, RS)  &  $125.17$& $256.45$& $260.15$\\
  RE($g(\hat{\overline{h}}_{PEML})$, RS $|$ \tnote{2} \  BC $g(\hat{\overline{h}}_{PEML})$, RS)  & $145.1$&$ 333.5$&$ 135.65$ \\  
 RE($g(\hat{\overline{h}}_{PEML})$, RHC $|$ \tnote{2} \  BC $g(\hat{\overline{h}}_{PEML})$, RHC)  & $86.93$&$ 238.32$&$ 292.89$\\ 
  RE($g(\hat{\overline{h}}_{PEML})$, SRSWOR $|$ \tnote{2} \  BC $g(\hat{\overline{h}}_{PEML})$, SRSWOR)  &  $93.707$&$ 101.93$&$ 121.44$\\ 
  RE($g(\hat{\overline{h}}_{H})$, SRSWOR $|$ \tnote{2} \  BC $g(\hat{\overline{h}}_{H})$, SRSWOR)  &$115.85$&$ 146.16$& $104.66$\\

 \hline
\end{tabular}
\end{threeparttable}
\end{table}

\begin{table}[h!]
\caption{Relative efficiencies of estimators for regression coefficient of $y_4$ on $y_2$. Recall from Table $4$ in Section $2$ that for regression coefficient of $y_4$ on $y_2$, $h(y_2,y_4)$=$(y_4, y_2, y^2_2, y_2 y_4)$ and $g(s_1,s_2,s_3, s_4)$=$(s_4-s_1 s_2)/(s_3-s^2_2)$.}
\label{table 20}
\centering
\begin{threeparttable}[b]
\begin{tabular}{|c|c|c|c|}
\hline 
\backslashbox{Relative efficiency}{Sample size}
&  $n$=$75$ & $n$=$100$ & $n$=$125$  \\
\hline
 RE($g(\hat{\overline{h}}_{H})$, RS $|$ $g(\hat{\overline{h}}_{PEML})$, RS)& $1.3146$& $1.6055$& $1.937$  \\ 
 RE($g(\hat{\overline{h}}_{H})$, RS $|$ $g(\hat{\overline{h}}_{PEML})$, RHC)& $1.652$& $2.7715$& $2.0362$ \\
 RE($g(\hat{\overline{h}}_{H})$, RS $|$ $g(\hat{\overline{h}}_{H})$, SRSWOR)& $3.8248$& $2.4388$& $3.4371$ \\
 RE($g(\hat{\overline{h}}_{H})$, RS $|$ $g(\hat{\overline{h}}_{PEML})$, SRSWOR)& $3.1843$ & $2.3399$& $3.038$\\
\hline
 RE($g(\hat{\overline{h}}_{H})$, RS $|$ \tnote{2} \  BC $g(\hat{\overline{h}}_{H})$, RS)  &  $47.3317$& $73.749$& $52.592$ \\
  RE($g(\hat{\overline{h}}_{H})$, RS $|$ \tnote{2} \  BC $g(\hat{\overline{h}}_{PEML})$, RS)  & $105.87$& $126.42$& $323.82$ \\  
 RE($g(\hat{\overline{h}}_{H})$, RS $|$ \tnote{2} \  BC $g(\hat{\overline{h}}_{PEML})$, RHC)  &$93.403$& $79.453$& $91.347$\\ 
  RE($g(\hat{\overline{h}}_{H})$, RS $|$ \tnote{2} \  BC $g(\hat{\overline{h}}_{PEML})$, SRSWOR)  &  $530.94$& $173.19$& $191.26$\\
 RE($g(\hat{\overline{h}}_{H})$, RS $|$ \tnote{2} \  BC $g(\hat{\overline{h}}_{H})$, SRSWOR)  &$394.29$& $156.27$& $164.7$\\ 
 \hline
\end{tabular}
\end{threeparttable}
\end{table}

\begin{table}[h!]
\caption{Average and standard deviation of lengths of asymptotically $95\%$ CIs for mean of $y_1$.}
\label{table 21}
\centering
\begin{threeparttable}
\begin{tabular}{|c|c|c|c|} 
\hline
& \multicolumn{3}{c|}{Average length}\\
&\multicolumn{3}{c|}{(Standard deviation)}\\
 \hline
\backslashbox{Estimator and \\sampling  design \\based on which CI is constructed}{Sample size}
&  $n$=$75$ & $n$=$100$ & $n$=$125$  \\ 
 \hline
\multirow{2}{*}{ $\hat{\overline{Y}}_{H}$, SRSWOR}& $0.7233$& $0.7303$& $0.7333$\\
&$(0.2304)$ & $(0.1885)$&$(0.1431)$\\
\multirow{2}{*}{\tnote{3}  \  $\hat{\overline{Y}}_{PEML}$, SRSWOR} & $0.3703$& $0.3734$ & $0.3847$\\
&$(0.1608)$&$(0.1534)$&$(0.1074)$\\
\multirow{2}{*}{$\hat{\overline{Y}}_{HT}$, RS}& $0.7738$& $0.7735$& $0.8271$\\
&$(0.2724)$&$(1.071)$&$(0.2001)$\\
\multirow{2}{*}{$\hat{\overline{Y}}_{H}$, RS}& $0.4345$& $0.455$& $0.5414$\\
&$(0.8312)$&$(8.807)$&$(0.5479)$\\
\multirow{2}{*}{\tnote{3} \ $\hat{\overline{Y}}_{PEML}$, RS}& $0.6784$& $0.7207$& $0.7896$\\
&$(0.3945)$&$(12.176)$&$(0.2694)$\\
\multirow{2}{*}{$\hat{\overline{Y}}_{RHC}$, RHC} &  $0.7415$& $0.7716$& $0.8014$ \\
&$(0.4007)$&$(0.6359)$&$(0.2931)$\\
\multirow{2}{*}{\tnote{3} \ $\hat{\overline{Y}}_{PEML}$, RHC}& $0.4911$& $0.5078$& $0.5289$\\
&$(0.9865)$&$(0.4992)$&$(0.3594)$\\
\hline
\end{tabular}
\end{threeparttable}
\end{table}

\begin{table}[h!]
\caption{Average and standard deviation of lengths of asymptotically $95\%$ CIs for variance of $y_1$. Recall from Table $4$ in Section $2$ that for variance of $y_1$, $h(y_1)$=$(y_1^2,y_1)$ and $g(s_1,s_2)$=$s_1-s_2^2$.}
\label{table 22}
\centering
\begin{tabular}{|c|c|c|c|} 
\hline
& \multicolumn{3}{c|}{Average length}\\
&\multicolumn{3}{c|}{(Standard deviation)}\\
 \hline
\backslashbox{Estimator and \\sampling  design \\based on which CI is constructed}{Sample size}
&  $n$=$75$ & $n$=$100$ & $n$=$125$  \\ 
 \hline
\multirow{2}{*}{$g(\hat{\overline{h}}_{H})$, SRSWOR}& $5.2879$& $4.2111$& $4.4304$\\
&$(8.762)$&$(9.309)$&$(6.856)$\\
\multirow{2}{*}{$g(\hat{\overline{h}}_{PEML})$, SRSWOR}& $2.7519$& $2.9935$& $3.0013$\\
&$(7.181)$&$(8.622)$&$(5.952)$\\
\multirow{2}{*}{$g(\hat{\overline{h}}_{H})$, RS}& $3.5121$& $3.1177$& $3.1095$\\
&$(1.345)$&$(11.37)$&$(10.88)$\\
\multirow{2}{*}{$g(\hat{\overline{h}}_{PEML})$, RS}&$3.7475$ & $3.939$& $3.792$\\
&$(4.041)$&$(16.14)$&$(11.08)$\\
\multirow{2}{*}{$g(\hat{\overline{h}}_{PEML})$, RHC}& $3.6365$& $3.4972$& $3.4158$\\
&$(14.99)$&$(8.278)$&$(10.95)$\\
\hline
\end{tabular}
\end{table}

\begin{table}[h!]
\caption{Average and standard deviation of lengths of asymptotically $95\%$ CIs for mean of $y_2$.}
\label{table 23}
\centering
\begin{threeparttable}
\begin{tabular}{|c|c|c|c|} 
\hline
& \multicolumn{3}{c|}{Average length}\\
&\multicolumn{3}{c|}{(Standard deviation)}\\
 \hline
\backslashbox{Estimator and \\sampling  design \\based on which CI is constructed}{Sample size}
&  $n$=$75$ & $n$=$100$ & $n$=$125$  \\ 
 \hline
\multirow{2}{*}{$\hat{\overline{Y}}_{H}$, SRSWOR}& $312.1$& $322.48$& $326.36$\\
&$(150.08)$&$(121.86)$&$(93.707)$\\
\multirow{2}{*}{\tnote{3}  \  $\hat{\overline{Y}}_{PEML}$, SRSWOR} & $243.23$& $216.42$& $198.11$\\
&$(65.059)$&$(55.256)$&$(44.972)$\\
\multirow{2}{*}{$\hat{\overline{Y}}_{HT}$, RS}& $184.98$& $160.79$& $144.43$\\
&$(24.336)$&$(17.942)$&$(13.89)$\\
\multirow{2}{*}{$\hat{\overline{Y}}_{H}$, RS}& $189.49$& $163.19$& $145.82$\\
&$(314.18)$&$(209.6)$&$(164.32)$\\
\multirow{2}{*}{\tnote{3} \ $\hat{\overline{Y}}_{PEML}$, RS}& $343.6$& $300.14$& $272.63$\\
&$(60.804)$&$(20.411)$&$(21.998)$\\
\multirow{2}{*}{$\hat{\overline{Y}}_{RHC}$, RHC} &  $277.91$& $240.09$& $214.78$\\
&$(16.039)$&$(12.042)$&$(9.2784)$\\
\multirow{2}{*}{\tnote{3} \ $\hat{\overline{Y}}_{PEML}$, RHC}& $279.97$& $242.43$& $217.09$\\
&$(52.788)$&$(58.394)$&$(21.356)$\\
\hline
\end{tabular}
\end{threeparttable}
\end{table}

\begin{table}[h!]
\caption{Average and standard deviation of lengths of asymptotically $95\%$ CIs for variance of $y_2$. Recall from Table $4$ in Section $2$ that for variance of $y_2$, $h(y_2)$=$(y_2^2,y_2)$ and $g(s_1,s_2)$=$s_1-s_2^2$.}
\centering
\label{table 24}
\begin{tabular}{|c|c|c|c|} 
\hline
& \multicolumn{3}{c|}{Average length}\\
&\multicolumn{3}{c|}{(Standard deviation)}\\
 \hline
\backslashbox{Estimator and \\sampling  design \\based on which CI\\ is constructed}{Sample size}
&  $n$=$75$ & $n$=$100$ & $n$=$125$  \\ 
 \hline
\multirow{2}{*}{$g(\hat{\overline{h}}_{H})$, SRSWOR}& $1498664$& $1588740$& $2418155$\\
&$(3236118)$&$(2694726)$&$(3205532)$\\
\multirow{2}{*}{$g(\hat{\overline{h}}_{PEML})$, SRSWOR}& $1035032$& $1077345$& $1002397$\\
&$(1472036)$&$(1376947)$&$(1573834)$\\
\multirow{2}{*}{$g(\hat{\overline{h}}_{H})$, RS}& $887813.9$& $764055.6$& $684218.5$\\
&$(464853)$&$(377760)$&$(298552)$\\
\multirow{2}{*}{$g(\hat{\overline{h}}_{PEML})$, RS}& $1385778$& $1168689$& $1055339$\\
&$(1584677)$&$(1339377)$&$(1177054)$\\
\multirow{2}{*}{$g(\hat{\overline{h}}_{PEML})$, RHC}& $1319413$& $1134532$& $1072290$\\
&$(1473379)$&$(1384754)$&$(1472584)$\\
\hline
\end{tabular}
\end{table}

\begin{table}[h!]
\caption{Average and standard deviation of lengths of asymptotically $95\%$ CIs for correlation coefficient between $y_1$ and $y_3$. Recall from Table $4$ in Section $2$ that for correlation coefficient between $y_1$ and $y_3$, $h(y_1,y_3)$=$(y_1,y_3,y_1^2,y_3^2,y_1 y_3)$ and $g(s_1,s_2,s_3,s_4,s_5)$=$(s_5-s_1s_2)/((s_3-s_1^2)(s_4-s_2^2))^{1/2}$.}
\label{table 25}
\centering
\begin{tabular}{|c|c|c|c|} 
\hline
& \multicolumn{3}{c|}{Average length}\\
&\multicolumn{3}{c|}{(Standard deviation)}\\
 \hline
\backslashbox{Estimator and \\sampling  design \\based on which CI is constructed}{Sample size}
&  $n$=$75$ & $n$=$100$ & $n$=$125$  \\ 
 \hline
\multirow{2}{*}{$g(\hat{\overline{h}}_{H})$, SRSWOR}& $0.3682$& $0.3753$& $0.3893$\\
&$(0.1138)$&$(0.1039)$&$(0.0936)$\\
\multirow{2}{*}{$g(\hat{\overline{h}}_{PEML})$, SRSWOR}& $0.2747$& $0.2881$& $0.2884$ \\
&$(0.1095)$&$(0.1008)$&$(0.0879)$\\
\multirow{2}{*}{$g(\hat{\overline{h}}_{H})$, RS}& $0.3351$& $0.3453$& $0.3587$\\
&$(0.1652)$&$(0.0938)$&$(0.1034)$\\
\multirow{2}{*}{$g(\hat{\overline{h}}_{PEML})$, RS}& $592.48$& $260.44$& $469.36$\\
&$(0.2859)$&$(0.3441)$&$(2.738)$\\
\multirow{2}{*}{$g(\hat{\overline{h}}_{PEML})$, RHC}& $3838.4$& $2740.5$& $2238.3$\\
&$(1.2271)$&$(0.1467)$&$(0.1104)$\\
\hline
\end{tabular}
\end{table}

\begin{table}[h!]
\caption{Average and standard deviation of lengths of asymptotically $95\%$ CIs for regression coefficient of $y_1$ on $y_3$. Recall from Table $4$ in Section $2$ that for regression coefficient of $y_1$ on $y_3$, $h(y_1,y_3)$=$(y_1, y_3, y^2_3, y_1 y_3)$ and $g(s_1,s_2,s_3, s_4)$=$(s_4-s_1 s_2)/(s_3-s^2_2)$.}
\label{table 26}
\centering
\begin{tabular}{|c|c|c|c|} 
\hline
& \multicolumn{3}{c|}{Average length}\\
&\multicolumn{3}{c|}{(Standard deviation)}\\
 \hline
\backslashbox{Estimator and \\sampling  design \\based on which CI is constructed}{Sample size}
&  $n$=$75$ & $n$=$100$ & $n$=$125$  \\ 
 \hline
\multirow{2}{*}{$g(\hat{\overline{h}}_{H})$, SRSWOR}& $1.6443$& $1.781$& $1.8077$\\
&$(1.223)$&$(1.127)$&$(0.8849)$\\
\multirow{2}{*}{$g(\hat{\overline{h}}_{PEML})$, SRSWOR}& $1.3984$& $1.4239$& $1.491$\\
&$(0.8867)$&$(0.7898)$&$(0.6645)$\\
\multirow{2}{*}{$g(\hat{\overline{h}}_{H})$, RS}& $1.4072$& $1.5299$& $1.5449$\\
&$(0.6463)$&$(0.4833)$&$(0.4883)$\\
\multirow{2}{*}{$g(\hat{\overline{h}}_{PEML})$, RS}& $3240.4$& $4938.4$& $1705.3$\\
&$(4.3202)$&$(1.659)$&$(2.017)$\\
\multirow{2}{*}{$g(\hat{\overline{h}}_{PEML})$, RHC}& $50701.7$& $17291.2$& $22245.7$\\
&$(2.659)$&$(3.93)$&$(1.51)$\\
\hline
\end{tabular}
\end{table}

\begin{table}[h]
\caption{Average and standard deviation of lengths of asymptotically $95\%$ CIs for regression coefficient of $y_3$ on $y_1$. Recall from Table $4$ in Section $2$ that for regression coefficient of $y_3$ on $y_1$, $h(y_1,y_3)$=$(y_3, y_1, y^2_1, y_1 y_3)$ and $g(s_1,s_2,s_3, s_4)$=$(s_4-s_1 s_2)/(s_3-s^2_2)$.}
\label{table 27}
\centering
\begin{tabular}{|c|c|c|c|} 
\hline
& \multicolumn{3}{c|}{Average length}\\
&\multicolumn{3}{c|}{(Standard deviation)}\\
 \hline
\backslashbox{Estimator and \\sampling  design \\based on which CI is constructed}{Sample size}
&  $n$=$75$ & $n$=$100$ & $n$=$125$  \\ 
 \hline
\multirow{2}{*}{$g(\hat{\overline{h}}_{H})$, SRSWOR}& $0.1387$& $0.1449$& $0.1508$\\
&$(0.091)$&$(0.072)$&$(0.0616)$\\
\multirow{2}{*}{$g(\hat{\overline{h}}_{PEML})$, SRSWOR}& $0.1015$& $0.0994$& $0.1002$\\
&$(0.0868)$&$(0.0692)$&$(0.0593)$\\
\multirow{2}{*}{$g(\hat{\overline{h}}_{H})$, RS}& $0.1305$& $0.1379$& $0.1447$\\
&$(0.0919)$&$(0.0438)$&$(0.0357)$\\
\multirow{2}{*}{$g(\hat{\overline{h}}_{PEML})$, RS}& $113.4$& $263.23$& $78.782$\\
&$(0.1712)$&$(0.0725)$&$(0.0545)$\\
\multirow{2}{*}{$g(\hat{\overline{h}}_{PEML})$, RHC}& $798.95$& $490.91$& $286.92$\\
&$(0.6227)$&$(0.0862)$&$(0.1107)$\\
\hline
\end{tabular}
\end{table}

\begin{table}[h!]
\caption{Average and standard deviation of lengths of asymptotically $95\%$ CIs for correlation coefficient between $y_2$ and $y_4$. Recall from Table $4$ in Section $2$ that for correlation coefficient between $y_2$ and $y_4$, $h(y_2,y_4)$=$(y_2,y_4,y_2^2,y_4^2,y_2 y_4)$ and $g(s_1,s_2,s_3,s_4,s_5)$=$(s_5-s_1s_2)/((s_3-s_1^2)(s_4-s_2^2))^{1/2}$.}
\label{table 28}
\centering
\begin{tabular}{|c|c|c|c|} 
\hline
& \multicolumn{3}{c|}{Average length}\\
&\multicolumn{3}{c|}{(Standard deviation)}\\
 \hline
\backslashbox{Estimator and \\sampling  design \\based on which CI is constructed}{Sample size}
&  $n$=$75$ & $n$=$100$ & $n$=$125$  \\ 
 \hline
\multirow{2}{*}{$g(\hat{\overline{h}}_{H})$, SRSWOR}& $0.3428$& $0.359$& $0.3821$\\
&$(0.191)$&$(0.1783)$&$(0.1844)$\\
\multirow{2}{*}{$g(\hat{\overline{h}}_{PEML})$, SRSWOR}& $0.3088$& $0.3279$& $0.3537$\\
&$(0.1886)$&$(0.171)$&$(0.1773)$\\
\multirow{2}{*}{$g(\hat{\overline{h}}_{H})$, RS}& $0.2924$& $0.2926$& $0.298$\\
&$(0.1561)$&$(0.1491)$&$(0.1568)$\\
\multirow{2}{*}{$g(\hat{\overline{h}}_{PEML})$, RS}& $833.87$& $300.13$& $242.51$\\
&$(0.5226)$&$(0.4406)$&$(0.8658)$\\
\multirow{2}{*}{$g(\hat{\overline{h}}_{PEML})$, RHC}& $7593.1$& $3526.1$& $2390.9$\\
&$(0.4385)$&$(0.4869)$&$(0.2661)$\\
\hline
\end{tabular}
\end{table}

\begin{table}[h!]
\caption{Average and standard deviation of lengths of asymptotically $95\%$ CIs for regression coefficient of $y_2$ on $y_4$. Recall from Table $4$ in Section $2$ that for regression coefficient of $y_2$ on $y_4$, $h(y_2,y_4)$=$(y_2, y_4, y^2_4, y_2 y_4)$ and $g(s_1,s_2,s_3, s_4)$=$(s_4-s_1 s_2)/(s_3-s^2_2)$.}
\label{table 29}
\centering
\begin{tabular}{|c|c|c|c|} 
\hline
& \multicolumn{3}{c|}{Average length}\\
&\multicolumn{3}{c|}{(Standard deviation)}\\
 \hline
\backslashbox{Estimator and \\sampling  design \\based on which CI is constructed}{Sample size}
&  $n$=$75$ & $n$=$100$ & $n$=$125$  \\ 
 \hline
\multirow{2}{*}{$g(\hat{\overline{h}}_{H})$, SRSWOR}& $1.1188$& $1.1117$& $1.1566$\\
&$(1.251)$&$(1.061)$&$(1.171)$\\
\multirow{2}{*}{$g(\hat{\overline{h}}_{PEML})$, SRSWOR}& $0.9865$& $1.0005$& $1.0534$\\
&$(0.9935)$&$(0.8784)$&$(0.8758)$\\
\multirow{2}{*}{$g(\hat{\overline{h}}_{H})$, RS}& $0.8575$& $0.847$& $0.8427$\\
&$(0.6472)$&$(0.5219)$&$(0.4524)$\\
\multirow{2}{*}{$g(\hat{\overline{h}}_{PEML})$, RS}& $1583.8$& $1647.2$& $1533.9$\\
&$(1.733)$&$(1.822)$&$(1.302)$\\
\multirow{2}{*}{$g(\hat{\overline{h}}_{PEML})$, RHC}& $24127.4$& $10798.8$& $5076.1$\\
&$(2.05)$&$(1.468)$&$(2.385)$\\
\hline
\end{tabular}
\end{table}
\clearpage

\begin{table}[h!]
\caption{Average and standard deviation of lengths of asymptotically $95\%$ CIs for regression coefficient of $y_4$ on $y_2$. Recall from Table $4$ in Section $2$ that for regression coefficient of $y_4$ on $y_2$, $h(y_2,y_4)$=$(y_4, y_2, y^2_2, y_2 y_4)$ and $g(s_1,s_2,s_3, s_4)$=$(s_4-s_1 s_2)/(s_3-s^2_2)$.}
\label{table 30}
\centering
\begin{tabular}{|c|c|c|c|} 
\hline
& \multicolumn{3}{c|}{Average length}\\
&\multicolumn{3}{c|}{(Standard deviation)}\\
 \hline
\backslashbox{Estimator and \\sampling  design \\based on which CI is constructed}{Sample size}
&  $n$=$75$ & $n$=$100$ & $n$=$125$  \\ 
 \hline
\multirow{2}{*}{$g(\hat{\overline{h}}_{H})$, SRSWOR}& $0.1607$& $0.1727$& $0.1682$\\
&$(0.2236)$&$(0.2175)$&$(0.1744)$\\
\multirow{2}{*}{$g(\hat{\overline{h}}_{PEML})$, SRSWOR}& $0.1456$& $0.1586$& $0.1577$\\
&$(0.2018)$&$(0.1868)$&$(0.1616)$\\
\multirow{2}{*}{$g(\hat{\overline{h}}_{H})$, RS}& $0.1219$& $0.1232$& $0.1273$\\
&$(0.0798)$&$(0.0663)$&$(0.0615)$\\
\multirow{2}{*}{$g(\hat{\overline{h}}_{PEML})$, RS}& $236.81$& $108.3$& $85.466$\\
&$(0.3529)$&$(0.1879)$&$(0.3227)$\\
\multirow{2}{*}{$g(\hat{\overline{h}}_{PEML})$, RHC}& $1568.1$& $2215.1$& $659.3$\\
&$(0.4045)$&$(0.197)$&$(0.1416)$\\
\hline
\end{tabular}
\end{table}

\vskip .65cm
\noindent
Anurag Dey\\
{\it Indian Statistical Institute}
\vskip 2pt
\noindent
E-mail: deyanuragsaltlake64@gmail.com
\vskip 2pt

\noindent
Probal Chaudhuri\\
{\it Indian Statistical Institute}
\vskip 2pt
\noindent
E-mail:probalchaudhuri@gmail.com